\documentclass[12pt,reqno,a4paper]{amsart}            

\usepackage[totalwidth=460pt, totalheight=738pt]{geometry}

\usepackage{scrpage2}    
\usepackage{scrtime}     

  \lohead{\textsf{\shorttitle}}
  \rehead{\textsf{\authors}}
  \automark[subsection]{section}
  \automark[subsection]{section}

\usepackage{graphicx}                

\usepackage{amsmath, amssymb}
\usepackage{amsthm}

\usepackage{amscd}

\usepackage{accents}     
\usepackage{mathtools}  

\usepackage{bbm}         

\usepackage{mathrsfs}    

\usepackage{xspace}  

\numberwithin{equation}{section}

\newcounter{myenumi}

\newcommand{\itemref}[1]{\eqref{#1}}

\newcommand{\myfont}{\sffamily}

\newcommand{\myparagraph}[1]{\noindent\textbf{\myfont{#1}}}

\newtheoremstyle{mythmstyle}
  {\topsep}
  {\topsep}
  {\itshape}
  {}
  {\bfseries \myfont}
  {.}
  {.5em}
  {}

\newtheoremstyle{mydefstyle}
  {\topsep}
  {\topsep}
  {\normalfont}
  {}
  {\bfseries \myfont}
  {.}
  {.5em}
  {}

\swapnumbers                    

\theoremstyle{mythmstyle}       

\newtheorem{theorem}{Theorem}[section]

\newtheorem{proposition}[theorem]{Proposition}
\newtheorem{lemma}[theorem]{Lemma}
\newtheorem{corollary}[theorem]{Corollary}

\newcounter{intro}

\theoremstyle{mydefstyle}        
\newtheorem{definition}[theorem]{Definition}

\newtheorem{example}[theorem]{Example}

\newtheorem{remark}[theorem]{Remark}
\newtheorem{remarks}[theorem]{Remarks}
\newtheorem*{remark*}{Remark}

\expandafter\let\expandafter\oldproof\csname\string\proof\endcsname
\let\oldendproof\endproof
\renewenvironment{proof}[1][\bfseries\myfont\proofname]{%
  \oldproof[\bfseries \myfont #1]%
}{\oldendproof}

\makeatletter
\renewcommand\section{\@startsection{section}{1}%
  \z@{.7\linespacing\@plus\linespacing}{.5\linespacing}%
  {\Large\myfont\bfseries}}
\renewcommand\subsection{\@startsection{subsection}{2}%
  \z@{-.5\linespacing\@plus-.7\linespacing}{.5\linespacing}%
  {\large\myfont\bfseries}}
\renewcommand\subsubsection{\@startsection{subsubsection}{3}%
  \z@{.5\linespacing\@plus.7\linespacing}{-.5em}%
  {\myfont\bfseries}}
\renewenvironment{abstract}{%
  \ifx\maketitle\relax
    \ClassWarning{\@classname}{Abstract should precede
      \protect\maketitle\space in AMS document classes; reported}%
  \fi
  \global\setbox\abstractbox=\vtop \bgroup
    \normalfont\Small
    \list{}{\labelwidth\z@
      \leftmargin3pc \rightmargin\leftmargin
      \listparindent\normalparindent \itemindent\z@
      \parsep\z@ \@plus\p@
      
    }%
    \item[\hskip\labelsep
      \myfont\bfseries
    \abstractname.]%
}{%
  \endlist\egroup
  \ifx\@setabstract\relax \@setabstracta \fi
}
\renewcommand\contentsnamefont{\myfont\bfseries}
\renewcommand\@starttoc[2]{\begingroup
  \setTrue{#1}%
  \par\removelastskip\vskip\z@skip
  \@startsection{}\@M\z@{\linespacing\@plus\linespacing}%
    {.5\linespacing}{
      \contentsnamefont}{#2}%
  \ifx\contentsname#2%
  \else \addcontentsline{toc}{section}{#2}\fi
  \makeatletter
  \@input{\jobname.#1}%
  \if@filesw
    \@xp\newwrite\csname tf@#1\endcsname
    \immediate\@xp\openout\csname tf@#1\endcsname \jobname.#1\relax
  \fi
  \global\@nobreakfalse \endgroup
  \addvspace{32\p@\@plus14\p@}%
  \let\tableofcontents\relax
}
\renewcommand\@settitle{\begin{center}%
  \baselineskip14\p@\relax
    \LARGE
    \bfseries
    \myfont
  \@title
  \end{center}%
}
\renewcommand\@setauthors{%
  \begingroup
  \def\thanks{\protect\thanks@warning}%
  \trivlist
  \centering\footnotesize \@topsep30\p@\relax
  \advance\@topsep by -\baselineskip
  \item\relax
  \author@andify\authors
  \def\\{\protect\linebreak}%
  \large
  \myfont\bfseries\authors
  \ifx\@empty\contribs
  \else
    ,\penalty-3 \space \@setcontribs
    \@closetoccontribs
  \fi
  \endtrivlist
  \normalfont\myfont\@setaddresses
  \endgroup
}
\renewcommand\@setaddresses{\par
  \nobreak \begingroup
\footnotesize
  \def\author##1{\nobreak\addvspace\bigskipamount}%
  \def\\{\unskip, \ignorespaces}%
  \interlinepenalty\@M
  \def\address##1##2{\begingroup
    \par\addvspace\bigskipamount\indent
    \@ifnotempty{##1}{(\ignorespaces##1\unskip) }%
    {
      \ignorespaces##2}\par\endgroup}%
  \def\curraddr##1##2{\begingroup
    \@ifnotempty{##2}{\nobreak\indent\curraddrname
      \@ifnotempty{##1}{, \ignorespaces##1\unskip}\/:\space
      ##2\par}\endgroup}%
  \def\email##1##2{\begingroup
    \@ifnotempty{##2}{\nobreak\indent\emailaddrname
      \@ifnotempty{##1}{, \ignorespaces##1\unskip}\/:\space
      \ttfamily##2\par}\endgroup}%
  \def\urladdr##1##2{\begingroup
    \def~{\char`\~}%
    \@ifnotempty{##2}{\nobreak\indent\urladdrname
      \@ifnotempty{##1}{, \ignorespaces##1\unskip}\/:\space
      \ttfamily##2\par}\endgroup}%
  \addresses
  \endgroup
}
\renewcommand\enddoc@text{\ifx\@empty\@translators \else\@settranslators\fi
}
\renewcommand\@secnumfont{\myfont\bfseries} 

\renewcommand\maketitle{\par
  \@topnum\z@ 
  \@setcopyright
  \thispagestyle{firstpage}
  \ifx\@empty\shortauthors \let\shortauthors\shorttitle
  \else \andify\shortauthors
  \fi
  \@maketitle@hook
  \begingroup
  \@maketitle
  \toks@\@xp{\shortauthors}\@temptokena\@xp{\shorttitle}%
  \toks4{\def\\{ \ignorespaces}}
  \edef\@tempa{%
    \@nx\markboth{\the\toks4
      \@nx\MakeUppercase{\the\toks@}}{\the\@temptokena}}%
  \@tempa
  \endgroup
  \c@footnote\z@
  \@cleartopmattertags
}
\makeatother

\newcommand{\Sec}[1]{Section~\ref{sec:#1}}
\newcommand{\Secs}[2]{Sections~\ref{sec:#1} and~\ref{sec:#2}}
\newcommand{\App}[1]{Appendix~\ref{app:#1}}

\newcommand{\Subsec}[1]{Subsection~\ref{ssec:#1}}
\newcommand{\Subsecs}[2]{Subsections~\ref{ssec:#1} and~\ref{ssec:#2}}

\newcommand{\Fig}[1]{Figure~\ref{fig:#1}}

\newcommand{\Thm}[1]{Theorem~\ref{thm:#1}}
\newcommand{\Thms}[2]{Theorems~\ref{thm:#1} and~\ref{thm:#2}}

\newcommand{\Ex}[1]{Example~\ref{ex:#1}}

\newcommand{\Lem}[1]{Lemma~\ref{lem:#1}}
\newcommand{\Lems}[2]{Lemmas~\ref{lem:#1} and~\ref{lem:#2}}

\newcommand{\Cor}[1]{Corollary~\ref{cor:#1}}

\newcommand{\Prp}[1]{Proposition~\ref{prp:#1}}
\newcommand{\Prps}[2]{Propositions~\ref{prp:#1} and~\ref{prp:#2}}

\newcommand{\Prpenum}[2]{Proposition~\ref{prp:#1}~(\ref{#2})}

\newcommand{\Rem}[1]{Remark~\ref{rem:#1}}
\newcommand{\Rems}[2]{Remarks~\ref{rem:#1} and~\ref{rem:#2}}

\newcommand{\Remenum}[2]{Remark~\ref{rem:#1}~(\ref{#2})}

\newcommand{\Def}[1]{Definition~\ref{def:#1}}

\newcommand{\Defenum}[2]{Definition~\ref{def:#1}~(\ref{#2})}

\newcommand{\abs}[2][{}]{\lvert{#2}\rvert_{{#1}}}    
\newcommand{\abssqr}[2][{}]{\lvert{#2}\rvert^2_{#1}} 
\newcommand{\bigabs}[2][{}]{\bigl\lvert{#2}\bigr\rvert_{#1}}     
\newcommand{\bigabssqr}[2][{}]{\bigl\lvert{#2}\bigr\rvert^2_{#1}}
\newcommand{\Bigabs}[2][{}]{\Bigl\lvert{#2}\Bigr\rvert_{#1}}     
\newcommand{\Bigabssqr}[2][{}]{\Bigl\lvert{#2}\Bigr\rvert^2_{#1}}

\newcommand{\normsymb}{\|}
\newcommand{\bignormsymb}[1]{#1\|}

\newcommand{\norm}[2][{}]{\normsymb{#2}\normsymb_{{#1}}}    
\newcommand{\normsqr}[2][{}]{\normsymb{#2}\normsymb^2_{#1}} 

\newcommand{\bignormsqr}[2][{}]{\bignormsymb{\bigl}{#2}%
                                \bignormsymb{\bigr}^2_{#1}}

\newcommand{\iprod}[3][{}]{\langle{#2},{#3}\rangle_{#1}}  

\newcommand{\set}[2]{\{ \, #1 \, | \, #2 \, \} }      
\newcommand{\bigset}[2]{\bigl\{ \, #1 \, \bigl|\bigr. \, #2 \, \bigr\} }
\newcommand{\Bigset}[2]{\Bigl\{ \, #1 \, \Bigl|\Bigr. \, #2 \, \Bigr\} }

\DeclareMathOperator*{\bigdcup}{\mathaccent\cdot{\bigcup}}
\DeclareMathOperator*{\dcup}   {\mathaccent\cdot\cup}

\newcommand{\map}[3]{ #1 \colon #2 \longrightarrow #3}    

\newcommand{\bd}  {\partial}          
\newcommand{\intr}[1]{\ring{{#1}}}    

\newcommand{\restr}[1]{{\restriction}_{#1}} 

\def\Xint#1{\mathchoice
   {\XXint\displaystyle\textstyle{#1}}%
   {\XXint\textstyle\scriptstyle{#1}}%
   {\XXint\scriptstyle\scriptscriptstyle{#1}}%
   {\XXint\scriptscriptstyle\scriptscriptstyle{#1}}%
   \!\int}
\def\XXint#1#2#3{{\setbox0=\hbox{$#1{#2#3}{\int}$}
     \vcenter{\hbox{$#2#3$}}\kern-.5\wd0}}
\def\XXsum#1#2#3{{\setbox0=\hbox{$#1{#2#3}{\int}$}
     \vcenter{\hbox{$#2#3$}}\kern-.60\wd0}}

\newcommand{\dashint}{\Xint-}   
\newcommand{\avint}{{\textstyle\dashint}}   

\newcommand{\card}[1]{\lvert#1\rvert}   

\newcommand{\dd}    {\, \mathrm d}    

\DeclareMathOperator{\dom}    {dom}

\DeclareMathOperator{\id}     {id}   

\DeclareMathOperator{\supp}   {supp}
\DeclareMathOperator{\vol}    {vol}

\newcommand{\specsymb} {\sigma} 

\newcommand{\spec}[2][{}]   {\specsymb_{\mathrm{#1}}(#2)}

\newcommand{\eps}{\varepsilon} 
\newcommand{\Eps}{\mathrm E} 
\renewcommand{\phi}{\varphi}   
\renewcommand{\rho}{\varrho}   

\newcommand{\conj}[1]{\overline {#1}} 

\newcommand{\R}{\mathbb{R}} 
\newcommand{\C}{\mathbb{C}} 
\newcommand{\N}{\mathbb{N}} 

\newcommand{\1}{\mathbbm 1}                    

\newcommand{\e}{\mathrm e}  

\newcommand{\wt}{\widetilde}           

\newcommand{\leCS}{\stackrel{\mathrm{CS}}\le}
\newcommand{\leCY}{\stackrel{\mathrm{CY}}\le}

\newcommand{\HS}{\mathscr H}           

\newcommand{\Sobsymb} {\mathsf H} 

\newcommand{\Contsymb} {\mathsf C}     
\newcommand{\Lsymb}    {\mathsf L}     
\newcommand{\lsymb}    {\ell}          

\newcommand{\Sobspace}[1][1]{\Sobsymb^{#1}} 

\newcommand{\Contspace}[1][{}]{\Contsymb^{#1}}     
\newcommand{\Lpspace}[1][p]    {\Lsymb_{#1}}     
\newcommand{\lpspace}[1][p]    {\lsymb_{#1}}     
\newcommand{\Lsqrspace}    {\Lpspace[2]}     
\newcommand{\lsqrspace}    {\lpspace[2]}          

\newcommand{\Cont}[2][{}]{\Contspace[#1]({#2})}

\newcommand{\Lsqr}[2][{}]{\Lsqrspace^{#1}({#2})} 
\newcommand{\lsqr}[2][{}]{\lsqrspace^{#1}({#2})}   

\newcommand{\Sob}[2][1]{\Sobspace [#1]({#2})}         

\newcommand{\Err}{\mathrm O}

\newcommand{\quadtext}[1]{\quad\text{#1}\quad}
\newcommand{\qquadtext}[1]{\qquad\text{#1}\qquad}

\usepackage{hyperref}
\usepackage{nicefrac}
\usepackage{enumitem}
\setenumerate{leftmargin=2em,itemindent=0ex,labelsep=1ex}  
\newcommand{\energy}{\mathcal E}    

\newcommand{\pcf}{pcf\xspace}

\newcommand{\deltaA}{\delta_{\mathrm a}}
\newcommand{\deltaB}{\delta_{\mathrm b}}
\newcommand{\deltaC}{\delta_{\mathrm c}}
\newcommand{\deltaD}{\delta_{\mathrm d}}

\newcommand{\wtdeltaC}{\delta_{\mathrm c}}
\newcommand{\wtdeltaD}{\delta_{\mathrm d}}

\newcommand{\MetGr}{M}

\newcommand{\Mfd}{X}
\newcommand{\Xv}[1][v]{{X_{#1}}} 
\newcommand{\intrXv}[1][v]{{\intr X_{#1}}} 
\newcommand{\cXv}[1][v]{{\check X_{#1}}} 
\newcommand{\Xe}[1][e]{{X_{#1}}} %
\newcommand{\Xve}[1][v,e]{{X_{#1}}} %
\newcommand{\Mv}[1][v]{{\MetGr_{#1}}} 
\newcommand{\intrMv}[1][v]{{\intr \MetGr_{#1}}} 
\newcommand{\cMv}[1][v]{{\check \MetGr_{#1}}} 
\newcommand{\Me}[1][e]{{\MetGr_{#1}}} %

\newcommand{\muVar}{\overline \mu} 
\newcommand{\gammaVar}{\overline \gamma} 
\newcommand{\ellVar}{\overline \ell} 

\newcommand{\vxeps}{{v,\eps}}
\newcommand{\edeps}{{e,\eps}}

\DeclareMathOperator{\dist}{dist}

\begin{document}

\title
{Approximation
  of fractals by manifolds and other graph-like spaces}

\author{Olaf Post}
\address{Fachbereich 4 -- Mathematik,
  Universit\"at Trier,
  54286 Trier, Germany}
\email{olaf.post@uni-trier.de}

\author{Jan Simmer}%
\address{Fachbereich 4 -- Mathematik,
  Universit\"at Trier,
  54286 Trier, Germany}
\email{simmer@uni-trier.de}

\date{\today, \thistime,  \emph{File:} \texttt{\jobname.tex}} 

\begin{abstract}
  We define a distance between energy forms on a graph-like metric
  measure space and on a discrete weighted graph using the concept of
  quasi-unitary equivalence.  We apply this result to metric graphs and
  graph-like manifolds (e.g.\ a small neighbourhood of an embedded
  metric graph) as metric measure spaces with energy forms associated
  with canonical Laplacians, e.g., the Kirchhoff Laplacian on a metric
  graph resp.\ the (Neumann) Laplacian on a manifold (with boundary)
  and express the distance of the associated energy forms in terms of
  geometric quantities.

  We showed in~\cite{post-simmer:pre17a} 
  that the approximating sequence of energy forms on weighted graphs
  used in the definition of an energy form on a \pcf fractal converge
  in the sense that the distance in the quasi-unitary equivalence
  tends to $0$.  By transitivity of
  quasi-unitary equivalence, we conclude that we can approximate the
  energy form on a \pcf fractal by a sequence of energy forms on
  metric graphs and graph-like manifolds.  In particular, we show that
  there is a sequence of domains converging to a \pcf fractal such
  that the corresponding (Neumann) energy forms converge to the
  fractal energy form.
  
  Quasi-unitary equivalence of energy forms implies a norm estimate
  for the difference of the resolvents of the associated Laplace
  operators.  As a consequence, suitable functions of the Laplacians
  are close resp.\ converge as well in operator norm, e.g.\ the
  corresponding heat operators and spectral projections.  The same is
  true for the spectra and the eigenfunctions in all above examples.
\end{abstract}


\maketitle



%
\section{Introduction}
\label{sec:intro}
%
The aim of this article is to approximate 
energy forms on metric spaces by energy forms on discrete graphs.  An
energy form here is a closed non-negative and densely defined
quadratic form in the corresponding $\Lsqrspace$-space.  We achieve
the approximation by defining a sort of ``distance'' between two
energy forms acting on different Hilbert spaces with the notion of
quasi-unitary equivalence introduced in~\cite{post:06} (see
also~\cite{post:12}).  We consider in this article two main examples,
namely we compare discrete weighted graphs with metric graphs and
graph-like manifolds.  We then apply these abstract results to \pcf
fractals where we prove that a suitable \pcf fractal, and its
corresponding sequence of metric graphs and graph-like manifolds
together with their energy forms are close in the sense of
quasi-unitary equivalence.

\subsection{Main results}
Let $G=(V,E)$ be a discrete graph with vertex and edge weight
functions $\map \mu V {(0,\infty)}$ and $\map \gamma E {(0,\infty)}$
respectively.  We consider the weighted Hilbert space $\lsqr{V,\mu}$
with norm given by $\normsqr[\lsqr{V,\mu}] f:=\sum_{v \in V}
\abssqr{f(v)}\mu(v)<\infty$ and the canonical non-negative quadratic
form given by
\begin{equation*}
  \energy(f) := \sum_{e=\{v,v'\} \in E} \gamma_e \abssqr{f(v)-f(v')}.
\end{equation*}
We assume that this form is bounded (which is equivalent to the fact
that the relative weight defined in~\eqref{eq:rel.weight} is bounded).
Moreover, let $X$ be a metric space with Borel measure $\nu$, together
with a non-negative and closed quadratic form $\energy_X$.

We embed the graph $G$ into a metric measure space $X$ with measure
$\nu$ and energy form $\energy_X$ via a partition of unity $\map
{\psi_v} X {[0,1]}$ ($v \in V$) such that $X_v:=\supp \psi_v$ is
compact and connected and such that $\intr X_v \cap \intr X_{v'} \ne \emptyset$
iff there is an edge between $v$ and $v' \in V$.  The
\emph{identification operators} needed in order to express the notion
of \emph{quasi-unitary equivalence} (see \Def{quasi-uni}) are given by
\begin{align*}
  \map J 
  {&\lsqr{V,\mu}}
  {\Lsqr{X,\nu}}, 
  &
  Jf &:= c \sum_{v \in V} f(v) \psi_v \\
  \map {J'}   {&\Lsqr{X,\nu}}{\lsqr{V,\mu}},
  &
  (J'u)(v) 
  &:= \frac 1 c \cdot \frac 1{\nu(v)} \int_X u \psi_v \dd \nu,
\end{align*}
where $\nu(v):=\int_X \psi_v \dd \nu$ and where $c>0$ is the so-called
\emph{isometric rescaling factor}.  We assume that $\nu(v)/\mu(v)$ is
equal or close to the constant $c^{-2}$.  Typically, we choose
$\psi_v$ to be \emph{harmonic} (in a suitable sense, specified later
in the examples).  We also need to relate the energy forms and
therefore use also identification operators acting on the form
domains, see the proof of \Thm{q-u-e} for details.

Our first main result is the following:
\begin{theorem}[see \Thm{q-u-e}]
  \label{thm:main1}
  Let the discrete weighted graph $(G,\mu,\gamma)$ be uniformly
  embedded into the metric measure space $(X,\nu,\energy_X)$ (see
  \Def{graph.emb.in.met.space}). Then, $\energy$ and $\wt\energy:=\tau
  \energy_X$ are $\delta$-quasi-unitarily equivalent, where $\delta$
  can be expressed entirely in quantities of the weighted graph and
  the metric measure space.
\end{theorem}
Here, $\tau$ is a suitable energy rescaling factor.  It will give us
some freedom in choosing a suitable length or weight scaling in our
examples.  The precise definition of $\delta$ is given in \Thm{q-u-e},
and its meaning is explained in \Rems{q-u.b}{q-u-e}.

\subsubsection*{Quasi-unitary equivalence}
The notion of $\delta$-quasi-unitary equivalence (reviewed briefly in
\App{norm.convergence}) is a generalisation of two concepts (see also
\Rem{quasi-uni}): if $\delta=0$ then $0$-quasi-unitary equivalence is
just ordinary \emph{unitary equivalence} and if \emph{both Hilbert
  spaces are the same} and if $J=J'=\id$ then it implies that the
operator norm of the difference of resolvents is bounded by $\delta$.
One condition to check for $\delta$-quasi-unitary equivalence is that
\begin{gather*}
  (f-J'Jf)(v)
  =\frac 1{\nu(v)}\sum_{v' \sim v}
     \bigl(f(v)-f(v')\bigr)\iprod{\psi_v}{\psi_{v'}}
  \quad\text{and}\\
  u-JJ'u
  =\sum_{v \in V} \Bigl(u-\frac 1{\nu(v)} \int_X u \psi_v \dd \nu\Bigr) \psi_v
\end{gather*}
are small in suitable norms.  The first one can be estimated by
\begin{equation*}
  \normsqr[\lsqr{V,\mu}]{f-J'Jf} 
  \le \frac {2\mu_\infty}{\gamma_0} \energy(f),
  \quadtext{where}
  \mu_\infty:= \sup_{v \in V} \mu(v)
  \quadtext{and}
  \gamma_0:= \inf_{e \in E} \gamma_e,
\end{equation*}
and $2\mu_\infty/\gamma_0$ is one of the terms appearing in $\delta$.
Note that $\mu_\infty/\gamma_0$ is related to the maximal inverse
relative weight, i.e., if we want to approximate an unbounded form on
$X$ (as in Case~\ref{app.a} below), the minimum of the relative weight
should tend to $\infty$.  The expression for $u-JJ'u$ contains the
orthogonal projection onto the first (constant) eigenfunction of a
weighted eigenvalue problem on $\Xv$, and hence can be estimated by
the inverse of the second (weighted) eigenvalue
$\lambda_2(\Xv,\psi_v)$ times the energy form restricted to $\Xv$ (see
\Lem{2nd.ev} for details).  Since $\lambda_2(\Xv,\psi_v)$ is bounded
from below by a suitable isoperimetric or Cheeger constant, we need
this isoperimetric constant to be uniformly large; this means that the
sets $\Xv$ are ``well connected''.

From the abstract theory of quasi-unitary equivalence of energy forms
we deduce the following (for more consequences we also refer
to~\cite[Sec.~3]{khrabustovskyi-post:pre17}):
\begin{theorem}[{\cite[Ch.~4]{post:12}}]
  Assume that $\energy$ and $\wt \energy$ are $\delta$-quasi-unitarily
  equivalent, then the associated opertors $\Delta\ge 0$ and $\wt
  \Delta\ge 0$ fulfil the following:
  \begin{align*}
    \norm{\eta(\Delta) - J^*\eta(\wt \Delta)J} 
    &\le C_\eta \delta, &
    \norm{\eta(\wt \Delta) - J\eta(\Delta)J^*} 
    &\le C'_\eta \delta,\\
    \dist\Bigl(\frac1{1+\spec \Delta},\frac1{1+\spec {\wt \Delta}}\Bigr)
    &\le \psi(\delta),&
    \bigabs{\lambda_k(\Delta)-\lambda_k(\wt \Delta)}
    &\le C_k \delta,
  \end{align*}
  where $\eta$ is a suitable function continuous in a neighbourhood of
  $\spec \Delta$, e.g., $\eta_z(\lambda)=(\lambda-z)^{-1}$ (resolvent
  in $z$), $\eta_t(\lambda)=\e^{-t\lambda}$ (heat operator) or
  $\phi=\1_I$ with $\bd I \cap \spec \Delta=\emptyset$ (spectral
  projection).  Moreover, $C_\eta$ and $C'_\eta$ depend only on
  $\eta$, $\psi(\delta) \to 0$ as $\delta \to 0$ and $\dist$ denotes
  the Hausdorff distance.  The last statement refers to the $k$-th
  eigenvalue of $\Delta$ resp.\ $\wt \Delta$ (counted with respect to
  multiplicity), and $C_k$ depends only on upper bounds of
  $\lambda_k(\Delta)$ and $\lambda_k(\wt \Delta)$.
\end{theorem}
We also have convergence of corresponding eigenfunctions in energy
norm, see \cite[Prop.~2.5]{post-simmer:pre17a} for details.

\subsubsection*{Applications}
We have the following applications in mind:
\begin{enumerate}[label=\myfont{\textbf{\Alph*.}},ref=\Alph*]
\item
  \label{app.a}
  \myparagraph{Given a \pcf fractal, choose a sequence of graphs:} Let
  $K$ be a \pcf fractal, and let $G=G_m$ be a sequence of weighted
  graphs approximating the fractal and its energy form; this example
  has been treated in~\cite{post-simmer:pre17a}; here, $c=1$ and
  $\tau=1$.

\item
  \label{app.b}
  \myparagraph{Given a discrete graph construct a metric graph:} given
  a weighted discrete graph $G$ we construct a metric graph $M$ with
  standard (also called Kirchhoff) Laplacian and with edge lengths
  reciprocally proportional with the discrete edge weights; the
  corresponding energy forms then are $\delta$-quasi-unitarily
  equivalent, see \Thm{mg.q-u-e}; here, $\delta$ is of the same order
  as the \emph{maximal inverse relative weight} defined by the
  discrete weighted graph (see~\eqref{eq:rel.weight} and
  \Remenum{unif.weighted.graph}{unif.weighted.graph.a''}).

\item
  \label{app.b'}
  \myparagraph{Given a metric graph, construct a sequence of discrete
    subdivision graphs:} given a metric graph $M$, then there is a
  sequence of discrete weighted subdivision graphs $SG_m$ such that
  the corresponding metric and discrete energy forms are
  $\delta_m$-quasi-unitarily equivalent, where $\delta_m$ is of order
  of the mesh width of the (metric) subdivision graph $SM_m$, see
  \Cor{subdiv.graph}; note that the energy forms on $M$ and $SM_m$ are
  unitarily equivalent as vertices of degree $2$ have no effect.

\item
  \label{app.c}
  \myparagraph{Given a discrete graph construct a graph-like
    manifold:} given a weighted discrete graph $G$ we construct a
  graph-like manifold $X$ with longitudinal length scale (edge
  lengths) again reciprocally proportional with the discrete edge
  weights; then the corresponding energy forms are
  $\delta$-quasi-unitarily equivalent; here, $\delta$ is of similar
  type as in the metric graph case, but has an additional parameter:
  the ratio of the transversal and longitudinal length scale, see
  \Thm{mfd.q-u-e} and \Cor{mfd.q-u-e}.

\item
  \label{app.d}
  \myparagraph{Given a sequence of discrete graphs, construct a
    sequence of metric graphs and graph-like manifolds:} We apply
  Cases~\ref{app.b} and~\ref{app.c} to the sequence $(G_m)_m$ of
  weighted discrete graphs approximating a given \pcf fractal $K$ from
  Case~\ref{app.a}, and hence obtain a sequence of metric graphs
  resp.\ graph-like manifolds $M_m$ resp.\ $X_m$ with energy forms
  being $\delta_m$-quasi-unitarily equivalent with the one on $G_m$,
  where $\delta_m \to 0$ exponentially fast, see
  \Thms{mg.frac.q-u-e}{mfd.frac.q-u-e}.
\end{enumerate}

For more information on fractals we refer
to~\cite{kigami:01,strichartz:06}, see also~\Subsec{fractals}
and~\cite[Sec.~3]{post-simmer:pre17a}.  For more information on metric
graphs and graph-like manifolds, we refer to the corresponding
\Subsecs{met.graphs}{mfds} and the references cited there.

\subsubsection*{Approximation of fractals by metric graphs}
Let us here be more precise on Case~\ref{app.d}: Consider a \pcf
fractal $K$ with energy form $\energy_K$ which can be approximated by
a sequence of discrete graphs $(G_m)_{m\in\N_0}$ and energy forms
$\energy_{G_m}$.  In~\cite{post-simmer:pre17a} we proved that
$\energy_K$ and $\energy_{G_m}$ are $\delta_m$-quasi-unitary
equivalent with $\delta_m = \Err((r/N)^{m/2})$ where $N \ge 2$ is the
number of fixed points of the IFS and $r \in (0,1)$ the energy
renormalisation parameter of the self-similarity (see
\Subsec{fractals} for details).  Now, the idea is to show that the
energy forms on the discrete graph $G_m$ and the corresponding metric
graph $\MetGr_m$ are $\delta'_m$-quasi-unitarily equivalent with
$\delta_m'=\Err((r/N)^{m/2})$ and then use the transitivity of
quasi-unitary equivalence, see \Prps{trans.q-u-e}{trans.q-u-e.op}.

Let $G_m=(V_m,E_m)$ be one of the approximating discrete graphs with
vertex and edge weights $\map {\mu_m}{V_m}{(0,\infty)}$ and $\map
{\gamma_m}{E_m}{(0,\infty)}$.  The length function $\ell_m\colon
E_m\to (0,\infty)$ of the metric graph $\MetGr_m$ is chosen to be
proportional to $1/\gamma_{m,e}$, in particular, $\ell_{m,e}$ decays
exponentially in $m$.  On a metric graph, we identify an edge $e \in
E$ with the interval $[0,\ell_{m,e}]$ with boundary points identified
according to the graph structure.  Moreover the canonical measure on
$\MetGr_m$ is given by the sum of the Lebesgue measures on the
intervals.  Hence, the associated Hilbert space is $\Lsqr{\MetGr_m,
  \nu_m}$ with norm given by
\begin{equation*}
  \normsqr[\Lsqr{\MetGr_m,\nu_m}] u
  = \sum_{e \in E} \int_0^{\ell_e} \abssqr{u_e(x)} \dd x
\end{equation*}
where we consider $u$ as a family $(u_e)_{e \in E_m}$ with $\map
{u_e}{[0,\ell_e]} \C$.  A canonical energy form on $\MetGr_m$ is
\begin{equation*}
  \energy_{\MetGr_m}(u)=\normsqr[\Lsqr{\MetGr_m,\nu_m}] {u'}
  = \sum_{e \in E} \int_{a_e}^{b_e} \abssqr{u'_e(x)} \dd x,
  \quad
  \dom(\energy_{\MetGr_m})=\Sob {\MetGr_m},
\end{equation*}
where $\Sob{\MetGr_m}$ consists of functions $u_e \in
\Sob{[0,\ell_e]}$ such that $u$ is continuous on $\MetGr_m$.  The
associated operator is the usual standard (also called Kirchhoff)
Laplacian.  The functions in the partition of unity $(\psi_{m,v})_{v
  \in V_m}$ used for the identification operators $J$ and $J'$ fulfil
$\psi_{m,v}(v)=1$ and $\psi_{m,v}(v')=0$ for all $v'\in
V_m\setminus\{v\}$.  Moreover, we assume that the functions
$\psi_{m,v}$ are harmonic functions, i.e., affine linear on the edges.
An application of \Thm{main1} is as follows (for a version where we
quantify the error term $\wt \delta_m$, see \Cor{mg.frac.q-u-e'}):

\begin{theorem}
  \label{thm:main2}[\Cor{mg.frac.q-u-e}]
  The fractal energy form $\energy_K$ and the rescaled approximating
  metric graph energy forms $\wt \energy_m = \tau_m
  \energy_{\MetGr_m}$ are $\wt \delta_m$-quasi-unitarily equivalent,
  where $\wt \delta_m \to 0$ as $m \to \infty$.
\end{theorem}

A natural choice is to let the metric graphs $\MetGr_m$ be generated
by the IFS, i.e., we start with $\MetGr_0$ (embedded as $K$ in
$\R^d$), and define $\MetGr_m$ as the $m$-th iterate under the IFS
with similitude factor $\theta \in (0,1)$ (see \Subsec{fractals} for
details).  In this case, the lengths at generation $m$ are of order
$\theta^m$, and the energy rescaling factor is
$\tau_m=\Err((N\theta^2/r)^m)$.  For the Sierpi\'nski triangle, we
have $N=3$, $r=3/5$ and $\theta=1/2$, hence $\tau_m=\Err((5/4)^m)$,
confirming analytically the energy rescaling factor for the
Sierpi\'nski triangle used in the numerical calculations of the
eigenvalues in~\cite[Sec.~3]{bhs:09}.

\subsubsection*{Approximation of fractals by manifolds}
Similarly, we can construct a sequence $(X_m)_m$ of graph-like
manifolds (see \Fig{graph-mfd} and for more details \Subsec{mfds})
with transversal length scale parameter $\eps_m$ (the ``thickness'' of
an edge).  One example of a graph-like manifold $X_m$ is the
$\eps_m$-neighbourhood of $\MetGr_m$ (if embedded in $\R^d$).  In this
case, the corresponding operator associated with the energy form given
by
\begin{equation*}
  \energy_{X_m}(u)=\int_{X_m} \abssqr{\nabla u(x)} \dd x,
  \qquad 
  u \in \dom \energy_{X_m} = \Sob{X_m},
\end{equation*}
is the Neumann Laplacian.  We can also construct boundaryless
manifolds such as the surface of a tubular neighbourhood of a metric
graph $\MetGr_m$ embedded in $\R^3$.
\begin{figure}[h]
  \label{fig:graph-mfd}
  \setlength{\unitlength}{1mm}
  \begin{picture}(120,80)
    \includegraphics[width=0.8\textwidth]{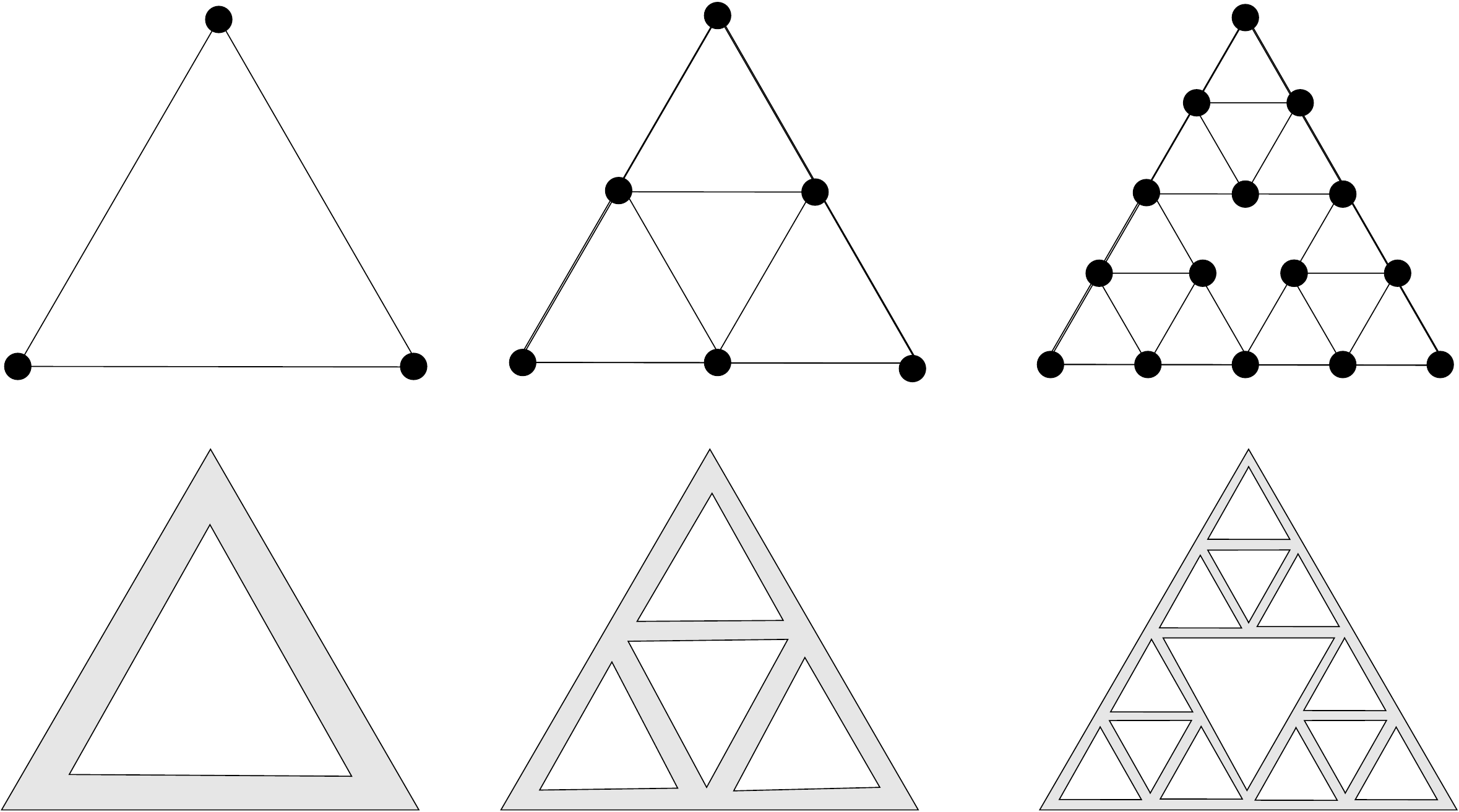}  
    \put(-132,43){$G_0$}
    \put(-88,43){$G_1$}
    \put(-42,43){$G_2$}
    \put(-132,8){$X_0$}
    \put(-87,8){$X_1$}
    \put(-40,8){$X_2$}
  \end{picture}
  \caption{The beginning of the sequence of graphs $(G_m)_m$
    approximating the Sierpi\'nski triangle and a corresponding
    sequence of graph-like manifolds $(X_m)_m$ close to the graph $G_m$
    for the generations $m=0,1,2$.}
\end{figure}
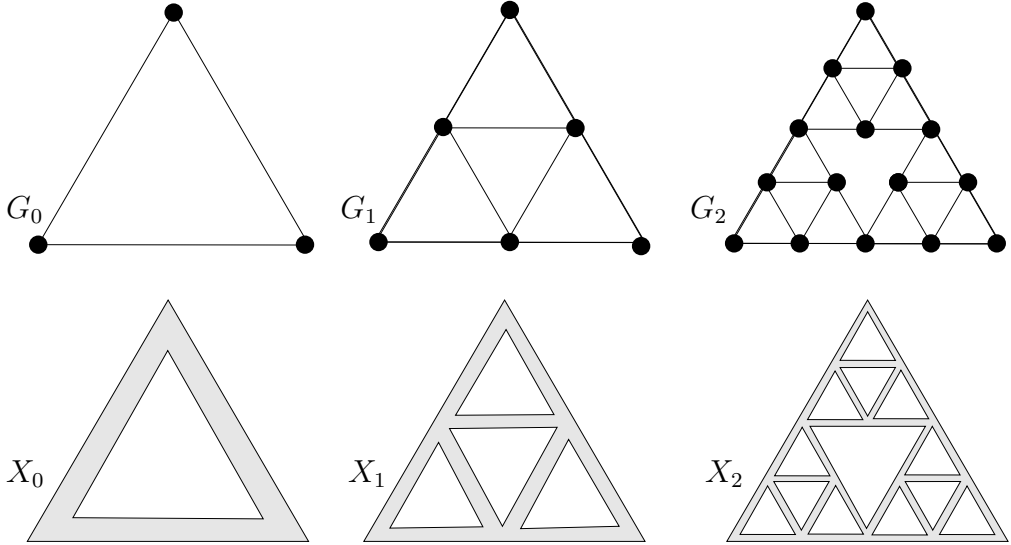
Again, we need an energy rescaling factor $\hat \tau_m$.
\begin{theorem}[\Cor{mfd.frac.q-u-e}]
  \label{thm:main3}
  Assume that the transversal length scale $\eps_m=\eps_0 \Eps^m$
  ($\Eps \in (0,1)$) decays faster than the longitudinal length scale
  $\ell_{m,e}=\Err(\theta^m)$ (i.e., $\Eps<\theta$), then the fractal
  energy form $\energy_K$ and the approximating graph-like manifold
  energy forms $\hat \energy_m := \hat \tau_m \energy_{X_m}$ are $\hat
  \delta_m$-quasi-unitarily equivalent, where $\hat \delta_m \to 0$ as
  $m \to \infty$.
\end{theorem}
For a version where we quantify the error term $\wt \delta_m$, see
\Cor{mfd.frac.q-u-e'}.  In particular, we can approximate the energy
form on a fractal by (rescaled) energy forms on a family of smooth
manifolds.  Unfortunately, we cannot treat the case $\Eps=\theta$,
i.e., when the longitudinal and transversal length scale are of the
same.  This case occurs e.g.\ when starting with a suitable
neighbourhood $\Mfd_0$ of the first metric graph $\MetGr_0$, and then
apply the IFS to generate a sequence of graph-like manifolds
$(\Mfd_m)_m$, see \Rem{no-mfd-ifs} for details.

\subsection{Previous and related works}
\label{ssec:prev.works}

Variational convergence (such as $\Gamma$- or Mosco convergence) of
discrete energy forms to suitable energy forms on metric measure
spaces is an often treated topic; we mention here only two of them and
refer to the references therein: Kasue~\cite{kasue:10} considers
sequences of compact metric spaces with resistance metric and energy
forms, e.g.\ finite resistance networks (i.e., weighted graphs with
trivial vertex weights $\mu=1$ and variable edge weights $\gamma$) and
$\Gamma$-convergence of such sequences, e.g., to infinite graphs.
Hinz and Teplyaev~\cite[Thm.~1.2]{hinz-teplyaev:16} consider
approximations of a bounded Dirichlet form by a sequence of finite
weighted graphs in the sense of Mosco.  The finite weighted graphs
appear as Dirichlet forms on a finitely generated measure space.  The
partition of unity is just the corresponding finite family of
indicator functions; as the limit Dirichlet form is bounded, these
indicator functions are in its domain.  Hinz and Teplyaev then use the
fact that any Dirichlet form can be approximated by a bounded
Dirichlet.  Note that the Mosco convergence is equivalent with some
notion of \emph{strong} resolvent convergence for varying Hilbert
spaces, see~\cite[Thm.~2.4]{kuwae-shioya:03}, hence our results are
stronger as they provide convergence in \emph{operator norm};
nevertheless we believe that the conditions of quasi-unitary
equivalence are often easier to check than the one for Mosco
convergence.

Discretisations of metric measure spaces have also been used in order
to check certain types of Poincar\'e inequalities on the metric
measure space by a suitable version on a graph, e.g.\
in~\cite{gill-lopez:15} (see also the references therein).  As in our
work, Gill and Lopez embed the graph into the metric measure space via
a partition of unity, On the other hand, Cheeger and
Kerner~\cite{cheeger-kleiner:15} use the opposite approach and
construct metric measure spaces fulfilling a Poincar\'e inequality
from an increasing sequence of discrete graphs (called metric measure
graphs there).

There are quite some articles about the relation between (mostly
infinite or classes of finite) graphs and (non-compact or classes of
compact) Riemannian manifolds under the name \emph{discretisation of a
  manifold}; most authors are interested in questions whether certain
properties are invariant under so-called \emph{rough isometries},
using a related property on a discrete graph.  We mention here only
the works
of~\cite{dodziuk-patodi:76,kanai:86a,kanai:86b,coulhon:92,mantuano:05,
  cgr:pre16} and references therein; and the
monograph~\cite{chavel:01} where Chavel defines similar maps as our
$J$ and $J'$, called \emph{smoothing} and \emph{discretisation} there.
The interest in all these works is to have uniform control of classes
of manifolds and discrete graphs, e.g.\ that $f-J'Jf$ is bounded by a
constant times the energy norm of $f$, but the constant is not
supposed to be small as in our case.  We will deal with the
\emph{approximation} of manifolds and their energy forms by discrete
graphs in a forthcoming publication.

Kigami considers in~\cite{kigami:03} energy forms (called resistance
forms there) and shows that they can be approximated by limits of
finite weighted nested graphs; the corresponding energy forms converge
monotonously.  In~\cite[Sec.~5--6]{teplyaev:08} Teplyaev considers
energy forms on sets with finitely ramified cell structure and uses an
approximation by metric graphs e.g.\ to characterise the operator
domain of the original energy form.
In the recent preprint~\cite{hinz-meinert:pre17} Hinz and Meinert use
the metric graph approximation of a Sierpi\'nki triangle to
approximate some non-linear differential equations on a fractal.

For the approximation of fractals by open subsets or manifolds not so
much is known: Berry, Heilman and Strichartz in~\cite{bhs:09} provide
numerical results on the eigenvalues of some \pcf fractals $K$
approximated by Neumann Laplacians on open set $X_m \supset K$
constructed according to the IFS.  We confirm analytically that their
energy rescaling factor for the Sierpi\'nski triangle and an
approximating sequence of open subsets in $\R^2$ is $\tau_m=5/4$ (see
Case~1 in \Ex{mfd.sierpinski}).  Better numerical results are obtained
by Blasiak, Strichartz and U\u gurcan~\cite[Sec.~3]{bsu:08} by a
sequence of open sets $X_m \supset K$ \emph{not} constructed from the
IFS.  We follow in our manifold example a similar strategy as our
approximating sets $X_m$ (if $K \subset \R^2$) are neither generated
by the IFS; moreover, they are neither subsets nor supsets of $K$.
Some analytic work is done in~\cite{mosco-vivaldi:15} (see also the
references therein): Mosco and Vivaldi construct a sequence of
\emph{weighted} energy forms on open domains that Mosco-converge to an
energy form on a nested fractal such as the Koch curve or the
Sierpi\'nski triangle.  The weights are chosen in such a way that the
passage along a vertex becomes very narrow.  We encounter a smilar
problem, see \Rem{no-mfd-ifs}.

\subsection{Structure of the article}
In \Sec{met.conv} we define when a discrete graph is uniformly
embedded into a metric measure space and prove \Thm{main1}.  In
\Secs{met.graphs}{mfds} we apply these results to metric graphs and
graph-like manifold.  Finally, in \Sec{fractals.met} we apply the
results from~\cite{post-simmer:pre17a} and the transitivity of
quasi-unitary equivalence (\Prps{trans.q-u-e}{trans.q-u-e.op}) and
deduce that a large class of \pcf fractals and its energy form can be
approximated by a suitable sequence of metric graphs resp.\ graph-like
manifolds together with their (renormalised) energy forms.
\App{norm.convergence} contains a brief introduction to the concept of
quasi-unitary equivalence.  In \App{2nd.ev}, we collect some estimates
on graph-like manifolds.

\subsection{Acknowledgements}
The authors would like to thank Uta Freiberg and Michael Hinz for
valuable comments and discussions.

%
\section{
  Convergence of energy forms on discrete graphs and metric spaces}
\label{sec:met.conv}
%

In this section, we provide a rather general setting: we measure how
``close'' a discrete graph is to a metric space, both with a suitable
energy form.  The ``distance'' is measured in terms of a parameter
$\delta \ge 0$ appearing in the definition of quasi-unitary
equivalence, see \Def{quasi-uni}.

We say that $\energy$ is an \emph{energy form} on a measure space
$(X,\nu)$, if $\energy$ is a non-negative, densely defined and closed
quadratic form, i.e., if $\energy(f) \ge 0$ for all $f \in \dom
\energy$, $\dom \energy$ is dense in $\Lsqr{X,\nu}$ and $\dom \energy$
with norm given by
\begin{equation*}
  \normsqr[\energy] f 
  := \normsqr[\Lsqr{X,\nu}] f + \energy(f)
\end{equation*}
is complete (i.e., a Hilbert space).  

\subsection{Metric spaces, energy forms and embedded graphs}
\label{ssec:the.spaces}

\subsubsection*{The graph and its energy form}
Let $G=(V,E,\bd)$ be a discrete graph, i.e., $V$ and $E$ are finite or
countable sets and $\map \bd E {V \times V}$, $\bd e =
(\bd_-e,\bd_+e)$, associates to each edge $e \in E$ its \emph{initial}
and \emph{terminal vertex} $\bd_-e$ and $\bd_+e$, respectively.  We
assume here that $G$ is \emph{simple}, i.e., that there are no loops
(i.e., edges with $\bd_-e=\bd_+e$) nor multiple edges (i.e., edges
$e_1,e_2$ with $\bd_+e_1=\bd_\pm e_2$ and $\bd_-e_1=\bd_\mp e_2$).  We
write $v \sim v'$ for two vertices if there is an edge $e$ with
$v=\bd_\pm e$ and $v'=\bd_\mp e$.  We also set
\begin{equation*}
  E_v^\pm := \set{e \in E}{\bd_\pm e=v} \qquadtext{and}
  E_v := E_v^- \dcup E_v^+.
\end{equation*}

In order to have a Hilbert space structure and an energy form, we need
two \emph{weight functions} $\map \mu V {(0,\infty)}$ and $\map \gamma
E {(0,\infty)}$, the \emph{vertex} and \emph{edge weights}. 

The following quantitative control using the variation of the weights
(and the maximal degree) will be useful:
\begin{definition}
  \label{def:unif.weighted.graph}
  Let $(G,\mu,\gamma)$ be a weighted graph.
  \begin{enumerate}
  \item 
    \label{unif.weighted.graph.a}
    The \emph{relative weight} (also called \emph{$\mu$-degree}) is
    the vertex weight defined by
    \begin{equation}
      \label{eq:rel.weight}
      \rho(v)
      :=\frac{\sum_{e \in E_v} \gamma_e}{\mu(v)}.
    \end{equation}
    
  \item 
    \label{unif.weighted.graph.a'}
    We call
    \begin{equation*}
      \frac {\mu_\infty}{\gamma_0}
      := \frac{\sup_{v \in V}\mu(v)} {\inf_{e \in E} \gamma_e}
    \end{equation*}
    the \emph{maximal inverse relative weight} of the weighted graph.
  \item 
    \label{unif.weighted.graph.b}
    Let $d_\infty \in \N$, $\muVar>0$ and $\gammaVar >0$.  The
    weighted graph is called
    \emph{$(d_\infty,\muVar,\gammaVar)$-uniform} or simply
    \emph{uniform} if the \emph{degree} $\deg v:= \card {E_v}$ is
    uniformly bounded, i.e., $\deg v \le d_\infty$, and if
    \begin{subequations}
      \label{eq:unif.weighted.graph}
      \begin{align}
        \label{eq:unif.weighted.grapha}
        \mu_0:=\inf_{v \in V}\mu(v)
        &\le \mu_\infty :=\sup_{v \in V}\mu(v) 
        \le \muVar \mu_0 \quad\text{and}\\
        \gamma_0:=\inf_{e \in E}\gamma_e 
        &\le \gamma_\infty :=\sup_{e \in E}\gamma_e
        \le \gammaVar \gamma_0.
      \end{align}
    \end{subequations}
  \end{enumerate}
\end{definition}

\begin{remark}
  \label{rem:unif.weighted.graph}
  \indent
  \begin{enumerate}
  \item
  \label{unif.weighted.graph.a''}
    The relative weight is bounded by
    \begin{equation}
      \label{eq:rel.weight.lower}
      \frac{\gamma_0}{\mu_\infty} 
      \le \rho(v) 
      \le \frac{d_\infty \gamma_\infty}{\mu_0} 
    \end{equation}
    if the corresponding numbers are in $(0,\infty)$.  In particular,
    $1/\rho(v)$ is bounded from above by $\mu_\infty/\gamma_0$, hence
    the name \emph{maximal inverse relative weight}.

    Moreover, the relative weight is \emph{bounded} for a
    $(d_\infty,\muVar,\gammaVar)$-uniform weighted graph, as the
    latter implies $0<\mu_0\le \mu_\infty<\infty$ and $0<\gamma_0 \le
    \gamma_\infty<\infty$.
  \item For a $(d_\infty,\muVar,\gammaVar)$-uniform
    weighted graph, we also have
    \begin{equation}
      \label{eq:rel.weight.bdd}
      \frac1{\gammaVar d_\infty \muVar} \cdot \frac {\mu_\infty}{\gamma_0}
      \le \frac 1{\rho(v)}
      \le \frac {\mu_\infty}{\gamma_0},
    \end{equation}
    i.e., the inverse of the relative weight is \emph{controlled} by
    the maximal inverse relative weight $\mu_\infty/\gamma_0$.  The latter
    quotient will be used later on as error, e.g.\ in \Prp{q-u.b}.
  \end{enumerate}
\end{remark}

The Hilbert space associated with a weighted graph is
\begin{equation*}
  \HS := \lsqr {V,\mu}, \qquad 
  \normsqr[\lsqr{V,\mu}] f := \sum_{v \in V} \abssqr{f(v)} \mu(v)
\end{equation*}
and the energy form is given by
\begin{equation}
  \label{eq:graph.energy}
  \energy(f) := \sum_{e \in E} \gamma_e \abssqr{(df)_e},
  \qquad (df)_e := f(\bd_+e)-f(\bd_-e).
\end{equation}
and bounded by $2\rho_\infty := 2\sup_v \rho(v)$.  Throughout this
article, we assume that the relative weight $\rho$ is bounded, i.e.,
$\rho_\infty<\infty$.  In particular, the corresponding discrete
Laplacian $\Delta_{(G,\mu,\gamma)}$ is also bounded by $2\rho_\infty$,
and acts as
\begin{equation}
  \label{eq:disc.lapl}
  (\Delta_{(G,\mu,\gamma)} \phi)(v)
  = \frac 1 {\mu(v)}\sum_{e \in E_v} \gamma_e \bigl(\phi(v)-\phi(v_e)\bigr),
\end{equation}
where $v_e$ is the vertex on $e$ opposite to $v$.

\subsubsection*{The metric measure space and energy form}
As space approximated by the graph $G$ we choose a metric measure
space $X$ with Borel measure $\nu$.  Then we have a
canonical Hilbert space, namely
\begin{equation*}
  \wt \HS := \Lsqr{X,\nu}, \qquad
  \normsqr[\Lsqr{X,\nu}] u := \int_X \abssqr u \dd \nu.
\end{equation*}
We assume that $X$ has a natural energy form\footnote{It is indeed in
  all our examples a Dirichlet form but we will not need this fact in
  our analysis.}  $\energy_X$ with domain $\wt \HS^1:=\dom \energy_X$.
We specify further properties later on.  We are mainly interested in
three examples together with their natural energy forms:
\begin{enumerate}
\item $X$ is a \pcf self-similar fractal (see
  e.g.~\cite[Sec.~3]{post-simmer:pre17a});
\item $X$ is a metric graph;
\item $X$ is a graph-like manifold.
\end{enumerate}
We will explain metric graphs and graph-like manifolds in
\Subsecs{met.graphs}{mfds}.  As we need to rescale the energy in some
cases, we introduce a constant $\tau>0$, called \emph{energy
  rescaling factor} and set
\begin{equation*}
  \wt \energy := \tau \energy_X
\end{equation*}
with the same domain.  We specify $\tau$ later on.

We come now to our main definition, relating a discrete graph with a
metric measure space and its energy form:
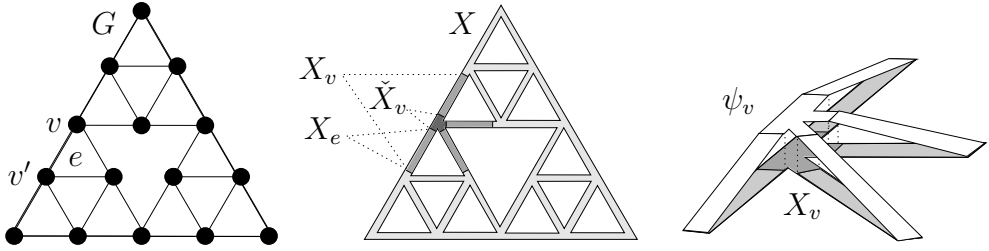
\begin{figure}[h]
  \label{fig:graph-emb-met}
  \centering
  \setlength{\unitlength}{1mm}
  \begin{picture}(120,40)
    \includegraphics[width=0.8\textwidth]{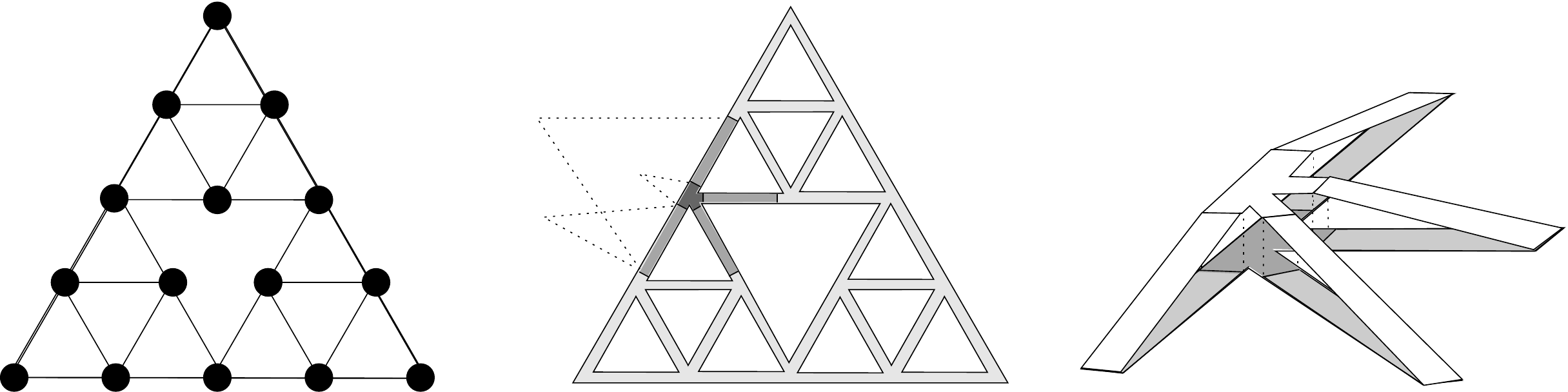}  
    \put(-118,28){$G$}
    \put(-124,15){$v$}
    \put(-129,8){$v'$}
    \put(-121,10.5){$e$}
    \put(-71,28){$X$}
    \put(-81.5,18){$\cXv$}
    \put(-90.7,22){$\Xv$}
    \put(-90,14){$\Xe$}
    \put(-35,18){$\psi_v$}
    \put(-27,4){$\Xv$}
  \end{picture}
  \caption{A graph uniformly embedded into a metric measure space,
    here a graph-like manifold.  Left: the discrete graph.  Middle: a
    metric space with vertex neighbourhood $\Xv$ (grey and dark grey) and core
    vertex neighbourhood $\cXv$ (dark grey).  Right: A function
    $\psi_v$ of the partition of unity on $\Xv$.}
\end{figure}
\begin{definition}
  \label{def:graph.emb.in.met.space}
  Let $X$ be a metric measure space with Borel measure $\nu$ and
  energy form $\energy_X$.  A weighted graph $(G,\mu,\gamma)$ is
  \emph{uniformly embedded} into $(X,\nu,\energy_X)$, if the following
  conditions hold:
  \begin{subequations}
    \begin{enumerate}
    \item
      \label{graph.emb.in.met.space.a}
      \myparagraph{Partition of unity and relation to graph
        structure:} There is a family of functions $\Psi:=(\psi_v)_{v
        \in V}$ with $\psi_v \in \dom \energy_X$ and $\sum_{v \in V}
      \psi_v = \1_X$ such that
      \begin{equation}
        \label{eq:graph.emb.part.1}
        \text{$\Xv:= \supp \psi_v$ is connected}
        \quadtext{and}
        \intrXv \cap \intrXv[v'] \ne \emptyset
        \;\iff\; v \sim v'.
      \end{equation}
    \item
      \label{graph.emb.in.met.space.b}
      \myparagraph{Decomposition of energy:} There are energy forms
      $\energy_{\Xv}$ on $(\Xv,\nu(\cdot \restr \Xv))$ ($v \in V$)
      with domains $\dom \energy_{\Xv} = \bigset{u \restr \Xv} {u \in
        \dom \energy_X}$ such that
      \begin{equation}
        \label{eq:energy.local}
        \energy_X(u)
        \le \sum_{v \in V} \energy_{\Xv}(u \restr \Xv)
        \le 2\energy_X(u)
      \end{equation}
      for all $u \in \dom \energy_X$.
    \item
      \label{graph.emb.in.met.space.c}
      \myparagraph{Local energy is uniformly spectrally small:} The form
      $\energy_{\Xv}$ is closable in the \emph{weighted Hilbert space}
      \begin{equation}
        \label{eq:weighted.hs}
        \Lsqr{\Xv,\psi_v}
        := \Bigset{u} 
        {\normsqr[\Lsqr{\Xv,\psi_v}] u 
          :=\int_X \abssqr u \psi_v \dd \nu < \infty},
      \end{equation}
      and the spectrum is purely discrete.  Moreover, the first
      eigenvalue is $\lambda_1(\Xv,\psi_v)=0$ with constant
      eigenfunction $\1_{\Xv}$, and the second eigenvalue (as family
      of $v \in V$) fulfils
      \begin{equation}
        \label{eq:2nd.ev.bdd}
        \tag{\ref{eq:weighted.hs}'}
        0<\lambda_2 
        := \inf_{v \in V} \lambda_2(\Xv,\psi_v).
      \end{equation}
    \item
      \label{graph.emb.in.met.space.d}
      \myparagraph{(Almost) compatibility of the weights:} The
      weighted graph has finite maximal inverse relative weight
      $\mu_\infty/\gamma_0 \in (0,\infty)$, where
      \begin{equation}
        \label{eq:gamma0}
        \mu_\infty := \sup_{v \in V} \mu(v)<\infty
        \qquadtext{and}
        \gamma_0 := \inf_{e \in E}\gamma_e >0.
      \end{equation}
      Moreover,
      \begin{equation}
        \label{eq:def.nu.decomp}
        \tag{\ref{eq:gamma0}'}
        \nu(v):= \int_{\Xv} \psi_v \dd \nu
        = \nu_0(v) + \check \nu(v)
        \quadtext{and}
        \frac 1 {c^2}
        := \frac{\nu_0(v)}{\mu(v)}
      \end{equation}
      for all $v \in V$, where the so-called \emph{isometric rescaling
        factor} $c>0$ is assumed to be \emph{independent} of $v \in
      V$, and where
      \begin{equation}
        \label{eq:def.alpha}
        \tag{\ref{eq:gamma0}''}
        \alpha_\infty := \sup_{v \in V} \abs{\alpha(v)} \le \frac 12
        \qquadtext{with}
        \alpha(v):=\frac{\check \nu(v)}{\nu_0(v)}.
      \end{equation}
    \end{enumerate}
  \end{subequations}
  We call $\Xv$ the \emph{(enlarged) vertex neighbourhood} of $v \in
  V$ in $X$.  Moreover, 
  \begin{equation}
    \label{eq:cxv.xe}
      \cXv := \psi_v^{-1}\{1\} = \set{x \in X_v}{\psi_v(x)=1}
      \quadtext{and}
    \Xe := \Xv[\bd_-e ] \cap \Xv[\bd_+e]
  \end{equation}
  are called the \emph{core vertex neighbourhood} of $v$ and the
  \emph{edge neighbourhood} of $e$ in $X$, respectively.  We call
  $\nu=(\nu(v))_{v \in V}$ the \emph{vertex weight} of $\Psi$ in
  $(X,\nu)$.  We say that $\nu$ and $\mu$ are \emph{compatible
    weights} if $\nu(v)/\mu(v)$ is independent of $v$, i.e., if
  $\nu=\nu_0$ or $\alpha=0$.
\end{definition}

\begin{remark}
  \label{rem:graph.emb.in.met.space}
  \indent
  \begin{enumerate}
  \item Note that we have the relation
    \begin{equation}
      \label{eq:trick'}
      \nu(v) 
      = \sum_{v' \in V} \int_X \psi_v \psi_{v'} \dd \nu
      = \sum_{v' \in V} \iprod {\psi_v}{\psi_{v'}}
      = \normsqr[\Lsqr{X,\nu}] {\psi_v} 
      + \sum_{v' \in V, v' \sim v} \iprod {\psi_v}{\psi_{v'}}
    \end{equation}
    as the partition of unity reflects the discrete graph structure.

  \item For ease of notation, we write $\energy_{\Xv}(u)$ instead of
    the more precise notation $\energy_{\Xv}(u \restr \Xv)$.
    Condition~\eqref{eq:energy.local} allows us to get finer estimates
    in the proof of \Prp{q-u.b}.

  \item Note that $\Lsqr \Xv \subset \Lsqr{\Xv,\psi_v}$
    (see~\eqref{eq:hs.weighted}), hence $\energy_{\Xv}$ need not to be
    closed in $\Lsqr{\Xv,\psi_v}$.  Moreover, the condition
    $\lambda_2>0$ in particular implies that $\lambda_2(\Xv,\psi_v)>0$
    for all $v \in V$, i.e., $\lambda_1(\Xv,\psi_v)=0$ is a
    \emph{simple} eigenvalue.  In our examples, this means that $\Xv$
    is connected (either as graph or as topological space) as we
    already assumed in~\itemref{graph.emb.in.met.space.a}.
  \item 
    \label{energy.spectral.weight}
    Condition~\eqref{eq:def.alpha} is a condition already for
    \emph{finite} graphs, as it assumes that the ``bad'' part of
    $\nu$, i.e., the deviation $\check \nu(v)$ from being compatible,
    is small.  The reason for assuming that $\alpha_\infty \le 1/2$
    becomes clear at the end of the proof of \Prp{q-u.b}.

  \item For a single finite graph $G$, the conditions on $\lambda_2$
    and $\mu_\infty/\gamma_0$ are always fulfilled.  We stress the
    constants $\lambda_2, c, \mu_\infty/\gamma_0 \in (0,\infty)$ and
    $\alpha_\infty \in [0,1/2]$, if we consider an infinite graph
    or an infinite family of graphs.  In both situations, we need a
    uniform control of certain parameters, see \Prp{q-u.b} and
    \Thm{q-u-e}.
  \end{enumerate}
\end{remark}
We denote the \emph{weighted average} of $u$ by
\begin{equation}
  \label{eq:def.dashint}
  \dashint_X u \dd \nu_v 
  := \frac1{\nu(v)}  \int_X u \dd \nu_v, 
  \qquadtext{where}
  \dd \nu_v := \psi_v \dd \nu.
\end{equation}
Note that
\begin{equation*}
  \dashint_X u \dd \nu_v
  = \frac1{\int_X \psi_v \dd \nu}  \int_X u \psi_v \dd \nu
  =  \frac 1{\nu(v)} \iprod[\Lsqr{\Xv,\psi_v}] u {\1_{\Xv}},
\end{equation*}
the latter representation will be useful in the next lemma:
\begin{lemma}
  \label{lem:2nd.ev}
  Assume that $\energy_{\Xv}$
  fulfils~\Defenum{graph.emb.in.met.space}{graph.emb.in.met.space.c}
  then
  \begin{equation}
    \label{eq:2nd.ev}
    \int_X \Bigabssqr{u - \avint_X u \dd \nu_v}
    \dd \nu_v
    \le \frac 1{\lambda_2(\Xv,\psi_v)} \energy_{\Xv}(u),
  \end{equation}
  for all $u \in \dom \energy_{\Xv}$.
\end{lemma}
\begin{proof}
  Note first that $\normsqr[\Lsqr{\Xv,\psi_v}] {\1_{\Xv}} = \int_X
  \psi_v \dd \nu = \nu(v)$, hence
  \begin{equation*}
    u- \frac 1 {\nu(v)} \iprod u {\1_{\Xv}} \1_{\Xv}
    = u - \avint_X u \dd \nu_v
  \end{equation*}
  is the projection onto the orthogonal complement of the first
  eigenspace $\C \1_{\Xv}$.  The result then follows from the min-max
  characterisation of eigenvalues.
\end{proof}

In some cases, it will be easier to have an estimate of the
corresponding \emph{unweighted} eigenvalue problem.  As above,
$\lambda_2(\Xv,\psi_v)$ denotes the second (first non-zero) eigenvalue
of the operator associated with $\energy_{\Xv}$ in the weighted
Hilbert space $\Lsqr{\Xv,\psi_v}$.
\begin{lemma}
  \label{lem:2nd.ev.alt}
  Let $\Phi_{2,v}$ be a normalised eigenfunction associated with
  $\lambda_2(\Xv,\psi_v)$.  Moreover, denote by $\lambda_2(\Xv)$ the
  second (first non-zero) eigenvalue of $\energy_{\Xv}$ in the
  \emph{unweighted} Hilbert space $\Lsqr \Xv$.  If $\Phi_{2,v} \in
  \Lsqr \Xv$, then we have
  \begin{equation}
    \label{eq:2nd.ev.lower}
    \lambda_2(\Xv,\psi_v)
    \ge \lambda_2(\Xv).
  \end{equation}
\end{lemma}
\begin{proof}
  Note first that
  \begin{equation}
    \label{eq:hs.weighted}
    \normsqr[\Lsqr{\Xv,\psi_v}] u
    = \int_{\Xv} \abssqr u \psi_v \dd \nu
    \le \int_{\Xv} \abssqr u \dd \nu
    = \normsqr[\Lsqr \Xv] u
  \end{equation}
  as $0 \le \psi_v \le 1$.  In particular, $\Lsqr \Xv \subset
  \Lsqr{\Xv,\psi_v}$, but the inclusion is in general strict.  By the
  min-max characterisation of the second eigenvalue, we have
  \begin{align*}
    \lambda_2(\Xv)
    = \inf_{D_2}
      \sup_{u \in D_2 \setminus \{0\}}
      \frac{\energy_{\Xv}(u)}{\normsqr[\Lsqr \Xv] u}
    \le \sup_{u \in D_2 \setminus \{0\}}
      \frac{\energy_{\Xv}(u)}{\normsqr[\Lsqr \Xv] u}
    \le \sup_{u \in D_2 \setminus \{0\}}
      \frac{\energy_{\Xv}(u)}{\normsqr[\Lsqr {\Xv,\psi_v}] u}
  \end{align*}
  where $D_2$ runs through all two-dimensional subspaces of
  $\Lsqr{X_v}\cap \dom \energy_{\Xv}$.  As the first eigenfunction for
  both problems is the constant $\1_{\Xv}$ and as the second
  eigenfunction $\Phi_{2,v}$ of the weighted problem is also in the
  unweighted Hilbert space, we can choose $D_2=\C \1_{\Xv}+\C
  \Phi_{2,v}$ as two-dimensional space.  For this choice, the latter
  Rayleigh quotient becomes $\lambda_2(\Xv,\psi_v)$, and the result
  follows.
\end{proof}
\begin{remark*}
  In our applications on metric graphs and graph-like manifolds, we
  choose $\psi_v$ to be harmonic, and due to the geometric
  assumptions, $\psi_v$ will be (piecewise) affine linear.  For this
  choice, the second eigenfunction of the weighted problem is an Airy
  function hence continuous, and therefore belongs also to the
  unweighted Hilbert space.
\end{remark*}

\subsection{The identification operators}
\label{ssec:id.op}
Let us now consider a \emph{weighted} discrete graph $(G,\mu,\gamma)$
with vertex weight $\map \mu V {(0,\infty)}$ and edge weight $\map
\gamma E {(0,\infty)}$.  We have then the associated Hilbert space
$\lsqr{V,\mu}$ as in \Subsec{the.spaces}.  We assume that
$(G,\mu,\gamma)$ is \emph{uniformly embedded} into
$(X,\nu,\energy_X)$, see \Def{graph.emb.in.met.space}, hence the
isoperimetric rescaling factor $c>0$ is defined.

We now define the identification operators $J$ and $J'$ from
$\HS=\lsqr{V,\mu}$ into $\wt \HS=\Lsqr{X,\nu}$ and vice versa: Let
\begin{equation*}
  Jf := c \sum_{v \in V} f(v) \psi_v
\end{equation*}
for $\map f V \C$ with finite support and
\begin{equation*}
  (J'u)(v) 
  = \frac 1 c \cdot \dashint_X u \dd \nu_v.
\end{equation*}
(for the notation $\avint_X$ and $\dd \nu_v$
see~\eqref{eq:def.dashint}).

\begin{remark}
  \label{rem:discretisation}
  If $X$ is a Riemannian manifold, then $J$ is also called
  \emph{smoothing} of functions on $G$ and $J'$ is called
  \emph{discretisation} of functions on $X$, see the
  monograph~\cite[Sec.~VI.5]{chavel:01} for details (and also
  \Subsec{prev.works}).  Moreover, the concept of having a partition
  of unity with respect to a suitable cover of the manifold labelled
  by the graph vertices is called \emph{discretisation} of $X$,
  see~\cite[Sec.~V.3.2]{chavel:01} for details.
\end{remark}

  The energy form spaces are
$\HS^1=\HS=\lsqr{V,\mu}$ (if the relative weight is bounded) and $\wt
\HS^1=\dom \energy_X$ with norms given by
\begin{equation*}
  \normsqr[\HS^1] f 
  := \normsqr[\lsqr{V,\mu}] f + \energy(f)
  \quadtext{and}
  \normsqr[\wt \HS^1] u
  := \normsqr[\Lsqr{X,\nu}] u + \wt \energy(u)
\end{equation*}
respectively, where $\wt \energy(u)=\tau \energy_X$ and where the energy
rescaling parameter $\tau>0$ is specified later on.  Note that
although the spaces $\HS$ and $\HS^1$ are the same, the norms differ,
and this fact matters when considering \emph{families} of graphs where
$\normsqr[\HS^1] f /\normsqr[\HS] f \le 1+\sup_v  \rho(v)$ is \emph{not
  uniformly} bounded.

The following result is of abstract nature:
\begin{proposition}
  \label{prp:q-u.b}
  Assume that $(G,\mu,\gamma)$ is uniformly embedded into
  $(X,\nu,\energy_X)$, then the following holds:
  \begin{enumerate}
  \item
    \label{q-u.b.a}
    $J$ and $J'$ are bounded;
  \item
    \label{q-u.b.b}
    $J$ and $J'$ fulfil~\eqref{eq:quasi-uni.b} with $\delta$ replaced
    by $\deltaB$, where
    \begin{equation*}
      \deltaB^2
      = \max \Bigl\{\frac {2 \mu_\infty}{\gamma_0},
      \frac 2{\tau \lambda_2} \Bigr\}.
    \end{equation*}
  \item
    \label{q-u.b.c}
    $J$ and $J'$ fulfil~\eqref{eq:quasi-uni.a} with $\delta$ replaced
    by 
    \begin{equation}
      \label{eq:def.alpha0}
      \deltaA
      = 2\alpha_\infty,
    \end{equation}
    where $\alpha_\infty$ is defined in~\eqref{eq:def.alpha}.
  \end{enumerate}
  In particular, $J$ is $\delta$-quasi-unitary with adjoint $J'=J^*$
  and $\delta=\max\{\deltaA,\deltaB\}$.
\end{proposition}
\begin{proof}
  \itemref{q-u.b.a}~The boundedness of $\map J {\HS=\lsqr{V,\mu}} {\wt
    \HS=\Lsqr{X,\nu}}$ follows from
  \begin{subequations}
    \begin{align}      \nonumber 
      \normsqr[\Lsqr{X,\nu}] {Jf} 
      &= \sum_{v \in V} \sum_{v' \in V} c^2f(v)\conj{f(v')}
           \underbrace{\iprod{\psi_v}{\psi_{v'}}}_{\ge 0} 
      \leCY c^2 \sum_{v \in V} \abssqr{f(v)} 
                \sum_{v' \in V} \iprod{\psi_v}{\psi_{v'}}
      \\\label{eq:norm.est.j} 
      &= c^2 \sum_{v \in V} \abssqr{f(v)} \nu(v)
      \le \sup_{v \in V} \underbrace{\frac{c^2\nu(v)}{\mu(v)}}_{=1+\alpha(v)} 
           \normsqr[\lsqr{V,\mu}] f
      \le (1+\alpha_\infty) \normsqr[\lsqr{V,\mu}] f,
    \end{align}
    where we used the partition of unity property~\eqref{eq:trick'}
    for the second equality.
    The boundedness of $J'$ can be seen from
    \begin{align}
      \nonumber \normsqr[\lsqr{V,\mu}] {J'u} &= \frac 1 {c^2} \sum_{v
        \in V} \frac{\mu(v)}{\nu(v)^2}
      \Bigabssqr{\int_X u \dd \nu_v}
      \label{eq:norm.est.j'}
      \leCS \frac 1 {c^2} \sum_{v \in V} \frac{\mu(v)}{\nu(v)^2}
      \int_X \abssqr u \dd \nu_v \cdot \nu(v)\\
      &\le  \sup_{v \in V} \frac1{1+\alpha(v)}
         \int_X \abssqr u \dd \nu
      \le \frac 1{1-\alpha_\infty} \normsqr[\Lsqr {X,\nu}] u
    \end{align}
  \end{subequations}
  using the partition of unity property $\sum_v \dd \nu_v=\dd \nu$ in
  the second last inequality.

  \itemref{q-u.b.b}~We are now checking the conditions
  of~\eqref{eq:quasi-uni.b}: We have
  \begin{equation*}
    f(v)-(J'Jf)(v) 
    = \frac 1{\nu(v)} \sum_{v' \sim v} \bigl(f(v)-f(v')\bigr)
        \iprod{\psi_{v'}}{\psi_v}
  \end{equation*}
  using~\eqref{eq:trick'}, hence
  \begin{align*}
    \normsqr[\HS]{f-J'Jf}
    &= \sum_{v \in V} \frac {\mu(v)}{\nu(v)^2}
     \Bigabssqr{\sum_{v' \sim v} (f(v)-f(v')) \iprod {\psi_{v'}}{\psi_v}}\\
    &\leCS \sum_{v \in V} \frac {\mu(v)}{\nu(v)^2}
        \sum_{e \in E_v} \gamma_e^{-1} \iprod {\psi_{v_e}}{\psi_v}^2
        \sum_{e \in E_v} \gamma_e \abssqr{f(v)-f(v_e)}.
  \end{align*}
  Here, $v_e$ denotes the vertex opposite to $v$ on $e$.

  We now continue with the second sum of the last formula and estimate
  \begin{align*}
    \sum_{e \in E_v} \gamma_e^{-1} \iprod {\psi_{v_e}}{\psi_v}^2
    &\le \frac 1 {\gamma_0} 
         \sum_{e \in E_v} 
           \Bigl(\int_X \psi_{v_e} \psi_v \dd \nu\Bigr)^2\\
    &\le \frac 1 {\gamma_0} 
         \sum_{e \in E_v} 
         \Bigl(\int_X \psi_{v_e}\psi_v \dd \nu\Bigr)
         \Bigl(\int_X \psi_v \dd \nu \Bigr)
    = \frac 1 {\gamma_0} 
         \Bigl(\int_X \psi_v \dd \nu\Bigr)^2
    = \frac {\nu(v)^2} {\gamma_0}
  \end{align*}
  using~\eqref{eq:gamma0} for the first inequality, $\psi_{v_e} \le
  1$ for the second inequality and the partition of unity property
  for the first equality in the last line.  From this estimate we
  conclude
  \begin{align*}
    \normsqr[\HS]{f-J'Jf}
    &\le \frac{\mu_\infty}{\gamma_0} \sum_{v \in V}
          \sum_{e \in E_v} \gamma_e \abssqr{f(v)-f(v_e)}
    =\frac{2\mu_\infty}{\gamma_0}\sum_{e \in E}
         \gamma_e \abssqr{(df)_e}
    =\frac{2\mu_\infty}{\gamma_0}\energy(f)
  \end{align*}
  (the factor $2$ appears because $\sum_{v \in V}\sum_{e \in E_v}
  a_e=\sum_{e \in E} \sum_{v=\bd_\pm e} a_e=2\sum_{e \in E} a_e$).

  For the second condition in~\eqref{eq:quasi-uni.b}, we note that
  \begin{equation*}
    JJ'u 
    = \sum_{v \in V} \dashint_X u \dd \nu_v  \psi_v
  \end{equation*}
  Moreover, using the partition of unity property we have $u = \sum_{v
    \in V} u \psi_v$ hence
  \begin{align*}
    \normsqr[\wt \HS]{u - JJ'u}
    &=\int_X\Bigabssqr{
             \sum_{v \in V} 
                \Bigl(
                   u - \dashint_X u \dd \nu_v
                \Bigr)
               \psi_v} \dd \nu\\
    &\leCS \int_X \sum_{v \in V} 
             \Bigabssqr{u - \dashint_X u \dd \nu_v}
               \psi_v
               \sum_{v \in V} \psi_v \dd \nu\\
    &=\sum_{v \in V} \int_X 
             \Bigabssqr{u - \dashint_X u \dd \nu_v}
               \psi_v \dd \nu\\
    & \le \frac 1 {\lambda_2}
        \sum _{v \in V} \wt \energy_v(u)
      \le \frac 2{\lambda_2}  \energy_X(u)
      = \frac 2 {\tau \lambda_2} \wt \energy(u)
  \end{align*}
  using~\eqref{eq:2nd.ev}, \eqref{eq:2nd.ev.bdd},
  \eqref{eq:energy.local} and the energy rescaling factor $\tau$ for
  the last line.

  \itemref{q-u.b.c}~For the second condition
  in~\eqref{eq:quasi-uni.a}, we first define the function $\Xi$ by
  \begin{equation*}
    \Xi(\xi)
    = \Bigabs{\sqrt \xi - \frac 1 {\sqrt \xi}}
    = 2\sinh \Bigabs{\frac 12 \log \xi};
  \end{equation*}
  note that we have $\Xi(1)=0$, $\Xi(1/\xi)=\Xi(\xi)$ and $0
  <\Xi(\xi)\le \xi-1$ for $\xi>1$.  Then we have
  \begin{align*}
    \bigabs{\iprod{Jf} u-\iprod f {J'u}}
    &= \Bigabs{\sum_{v \in V} \Bigl(c - \frac{\mu(v)}{c\nu(v)}\Bigr)
        f(v) \iprod {\psi_v} u}\\
    &\le \sup_{v \in V} \Xi \Bigl(\frac {c^2\nu(v)}{\mu(v)} \Bigr)
        \Bigabs{\sum_{v \in V} \sqrt{\mu(v)} f(v) 
          \frac 1{\sqrt{\nu(v)}} \iprod {\psi_v} u
           }  \\
    &\leCS \sup_{v \in V} \Xi \bigl(1+\alpha(v)\bigr)
     \Bigl(
       \sum_{v \in V} \mu(v) \abssqr{f(v)}
       \sum_{v \in V} \frac 1{\nu(v)} \int_X \abssqr u \psi_v \dd \nu
       \int_X \psi_v \dd \nu
     \Bigr)^{1/2}\\
     &\le \max \{\Xi(1+\alpha_\infty),\Xi(1+\inf_{v \in V} \alpha(v))\}
      \norm[\lsqr{V,\mu}] f \norm[\Lsqr{X,\nu}] u.
  \end{align*}
  The first term in the maximum appears when $\alpha(v) \ge 0$, the
  second when $\alpha(v)<0$.  The latter one can further be estimated
  by
  \begin{equation*}
    \Xi(1-\alpha_\infty)
    =\Xi(1/(1-\alpha_\infty))
    \le 1/(1-\alpha_\infty)-1
    =\alpha_\infty/(1-\alpha_\infty)
    \le 2\alpha_\infty
  \end{equation*}
  provided $\alpha_\infty \le 1/2$.
  From~\eqref{eq:norm.est.j}--\eqref{eq:norm.est.j'} and the last
  estimate we see that
  \begin{equation*}
    \deltaA 
    = \max
    \Bigl\{
      \sqrt{1+\alpha_\infty}-1,
      \frac 1{\sqrt{1-\alpha_\infty}}-1,
      \Xi(1+\alpha_\infty),
      2\alpha_\infty
    \Bigr\}
  \end{equation*}
  where the last term wins, hence we chose $\deltaA=2\alpha_\infty$,
  see also \Remenum{graph.emb.in.met.space}{energy.spectral.weight}.
\end{proof}

\begin{remark}
  \label{rem:q-u.b}
  Let us comment on the error terms in $\deltaB$ and $\deltaA$:
  \begin{enumerate}
  \item
    \label{rm.q-u.b.a}
    The maximal inverse relative weight $\mu_\infty/\gamma_0$ is an
    upper bound on the inverse of the relative weight $\rho(v)=\sum_{e
      \in E_v} \gamma_e/\mu(v)\ge \gamma_0/\mu_\infty$, hence a
    necessary condition for the maximal inverse relative weight
    $\mu_\infty/\gamma_0$ to be small is that the relative weight is
    large.  For a uniform weighted graph, this condition is also
    sufficient (see \Rem{unif.weighted.graph}).

    Note that if we plug in $f=\delta_v$ into $f-J'Jf$, we see that
    \begin{equation*}
      \frac{\normsqr{f-J'Jf}}{\energy(f)}
      \ge \frac 1 {\rho(v)}.
    \end{equation*}
    In particular, if the weighted graph is uniform, then
    $1/(\gammaVar d_\infty\muVar)\cdot \mu_\infty/\gamma_0$ is a lower
    bound, hence the error estimate is \emph{optimal}.
  \item
    \label{rm.q-u.b.b}
    The quotient $1/\lambda_2$ in the second error term means that the
    cells $\Xv$ are ``spectrally small'', i.e., they are small and
    sufficiently connected (we hence expect that the energy rescaling
    factor $\tau$ is not small).  Note that a lower bound on the second
    eigenvalue is given by a Cheeger-like isoperimetric constant; and
    a large Cheeger constant means a ``well-connected'' space $\Xv$.
    In particular, if this constant is uniformly bounded and large, we
    will get a small error $1/\lambda_2$.

    This error is also optimal as one can see by plugging in the
    eigenfunction associated with $\lambda_2(\Xv,\psi_v)$ into the
    estimate.
  \item
    \label{rm.q-u.b.c}
    The error term $\alpha_\infty$ measures how far the weights
    $(\nu(v))_{v \in V}$ given by $\nu(v)=\int_X \psi_v \dd \nu$ are
    away from being a constant multiple of the vertex weights $\mu$ on
    the graph.
  \end{enumerate}
\end{remark}

\subsection{Quasi-unitary equivalence of energy forms}
\label{ssec:q-u-e.met}

Let us now show under some additional assumptions how to obtain the
quasi-unitary equivalence of the energy forms on the metric space and
the discrete graph:
\begin{theorem}
  \label{thm:q-u-e}
  Assume that $(G,\mu,\gamma)$ is uniformly embedded into
  $(X,\nu,\energy_X)$ (see \Def{graph.emb.in.met.space}).  Moreover,
  we assume that
  \begin{subequations}
    \begin{enumerate}
    \item
      \label{graph.emb.en.c}
      for each $v \in V$ there exists $\map{\Gamma_v}{\dom
        \energy_{\Xv}}\C$ and $\wtdeltaC(v)\ge 0$ such that
      \begin{equation}
        \label{eq:met.q.u.c}
        \nu_0(v)\Bigabssqr{\Gamma_v u - \dashint_X u \dd \nu_v}
        \le \deltaC(v)^2 \energy_{\Xv}(u)
      \end{equation}
      holds for all $u \in \dom \energy_X$;
    \item
      \label{graph.emb.en.d}
      for each $v \in V$ and $e \in E_v$ there exists
      $\map{\Gamma_{v,e}}{\dom \energy_{\Xv}}\C$ and $\wtdeltaD(v)\ge
      0$ such that
      \begin{equation}
        \label{eq:met.q.u.d}
        c^2\tau \energy_X(u,\psi_v) 
        = \sum_{e \in E_v} \gamma_e (\Gamma_{v,e}u-\Gamma_{v_e,e} u)
      \end{equation}
      and
      \begin{equation}
        \label{eq:met.q.u.d'}
        \tag{\ref{eq:met.q.u.d}'}
        \frac 1{c^2}\sum_{e \in E_v} \gamma_e 
        \abssqr{\Gamma_{v,e} u -\Gamma_v u}
        \le \wtdeltaD(v)^2  \;\energy_{\Xv}(u)
      \end{equation}
      hold for all $u \in \dom \energy_X$.
    \end{enumerate}
  \end{subequations}
  Then $\energy$ and $\wt \energy=\tau \energy_X$ are
  $\delta$-quasi-unitarily equivalent with isometric and energy
  rescaling factors $c>0$ and $\tau>0$, respectively, where
  \begin{equation}
    \label{eq:q-u-e.delta}
    \delta^2:=
    \max \Bigl\{
      2\alpha_\infty,\;
      \frac {2 \mu_\infty}{\gamma_0},
      \frac 2{\tau \lambda_2},\;
      \frac 2\tau \sup_{v \in V} \wtdeltaC(v)^2,\;
      \frac 4\tau \sup_{v \in V} \wtdeltaD(v)^2
    \Bigr\}.
  \end{equation}
\end{theorem}

\begin{remarks}
  \label{rem:q-u-e}
  \indent
  \begin{enumerate}
  \item
    \label{cap.zero}
    If points have positive capacity in $X$, e.g.\ if $X$ is a
    \pcf fractal (see~\cite{post-simmer:pre17a}) or a metric graph
    (see~\Sec{met.graphs}), one can choose
    \begin{equation*}
      \Gamma_vu=\Gamma_{v,e}u = u(v),
    \end{equation*}
    and hence $\wtdeltaD(v)=0$.  In this case,~\eqref{eq:met.q.u.c}
    follows from a H\"older estimate of $u$, and $\wtdeltaC(v)^2$ is
    of order as the diameter of $\Xv$ (for suitable spaces $X$).

  \item The condition in~\eqref{eq:met.q.u.c} means that the
    ``evaluation'' $\Gamma_vu$ is close to the (weighted) average on
    $\Xv$.  In particular, we conclude from~\eqref{eq:met.q.u.c} that
    $\Gamma_v \1_{\Xv}=1$, as $\dashint_X \1_{\Xv} \dd \nu_v=1$ and
    $\energy_{X_v}(\1_{\Xv})=0$ (see
    \Defenum{graph.emb.in.met.space}{graph.emb.in.met.space.c}).

  \item The choice $\Gamma_vu=\avint_X u \dd \nu_v$ is possible, but
    in our applications bad: Although then $\deltaC(v)=0$, one obtains
    $\Gamma_{v,e}u=u(v)$ from~\eqref{eq:met.q.u.d} in the metric graph
    case.  But then the estimate on $u(v)-\avint_X u \dd \nu_v$
    appears in~\eqref{eq:met.q.u.d'} with the edge weight $\gamma_e$
    as factor: this weight is generally \emph{large}, and the
    resulting error term $\deltaD(v)$ will not be small (compared to
    the choice $\Gamma_v u=\Gamma_{v,e}u=u(v)$ as in
    \itemref{cap.zero}, where we have to estimate $u(v)-\avint_X u \dd
    \nu_v$ together with the \emph{small} factor $\nu_0(v)$).  A
    similar remark holds for graph-like manifolds and fractals.

  \item The conditions~\eqref{eq:met.q.u.d}--\eqref{eq:met.q.u.d'} can
    be understood as follows: It is not hard to see that 
    \begin{equation*}
      \map \Gamma {\dom \energy_X}{\lsqr{V,\mu}}
      \qquadtext{with}
      \Gamma u=(\Gamma_v u)_{v \in V}
    \end{equation*}
    is bounded using~\eqref{eq:met.q.u.c} and~\eqref{eq:energy.local}.
    In particular, $(\Gamma,\lsqr{V,\mu})$ is a (generalised) boundary
    pair in the sense of~\cite{post:16}.
    
    By~\eqref{eq:met.q.u.d'}, $\Gamma_v u$ and $\Gamma_{v,e} u$ are
    close to each other.  Assume here for simplicity that $\Gamma_v u
    = \Gamma_{v,e} u$ as in~\itemref{cap.zero}.  If $\psi_v$ is
    \emph{harmonic}, i.e., if $\psi_v$ minimises $\energy_X(\psi_v)$
    among all functions $u \in \dom \energy_X$ with $\Gamma_vu=1$,
    then $\energy_X(u,\psi_v)=\iprod[\lsqr{V,\mu}]{\Lambda_0\Gamma
      u}{\delta_v}$, where $\Lambda_0$ is the Dirichlet-to-Neumann
    operator of the boundary pair $(\Gamma,\lsqr{V,\mu})$ (at the
    spectral value $0$) and where $\delta_v$ is the Kronecker delta,
    see~\cite{post:16}.  In particular,
    conditions~\eqref{eq:met.q.u.d}--\eqref{eq:met.q.u.d'} with
    $\wtdeltaD(v)=0$ mean that $c^2\tau \Lambda_0 =
    \Delta_{(G,\mu,\gamma)}$, i.e., that $c^2\tau \Lambda_0$ equals
    the discrete Laplacian $\Delta_{(G,\mu,\gamma)}$ given
    by~\eqref{eq:disc.lapl}.  Note that this is the operator
    associated with the form $\energy$ on $(G,\mu,\gamma)$.  In
    particular, the discrete energy form $\energy$ equals the
    (rescaled) Dirichlet-to-Neumann form of the boundary pair.  The
    general case $\wtdeltaD(v)>0$ is just a small deviation from this
    situation.
  \end{enumerate}
\end{remarks}

\begin{proof}[Proof of \Thm{q-u-e}]
  We have already shown in \Prp{q-u.b} that $J$ is
  $\max\{\deltaA,\deltaB\}$-quasi-unitary with adjoint $J'$,
  explaining the first three members in the definition of $\delta$.

  For the remaining conditions of quasi-unitary equivalence, we have
  to define the identification operators on the level of the energy
  forms.  Namely, we set
  \begin{equation*}
    \map{J^1} {\HS=\lsqr{V,\mu}} {\wt \HS^1=\dom \energy_X}, \qquad J^1f =Jf,
  \end{equation*}
  and this is well-defined as $\psi_v \in \dom \wt \energy$ and $Jf \in
  \dom \wt \energy$ for any $f \in \HS=\lsqr{V,\mu}$.  For the opposite
  direction, we define
  \begin{equation*}
    \map{J^{\prime 1}}{\wt\HS^1}{\HS=\lsqr{V,\mu}} \qquadtext{by}
    (J^{\prime 1} u)(v) := \frac 1 c (\Gamma_vu), \quad v \in V.
  \end{equation*}
  The first condition of~\eqref{eq:quasi-uni.c} is trivially fulfilled, and
  for the second, we estimate
  \begin{align*}
    \normsqr[\lsqr{V,\mu}] {J'u-J^{\prime 1}u}
    &= \sum_{v \in V} \nu_0(v)
      \Bigabssqr{\dashint_X u \dd \nu_v - \Gamma_vu}\\
    &\le \sup_{v \in V} \wtdeltaC(v)^2
        \sum_{v \in V} \energy_{\Xv}(u)
    \le \frac 2{\tau} \sup_{v \in V} \wtdeltaC(v)^2 \wt \energy(u).
  \end{align*}
  using~\eqref{eq:def.nu.decomp} for the first equality, 
  using~\eqref{eq:met.q.u.c} and~\eqref{eq:energy.local}.

  We now check estimate~\eqref{eq:quasi-uni.d}: we have
  \begin{align*}
    \energy(f, J^{\prime 1} u) - \wt \energy(Jf,u)
    &=\frac
    1c \sum_{e \in E} (df)_e (d\Gamma \conj u)_e \gamma_e
    - c \sum_{v \in V} f(v) \wt \energy(\psi_v,u)\\
    &=\frac 1c \sum_{e \in E}
    \gamma_e (df)_e\bigl(
      (\Gamma_{\bd_+e} \conj u - \Gamma_{\bd_+e,e} \conj u) -
      (\Gamma_{\bd_-e} \conj u - \Gamma_{\bd_-e,e} \conj u)
    \bigr)
  \end{align*}
  using~\eqref{eq:met.q.u.d}, where $\Gamma u=(\Gamma_vu)_{v \in V}$
  and $(dh)_e=h(\bd_+e)-h(\bd_-e)$, and where we use the reordering
  $\sum_{v \in V} \sum_{e \in E_v} = \sum_{e \in E} \sum_{v=\bd_\pm
    e}$.  In particular,
  \begin{align*}
    \bigabssqr{\energy(f, J^{\prime 1} u) - \wt \energy(Jf,u)}
    &\leCS \frac 2{c^2} \energy(f)
    \sum_{e \in E} 
       \gamma_e \sum_{v=\bd_\pm e} \abssqr{\Gamma_v u - \Gamma_{v,e} u}\\
    &= \frac 2{c^2} \energy(f)
    \sum_{v \in V} 
       \sum_{e \in E_v} \gamma_e \abssqr{\Gamma_v u - \Gamma_{v,e} u}\\
    &\le 2\energy(f)
      \sum_{v \in V}
      \wtdeltaD(v)^2 \energy_{\Xv}(u)
    \le \frac 4\tau \sup_{v \in V} \wtdeltaD(v)^2
      \energy(f) \wt \energy(u)
  \end{align*}
  using \eqref{eq:met.q.u.d'} for the second last estimate and
  again~\eqref{eq:energy.local} for the last one.
\end{proof}

%
\section{Convergence of energy forms on metric and discrete graphs}
\label{sec:met.graphs}
%
In this section, the metric measure space $X$ is a metric graph,
called $\MetGr$ here.

\subsection{Metric graphs}
\label{ssec:met.graphs}
We briefly introduce the notion of a metric graph here.  More details
on metric graphs can be found in~\cite{post:12}
or~\cite{berkolaiko-kuchment:13}.  For a metric graph, we need a
discrete graph $(V,E,\bd)$ together with a function $\map \ell E
{(0,\infty)}$.  We will interpret $\ell_e>0$ as the \emph{length} of
an edge $e$.  A \emph{metric graph} $\MetGr$ is now given by
$(G,\ell)$ and can be defined as the topological space
\begin{equation*}
  \MetGr := \bigdcup_{e \in E} \Me / \omega,
\end{equation*}
where $\Me := [0,\ell_e]$, and where $\map \omega {\bigdcup_e \{0,
  \ell_e\}} V$ identifies $0 \in \Me$ with the initial vertex $\bd_-e
\in V$ and $\ell_e \in \Me$ with the terminal vertex $\bd_+e \in V$.
The topological space $\MetGr$ is a metric space by choosing as
distance of two points $x,y \in \MetGr$ the length of the (not
necessarily unique) shortest path $\map{\gamma_{x,y}}{[0,a]} \MetGr$
in $\MetGr$ realising the distance $d(x,y)=a$.  Moreover, we have a
canonical measure $\nu$ on $\MetGr$, given by the sum of the Lebesgue
measures $\dd x_e$ on each interval $\Me$.

The Hilbert space is here
\begin{equation*}
  \wt \HS = \Lsqr{\MetGr, \nu}, \qquad
  \normsqr[\Lsqr{\MetGr,\nu}] u
  = \sum_{e \in E} \int_0^{\ell_e} \abssqr{u_e(x)} \dd x,
\end{equation*}
where we consider $u=(u_e)_{e \in E}$ with $\map {u_e}{[0,\ell_e]} \C$.
The energy form on $\MetGr$ is
\begin{equation}
  \label{eq:met.graph.energy}
  \energy_\MetGr(u)=\normsqr[\Lsqr{\MetGr,\nu}] {u'}
  = \sum_{e \in E} \int_0^{\ell_e} \abssqr{u'_e(x)} \dd x
\end{equation}
with
\begin{equation*}
  \wt\HS^1=\Sob \MetGr 
  =\Cont \MetGr \cap \bigoplus_{e \in E} \Sob{[0,\ell_e]},
\end{equation*}
i.e., $u \in \Sob \MetGr$ if and only if $u_e \in \Sob{[0,\ell_e]}$,
$\sum_e \normsqr[{\Sob{[0,\ell_e]}}] {u_e} < \infty$ (the latter
condition is only necessary if $E$ is not finite) and
\begin{equation*}
  u_e(v):=
  \begin{cases}
    u_e(0), & v=\bd_-e,\\
    u_e(\ell_e), & v=\bd_+e
  \end{cases}
  \qquad
  \text{is independent of $e \in E_v$}
\end{equation*}
for all $v \in V$.  Note that the corresponding operator
$\Delta_{\MetGr}$ acts as $(\Delta_{\MetGr}f)_e=-f_e''$ on each edge
with $f \in \bigoplus_e \Sob[2]{M_e}$, where $f$ is continuous at the
vertices and where $\sum_{e \in E_v} f_e'(v)=0$.  Here, $f_e'(v)$
denotes the \emph{inwards} derivative of $f$ at $v$ along $e \in E_v$.
This operator is called the \emph{standard} (or by many authors also)
\emph{Kirchhoff Laplacian}.


\subsection{Quasi-unitary equivalence of metric and discrete graphs}
\label{ssec:q-u-e.mg}
Let us first specify the partition of unity related to the graph
structure (see \Def{graph.emb.in.met.space}).  Let $\map
{\psi_v}\MetGr {[0,1]}$ be the function, affine linear on each edge
$M_e$, such that $\psi_v(v)=1$ and $\psi_v(v')=0$ if $v' \in V
\setminus \{v\}$.  Then the vertex neighbourhood (i.e., the support of
$\psi_v$) is
\begin{equation*}
  \Mv = \supp \psi_v = \bigdcup_{e \in E_v} \Me/\omega,
\end{equation*}
i.e., the star graph around $v$ consisting of all edges adjacent with
$v$.  Note that with this definition, $\intrMv \cap \intrMv[v']
\ne\emptyset$ if and only if $v \sim v'$ and that the edge
neighbourhood $\Me = \Mv[\bd_-e] \cap \Mv[\bd_+e]$ is the edge (as
interval) $\Me$ as already defined above.  Moreover, the vertex core
is $\cMv=\psi_v^{-1}\{1\}=\{v\}$, i.e., a point in $\MetGr$.

The vertex weights here are given by
\begin{equation}
  \label{eq:vx.meas.mg}
  \nu(v) = \int_\MetGr \psi_v \dd x
  = \sum_{e \in E_v} \frac 1{\ell_e}\int_0^{\ell_e} x \dd x
  =\frac 12 \sum_{e \in E_v} \ell_e
  = \frac 12 \nu(\Mv).
\end{equation}
The following definition assures that the weights $\mu$ and $\nu$ are
compatible:

\begin{definition}
  \label{def:mg.comp}
  We say that a metric graph $\MetGr$ and a weighted discrete graph
  $(G,\mu,\gamma)$ are \emph{compatible}, if the underlying discrete
  graphs are the same and if there exists $c>0$ and $\tau>0$ such that
  the length function $\ell$ and the weights $\mu$ and $\gamma$ fulfil
  \begin{subequations}
    \begin{equation}
      \label{eq:mu.nu.comp.mg}
      \frac 1{2\mu(v)}\sum_{e \in E_v} \ell_e
      =\frac 1{c^2}
    \end{equation}
    for all $v \in V$ and
    \begin{equation}
      \label{eq:ell.tau}
      \ell_e \gamma_e = c^2 \tau
    \end{equation}
  \end{subequations}
  for all $e \in E$.
\end{definition}

\begin{remark}
  \label{rem:mg.comp}
  Let $\MetGr$ and $(G,\mu,\gamma)$ be compatible, then the following holds:
  \begin{enumerate}
  \item 
    \label{mg.comp.a}
    The measures $\mu$ and $\nu$ are compatible in the sense of
    \Def{graph.emb.in.met.space}.  Moreover, $\nu_0(v)=\nu(v)$ and
    $\alpha_\infty=0$, see~\eqref{eq:def.alpha}).

  \item 
    \label{mg.comp.b}
    Condition~\eqref{eq:ell.tau} is dictated
    by~\eqref{eq:why.ell.tau}.  Moreover, a lower bound $\gamma_0
    :=\inf_{e \in E}\gamma_e>0$ implies an upper bound
    \begin{equation}
      \label{eq:ell.infty}
      \ell_\infty:=\sup_{e \in E} \ell_e = \frac {c^2\tau}{\gamma_0}<\infty.
    \end{equation}

  \item 
    \label{mg.comp.c}
    There is still some freedom in the choice of the parameters
    $\ell_e$, $\tau$ and $c$ (and the parameters $\mu(v)$ and
    $\gamma_e$ from the weighed graphs), as they have to fulfil only
    two equations~\eqref{eq:mu.nu.comp.mg}--\eqref{eq:ell.tau}.
    Given $\ell_e$, $\mu(v)$ and $\gamma_e$, we conclude
    \begin{align}
      \label{eq:c.tau.mg}
      c&=\Bigl( \frac 1{2\mu(v)}\sum_{e \in E_v} \ell_e \Bigr)^{-1/2}
      \qquad\text{and} &\tau &= \frac 1{2\mu(v)}\sum_{e \in E_v}
      \ell_e^2\gamma_e.
    \end{align}

  \item 
    \label{mg.comp.d}
    Note that~\eqref{eq:mu.nu.comp.mg} and~\eqref{eq:ell.tau} give a
    restriction on the vertex and edge weights of the weighted graph,
    namely that
    \begin{equation}
      \label{eq:cond.weighted.graph}
      \frac1{2\mu(v)} \sum_{e \in E_v} \frac 1{\gamma_e}
      = \frac1{c^4 \tau}
    \end{equation}
    is independent of $v \in V$.  This condition looks a bit
    surprising --- a natural condition for a weighted graph would be
    that the \emph{relative weight} is independent of $v \in V$
    (see~\eqref{eq:rel.weight}); in this case, the discrete Laplacian
    and a corresponding weighted adjacency operator are related by an
    affine linear transformation.
  \end{enumerate}
\end{remark}
For Sobolev spaces of order $1$ on one-dimensional spaces, the
evaluation in a point is well-defined, hence $u \mapsto u(v)$ makes
sense for $u \in \Sob \MetGr$ (see e.g.~\cite[Ch.~2]{post:12}).  We
therefore define
\begin{equation*}
  \Gamma_v u := u(v).
\end{equation*}
Let us now check~\eqref{eq:met.q.u.c},
\eqref{eq:met.q.u.d}--\eqref{eq:met.q.u.d'}:
\begin{lemma}
  \label{lem:av.mg}
  Let $u \in \Sob \MetGr$, then~\eqref{eq:met.q.u.c} holds, i.e.,
  \begin{equation*}
    \nu(v)\Bigabssqr{\dashint_{\Mv} u \dd \nu_v - \Gamma_vu}
    \le \wtdeltaC(v)^2 \;\energy_{\Mv}(u)
    \qquadtext{with}
    \wtdeltaC(v)^2:=\frac {\ell_\infty^2}2.
  \end{equation*}
  If we also have~\eqref{eq:ell.tau} then~\eqref{eq:met.q.u.d}
  and~\eqref{eq:met.q.u.d'} hold with $\Gamma_{v,e} u = \Gamma_v u$
  and $\wtdeltaD(v)=0$.
\end{lemma}
\begin{proof}
  Let $v \in V$ and $e \in E_v$.  Assume that $v$ is the terminal
  vertex for all $e \in E_v$, i.e., $v$ corresponds to $\ell_e \in
  \Me$ for all $e \in E_v$.  Then $\dd \nu_v(x)=(x/\ell_e)\dd x$ on
  $\Me$.  The proof of the first assertion is an application of the
  fundamental theorem of calculus, namely we have $u_e(x)-u(v) =
  \int_{\ell_e}^x u_e'(t) \dd t$.  After integration with respect to
  $x$ and the probability measure $\avint_{\Mv} \dd \nu_v$ we obtain
  \begin{align*}
    \dashint_{\Mv} u \dd \nu_v - u(v)
    = \dashint_{\Mv} (u  - u(v)) \dd \nu_v
    &= \frac 1{\nu(v)} \sum_{e \in E_v} \int_0^{\ell_e} (u_e(x)  - u(v)) 
          \dd\nu_v(x)\\
    &= \frac 1{\nu(v)} \sum_{e \in E_v}  \int_0^{\ell_e}
                     \int_{\ell_e}^x u_e'(t) \dd t \dd \nu_v(x),
  \end{align*}
  hence
  \begin{align*}
    \Bigabssqr{\dashint_{\Mv} u \dd \nu_v - u(v)}
    &\le \frac 1{\nu(v)^2} 
       \Bigl(\sum_{e \in E_v} \int_0^{\ell_e} \int_0^{\ell_e} \abs{u_e'(t)} \dd t 
                            \dd \nu_v(x)\Bigr)^2\\
    &= \frac1{\nu(v)^2}
        \Bigl(\sum_{e \in E_v} \int_0^{\ell_e} \frac{\ell_e}2 \abs{u_e'(t)} \dd t 
        \Bigr)^2\\
    &\leCS \frac1{4\nu(v)^2} \Bigl(\sum_{e \in E_v} \ell_e^3 \Bigr)
           \int_{M_v} \abssqr{u'} \dd \nu
    \le \frac{\max_{e \in E_v} \ell_e^2}{2\nu(v)} \; \energy_{\Mv} (u)
  \end{align*}
  using $\dd \nu_v(x)=(x/\ell_e)\dd x$ for the equality
  and~\eqref{eq:vx.meas.mg} for the last inequality.

  For the validity of the last assertion we calculate 
  \begin{equation}
    \label{eq:why.ell.tau}
    c^2 \tau \energy_\MetGr(u,\psi_v)
    = \sum_{e \in E_v} \frac {c^2 \tau}{\ell_e} \int_0^{\ell_e} u_e'(x) \dd x
    = \sum_{e \in E_v} \gamma_e \bigl(\Gamma_v u - \Gamma_{v_e} u \bigr)
  \end{equation}
  using $(\psi_v')_e=1/\ell_e$ for the first and~\eqref{eq:ell.tau}
  for the second equality.  Note that $v$ corresponds to $\ell_e$ and
  $v_e$ to $0$.  In particular, we can choose $\Gamma_{v,e} u =
  \Gamma_v u$, hence $\deltaD(v)=0$.
\end{proof}
We need a lower bound on the second eigenvalue:
\begin{lemma}
  \label{lem:2nd.ev.mg}
  We have 
  \begin{equation*}
    \frac 2{\ell_\infty^2}
    \le \lambda_2(\Mv,\psi_v).
  \end{equation*}
\end{lemma}
\begin{proof}
  Eigenfunctions of the weighted problem on an edge are solutions of
  the ODE $-u_e''(x) = \lambda x u_e(x)$, where $x \in [0,\ell_e]$,
  and $0$ corresponds to the vertex of degree $1$ on the star graph.
  Such eigenfunctions are linear combinations of (rescaled) Airy
  functions, and hence continuous also at $x=0$.  In particular, the
  solutions are also in the unweighted Hilbert space $\Lsqr \Xv$, and
  we obtain the estimate $\lambda_2(\Xv,\psi_v) \ge \lambda_2(\Xv)$
  from \Lem{2nd.ev.alt}.  Consider now the Rayleigh quotient for the
  unweighted problem, it is given by
  \begin{equation*}
    \frac{\energy_{\Mv}(u)} {\normsqr[\Lsqr \Xv] u}
    = \frac{\sum_{e \in E_v} \int_0^{\ell_e} \abssqr{u_e'(x)} \dd x}
        {\sum_{e \in E_v} \int_0^{\ell_e} \abssqr{u_e(x)} \dd x}
    = \frac{\sum_{e \in E_v} \frac1{\ell_e}\int_0^1 \abssqr{\wt u_e'(t)} \dd t}
        {\sum_{e \in E_v} \ell_e \int_0^1 \abssqr{\wt u_e(t)} \dd t},
  \end{equation*}
  where $x=t \ell_e$ and $\wt u(t)=u(t \ell_e)$; and the latter
  expression is monotonously decreasing in $\ell_e$.  In particular,
  as $\ell_e$ is bounded from above by $\ell_\infty:=\max_{e \in E_v}
  \ell_e$, we have $\lambda_2 \ge \lambda_{2,0}/\ell_\infty^2$, where
  $\lambda_{2,0}=\pi^2/4 \ge 2$ is the second eigenvalue of a star
  graph with all edges having length $1$.
\end{proof}

We are now prepared to prove the main result of this subsection.
\begin{theorem}
  \label{thm:mg.q-u-e}
  Assume that $(G,\mu,\gamma)$ is a weighted graph with
  \begin{equation*}
    \mu_\infty:=\sup_{v \in V} \mu(v) < \infty,
    \qquadtext{and}
    0<\gamma_0:=\inf_{e \in E} \gamma_e
    \le \gamma_\infty:=\sup_{e \in E} \gamma_e
    <\infty.
  \end{equation*}
  Assume in addition that $M$ is a metric graph compatible with
  $(G,\mu,\gamma)$, i.e., its edge lengths $\ell_e$
  fulfil~\eqref{eq:mu.nu.comp.mg}--\eqref{eq:ell.tau}, namely
  \begin{equation*}
    \frac 1{2\mu(v)}\sum_{e \in E_v} \ell_e
    =\frac 1{c^2}
    \qquadtext{and}
    \ell_e=\frac {c^2 \tau}{\gamma_e}
  \end{equation*}
  for some $c>0$ and $\tau>0$, independently of $v \in V$ and $e \in
  E$.  Then the graph energy form $\energy$ associated with the
  weighted discrete graph $(G,\mu,\gamma)$ and the rescaled metric
  graph energy form $\wt \energy=\tau \energy_M$ are
  $\delta$-quasi-unitarily equivalent with
  \begin{equation*}
    \delta^2:=
    2\frac{\gamma_\infty}{\gamma_0}\cdot \frac{\mu_\infty}{\gamma_0}.
  \end{equation*}
\end{theorem}
\begin{proof} 
  Note first that $(G,\mu,\gamma)$ is uniformly embedded into
  $(\MetGr,\nu,\energy_\MetGr)$ (see \Def{graph.emb.in.met.space}) by
  the assumptions of the theorem: in particular, $\lambda_2(M_v,\psi)
  \ge 2/\ell_\infty^2$ by \Lem{2nd.ev.mg}.  The remaining assumptions
  of \Thm{q-u-e} are fulfilled by \Lem{av.mg}.  Let us now check the
  individual terms in the error $\delta$ in~\eqref{eq:q-u-e.delta}: As
  the weights are compatible, we have $\alpha_\infty=0$.  Moreover,
  the error term $2\mu_\infty/\gamma_0$ is already covered as
  $\gamma_\infty/\gamma_0 \ge 1$.  For the third term
  in~\eqref{eq:q-u-e.delta} we have
  \begin{equation*}
    \frac 2 {\tau \lambda_2}
    \le \frac{\ell_\infty^2}\tau
    = \frac{c^4\tau}{\gamma_0^2}
    \le \frac{2\mu_\infty\gamma_\infty}{\gamma_0^2}
    = \delta^2
  \end{equation*}
  using again \Lem{2nd.ev.mg} for the first, \eqref{eq:ell.infty} for
  the second and~\eqref{eq:cond.weighted.graph} for the third step.
  The fourth error term in~\eqref{eq:q-u-e.delta} (the one with
  $\deltaC(v)$) is treated in the same way as the third one as
  $2\deltaC(v)^2/\tau \le \ell_\infty^2/\tau$ by \Lem{av.mg}.  The
  last error term is $0$ by \Lem{av.mg}.
\end{proof}

From \Defenum{unif.weighted.graph}{unif.weighted.graph.b}
and~\eqref{eq:rel.weight.bdd} we immediately conclude:
\begin{corollary}
  \label{cor:mg.q-u-e}
  Assume that $(G,\mu,\gamma)$ is a
  $(d_\infty,\muVar,\gammaVar)$-uniform weighted graph with
  corresponding compatible metric graph lengths then $\energy$ and
  $\wt \energy$ are $\delta$-quasi-unitarily equivalent with
  \begin{equation*}
    \delta^2
    =2\gammaVar^2 d_\infty \muVar \cdot \frac 1{\rho_0},
  \end{equation*}
  where $\rho_0=\inf_{v \in V} \rho(v) = \inf_{v \in V}{\frac
    1{\mu(v)}\sum_{e \in E_v}\gamma_e}$ is a lower bound on the
  relative weight.
\end{corollary}

\subsubsection*{Metric graphs approximated by discrete weighted
  subdivision graphs}
We have another application of our result: A \emph{subdivision graph}
$SG$ of a discrete graph $G=(V,E,\partial)$ is a discrete graph with
additional vertices on the edges.  We denote the graph objects
associated with $SG$ by $V(SG)$, $E(SG)$ etc.

If $M$ is a metric graph with underlying discrete graph $G$ and length
function $\map \ell E {(0,\infty)}$, then we call $SM$ a \emph{metric
  subdivision graph} if the lengths $\ell_{e_1},\dots,\ell_{e_r}$ of
the additional edges $e_1,\dots,e_r$ on the original edge $e$ add up
to the original length $\ell_e$, i.e.,
\begin{equation*}
  \sum_{j=1}^r \ell_{e_j}=\ell_e.
\end{equation*}
Note that additional vertices of degree $2$ on an edge lead to
unitarily equivalent metric graph energy forms and Laplacians with
natural unitary map.  In particular, the energy form and the Laplacian
on a metric subdivision graph are unitarily equivalent with the energy
form and the Laplacian on the original metric graph.  We define
\begin{equation*}
  \ell_0(SM):= \inf_{e \in E(SG)} \ell_e
  \qquadtext{and}
  \ell_\infty(SM):= \sup_{e \in E(SG)} \ell_e,
\end{equation*}
the \emph{minimal} and the \emph{maximal mesh width} of the
subdivision graph $SM$, respectively.  We have now the following
result:
\begin{corollary}
  \label{cor:subdiv.graph}
  Assume that $M$ is a metric graph with edge length fulfilling
  $0<\ell_0\le \ell_e \le \ell_\infty < \infty$ for all $e \in E$ and
  uniformly bounded degree, i.e., $\deg v \le d_\infty$ for all $v \in
  V$.  Then there is a sequence of metric subdivision graphs $SM_m$
  and compatible weighted discrete subdivision graphs
  $(SG_m,\mu_m,\gamma_m)$ such that the associated discrete energy
  form $\energy_m$ is $\delta_m$-unitarily equivalent with the energy
  form $\energy_M$ of the original metric graph, where $\delta_m^2\le
  d_\infty \ell_\infty(SM_m)^3/\ell_0(SM_m)$.
\end{corollary}
\begin{proof}
  Let $SM_m$ be a sequence of metric subdivision graphs with edge
  length function denoted by $\map{\ell_m}{E(SM_m)} {(0,\infty)}$ with
  $\ell_0(SM_m) \to 0$.  Define the weights $\mu_m$ and $\gamma_m$ of
  the discrete underlying subdivision graph $SG_m$ by
  \begin{equation*}
    \mu_m(v):=\frac 12 \sum_{e \in E_v(SG_m)} \ell_{m,e}
    \quadtext{and}
    \gamma_{m,e}:=\frac 1{\ell_{m,e}}.
  \end{equation*}
  In particular, we then have $c=1$ and $\tau=1$ and $SM_m$ and
  $(SG_m,\mu_m,\gamma_m)$ are compatible.  From \Thm{mg.q-u-e} we
  conclude that the energy forms associated with the metric
  subdivision graph $\energy_{SM_m}$ and the weighted discrete
  subdivision graph $\energy_{SG_m}$ are $\delta_m$-quasi-unitarily
  equivalent with
  \begin{equation*}
    \delta_m^2
    =\frac {2\ell_\infty(SM_m)}{\ell_0(SM_m)} \cdot
       \frac{\sup_{v \in V}\sum_{e \in E_v(SG_m)}\ell_{m,e}/2}
            {\ell_\infty(SM_m)^{-1}}
    \le d_\infty \ell_\infty(SM_m)^3/\ell_0(SM_m).
  \end{equation*}
  Note that the maximal degree of a subdivision graph is the same as
  for the original graph and that $\energy_{SM_m}$ and $\energy_M$ are
  unitarily equivalent.
\end{proof}

%
\section{Convergence of energy forms on graph-like manifolds and
  discrete graphs}
\label{sec:mfds}
%

\subsection{Graph-like manifolds}
\label{ssec:mfds}
We introduce here the notion of a graph-like manifold.  More details
can be found in~\cite{post:12}.  Let $\Mfd$ be a Riemannian manifold
of dimension $d \ge 2$.  The standard example of a graph-like manifold
with boundary is the thickened metric graph as in \Ex{emb} or
\Fig{graph-like-mfd}.
\begin{definition}
  \label{def:gl-mfd}
  We say that $\Mfd$ is a \emph{graph-like manifold} with associated
  discrete graph $(V,E,\bd)$ and edge length function $\map \ell E
  {(0,\infty)}$ if there are compact subsets $\cXv$ and $\Xe$ of $X$
  with the following properties:
  \begin{enumerate}
  \item
    \label{gl-mfd.a}
    $\Mfd=\bigcup_{v \in V} \cXv \cup \bigcup_{e \in E}\Xe$ and $\cXv
    \cap \Xe \neq \emptyset$ if and only if $e \in E_v$; all other
    sets $\cXv$ and $\Xe$ are pairwise disjoint;
  \item
    \label{gl-mfd.b}
    $\Xe$ is isometric with $\Me \times Y_e$, where $\Me=[0,\ell_e]$
    for some $\ell_e>0$ and some $(d-1)$-dimensional Riemannian
    manifold $Y_e$;
  \item
    \label{gl-mfd.c}
    there exists $\kappa \in (0,1]$ such that $\bd_e \cXv:= \cXv \cap
    \Xe$ (isometric with $Y_e$) has a $\kappa\ell_e$-collar
    neighbourhood $\Xve$ inside $\cXv$, i.e., $\Xve$ is isometric with
    $[0,\kappa \ell_e]\times Y_e$; we assume that $(\Xve)_{e \in E_v}$
    are pairwise disjoint.
  \end{enumerate}
  We call $\cXv$ the \emph{core vertex neighbourhood} of $v \in V$ and
  $\Xe$ the \emph{edge neighbourhood} of $e \in E$.  We call $Y_e$ the
  \emph{transversal manifold} of $e$.  Moreover, we call $\Xv:=\cXv
  \cup \bigcup_{e \in E_v} \Xe$ the \emph{(enlarged) vertex
    neighbourhood} of $v$.
\end{definition}
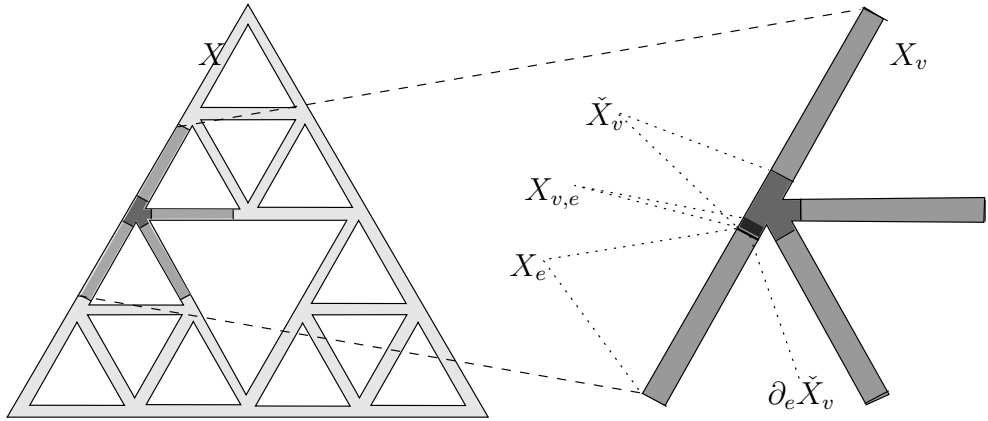
\begin{figure}[h]
  \label{fig:graph-like-mfd}
  \centering
  \setlength{\unitlength}{1mm}
  \begin{picture}(120,55)
    \includegraphics[width=0.8\textwidth]{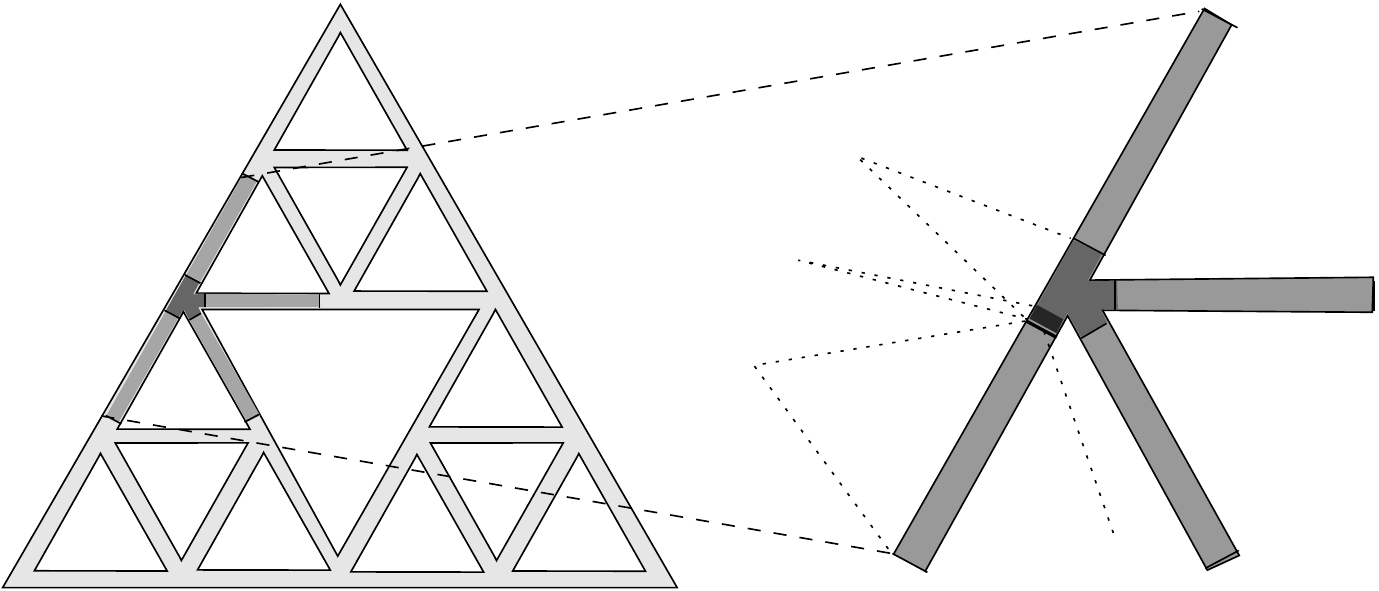}  
    \put(-104,47){$X$}
    \put(-13,47){$\Xv$}
    \put(-53,39){$\cXv$}
    \put(-61,29){$\Xve$}
    \put(-63,19){$\Xe$}
    \put(-29,2){$\bd_e\cXv$}
  \end{picture}
  \caption{Left: A graph-like manifold. Right: The vertex
    neighbourhood $\Xv$ (grey) with its core vertex manifold $\cXv$
    (dark grey) and the collar neighbourhood (very dark grey) with the
    boundary component $\bd_e \cXv$ (thick black line); the edge $e$
    corresponds to the lower left leg.}
\end{figure}
\begin{remark*}
  \indent
  \begin{enumerate}
  \item A graph-like manifold may have boundary or not; the boundary
    may even be Lipschitz (see \cite[App.~A]{mitrea-taylor:99}) for a
    precise definition).
  \item A graph-like manifold $\Mfd$ can be constructed from a metric
    graph $\MetGr$ with the same edge length function $\ell$.  In this
    case, $\Mfd$ is defined as an abstract space.  For the case when
    the metric graph and the graph-like manifold are embedded in
    $\R^d$, see \Ex{emb}.
  \item The decomposition into vertex and edge neighbourhoods is, of
    course, not unique.
  \item For a compact graph-like manifold, condition
    \Defenum{gl-mfd}{gl-mfd.c} follows from
    \itemref{gl-mfd.a}-\itemref{gl-mfd.b}: Take away a little piece of
    $\Xe$ and add it to $\Xv$, this means that $\ell_e$ becomes a bit
    smaller.
  \end{enumerate}
\end{remark*}
An important example is a neighbourhood of a metric graph embedded in
$\R^d$:
\begin{example}[Thickened metric graph as graph-like manifold]
  \label{ex:emb}
  Let $\MetGr$ be a compact metric graph embedded in $\R^d$ in such a
  way that the edges $\Me$ are line segments.  Then the edge lengths
  of an edge $e$ is $\abs{\bd_+e-\bd_-e}$ (where $V$ is considered as
  a subset of $\R^d$).  Let $\Mfd$ be the closed $\eps$-neighbourhood
  of $\MetGr$ in $\R^d$.  If $\eps>0$ is small enough, it can be seen
  that $\MetGr$ is a graph-like manifold.  In particular, one can
  choose $\ell_e=(1-2\eps) \abs{\bd_+e-\bd_-e}$ as edge length
  function.
\end{example}
Metric graphs with edges embedded as curved segments can also be
treated as a perturbation of abstract metric graphs (not necessarily
embedded) with straight edges, see~\cite[Sec.~5.4 and~6.7]{post:12}).

Let $\nu$ denote the Riemannian measure on $\Mfd$.  The associated
Hilbert space is
\begin{equation*}
  \wt \HS = \Lsqr{\Mfd, \nu}, \qquad
  \normsqr[\Lsqr{\Mfd,\nu}] u
  = \int_\Mfd \abssqr{u(x)} \dd \nu(x).
\end{equation*}
The energy form on $\Mfd$ is
\begin{equation}
  \label{eq:mfd.energy}
  \energy_\Mfd(u)=\int_\Mfd \abssqr[x]{\nabla u(x)} \dd\nu(x)
\end{equation}
where $\nabla$ is the gradient and $\abs[x]\cdot$ is the norm induced
by the Riemannian metric tensor at $x \in \Mfd$ and where
$\wt\HS^1=\Sob \Mfd$ is the closure of Lipschitz continuous functions
with compact support in $\Mfd$ with respect to the norm given by
$\normsqr[\Sob \Mfd] u = \normsqr[\Lsqr X] u + \energy_\Mfd(u)$.

\subsection{Quasi-unitary equivalence of discrete graphs and
  graph-like manifolds}
\label{ssec:q-u-e.graph.mfd}

The choice of the partition of unity $\Psi=(\psi_v)_{v \in V}$ is
almost obvious now: Let
\begin{equation*}
  \psi_v(x)=1 \quadtext{if $x \in \cXv$ \quad and}
  \psi_v(x)=\frac 1{\ell_e}\, t,
  \quad\text{where $x=(t,y)$, $t \in \Me$, $y \in Y_e$}
\end{equation*}
are coordinates on $\Xe$.  We assume here that $v=\bd_+e$ is the
terminal vertex, i.e., $v$ corresponds to $\ell_e \in \Me=[0,\ell_e]$.
Let $\psi_v(x)=0$ for any other point $x$ not in $\cXv$ and $\Xe$, $e
\in E_v$.

In particular, $\psi_v$ is Lipschitz continuous on $\Mfd$.  As $\Xe$ is
a product, $\psi_v$ is harmonic on $\Xe$ (affine linear in the
longitudinal direction times a constant function in the transversal
direction).  Now, it is obvious, that $\Xe$ is the edge neighbourhood
of $e \in E$ also in the sense of \Def{graph.emb.in.met.space}, and $\cXv$ is
the (core) vertex neighbourhood of $v \in V$ in $\Mfd$.  Moreover, the
(enlarged) vertex neighbourhood $\Xv$ is
\begin{equation*}
  \Xv = \cXv \cup \bigcup_{e \in E_v} \Xe.
\end{equation*}
The vertex measure here is
\begin{equation}
  \label{eq:vx.meas.mfd}
  \nu(v) = \int_\Mfd \psi_v \dd x
  = \sum_{e \in E_v} \vol Y_e \int_0^{\ell_e} \frac1{\ell_e}t \dd t
  + \vol \cXv
  =\frac 12 \sum_{e \in E_v} \vol \Xe + \vol \cXv
\end{equation}
and hence $\vol \Xv/2 \le \nu(v) \le \vol \Xv$.  Here, $\vol = \nu$ is
the $d$-dimensional volume on $\Mfd$; with the exception that $\vol
Y_e$ denotes the $(d-1)$-dimensional volume of $Y_e$.
For the decomposition of $\nu(v)$ in~\eqref{eq:def.nu.decomp} we choose
\begin{equation}
  \label{eq:def.nu0}
  \nu_0(v)
  := \frac 12\sum_{e \in E_v} \vol \Xe
  = \frac 12\sum_{e \in E_v} \ell_e (\vol Y_e)
  \qquadtext{and}
  \check \nu(v) := \vol \cXv
\end{equation}
for all $v \in V$.  The following definition assures that a graph-like
manifold is well-adopted to a given weighted graph:
\begin{definition}
  \label{def:mfd.comp}
  \indent
  \begin{subequations}
    \begin{enumerate}
    \item
      \label{mfd.comp.a}
      We say that a graph-like manifold $\Mfd$ and a weighted
      discrete graph $(G,\mu,\gamma)$ are \emph{compatible}, if the
      underlying discrete graphs are the same and if there exists
      $c>0$ and $\tau>0$ such that the edge length function $\ell$,
      the weights $\mu$ and $\gamma$ and the transversal volumes
      $(\vol Y_e)_e$ fulfil
      \begin{equation}
        \label{eq:comp.mfd}
        \frac 1{2\mu(v)}\sum_{e \in E_v} \ell_e (\vol Y_e)
        = \frac 1{c^2}
      \end{equation}
      for all $v \in V$ and
      \begin{equation}
        \label{eq:ell.tau.y}
        \frac{\gamma_e \ell_e}{\vol Y_e}
        = c^2 \tau
      \end{equation}
      for all $e \in E$.
    \item
      \label{mfd.comp.b}
      We say that $\Mfd$ has \emph{uniformly small (core) vertex
        neighbourhoods}, if
      \begin{equation}
        \label{eq:def.alpha.mfd}
        \alpha_\infty
        =\sup_{v \in V} \alpha(v)
        \le \frac 12
        \quadtext{and}
        \alpha_0
        =\inf_{v \in V} \alpha(v)
        >0,
      \end{equation}
      where
      \begin{equation}
        \label{eq:def.alpha.mfd.v}
        \alpha(v)=\frac{2\vol \cXv}{\sum_{e \in E_v}\vol \Xe},
      \end{equation}
      and if
      \begin{equation}
        \label{eq:lambda2.cxv}
        \check \lambda_2 := \inf_{v \in V} \lambda_2(\cXv) >0.
      \end{equation}
    \item
      \label{mfd.comp.c}
      We say that $\Mfd$ has \emph{uniform transversal volume} if
      \begin{equation}
        \label{eq:vol.y.unif}
        0<\vol_0
        := \inf_{e \in E} \vol Y_e 
        \le \vol_\infty 
        := \sup_{e \in E} \vol Y_e 
        < \infty.
      \end{equation}
    \item
      \label{mfd.comp.d}
      We say that a graph-like manifold $\Mfd$ and a weighted discrete
      graph $(G,\mu,\gamma)$ are \emph{uniformly compatible}, if they
      are compatible and if $\Mfd$ has uniformly small vertex
      neighbourhoods and uniform transversal volume.
    \end{enumerate}
  \end{subequations}
\end{definition}

\begin{remark}
  \indent
  \begin{enumerate}
  \item Condition~\eqref{eq:comp.mfd} is the compatibility of the
    vertex weights $\nu_0$ and $\mu$. 
  \item Note that if we consider an embedded metric graph with length
    function $\ell$ together with a small $\eps$-neighbourhood as
    graph-like manifold $\Mfd$ as in \Ex{emb}, then the edge length
    function of $\Mfd$ is $(1-2\eps)\ell$; the common factor
    $(1-2\eps)$ does not destroy the compatibility in the sense of
    \Defenum{mfd.comp}{mfd.comp.a}, it just changes the factors $c$
    and $\tau$ slightly.

  \item Equation~\eqref{eq:forms.eq} forces that $\gamma_e$ is given
    by~\eqref{eq:ell.tau.y}.

  \item As for metric graphs, there is still some freedom in the
    choice of the parameters $\ell_e$, $\vol Y_e$, $\tau$ and $c$, as
    they have to fulfil only two equations in~\eqref{eq:comp.mfd}
    and~\eqref{eq:ell.tau.y}. Nevertheless, from~\eqref{eq:comp.mfd}
    and~\eqref{eq:ell.tau.y} we conclude
    \begin{equation}
      \label{eq:c.tau.mfd}
      c=\Bigl( \frac 1{2\mu(v)}\sum_{e \in E_v} \ell_e \vol Y_e \Bigr)^{-1/2}
      \qquadtext{and}
      \tau
      = \frac 1{2\mu(v)}\sum_{e \in E_v} \ell_e^2\gamma_e.
    \end{equation}

  \item As for metric graphs, the compatibility
    conditions~\eqref{eq:comp.mfd} and~\eqref{eq:ell.tau.y} give a
    restriction on the vertex and edge weights of the weighted graph
    and the transversal volume $\vol Y_e$, namely that
  \begin{equation}
    \label{eq:cond.weighted.graph.mfd}
    \frac1{2\mu(v)} \sum_{e \in E_v} \frac {(\vol Y_e)^2}{\gamma_e}
    = \frac1{c^4 \tau}
  \end{equation}
  is independent of $v \in V$.

\item Note that $\alpha(v)=\check \nu(v)/\nu_0(v)>0$ and
  $\alpha_\infty=\sup_v\alpha(v)$ as in~\eqref{eq:def.alpha}.  The
  upper bound $\alpha_\infty$ is needed for $J'$ being close to the
  adjoint of $J$, see \Prp{q-u.b}, while the lower bound $\alpha_0$ is
  needed in \Lem{diff.av.mfd}.
  \end{enumerate}
\end{remark}

As ``evaluation'' $\Gamma_v u$, we set
\begin{equation*}
  \Gamma_v u := \dashint_{\cXv} u \dd \nu.
\end{equation*}
Note that the pointwise evaluation $u \mapsto u(v)$ does not make
sense here, as $\Mfd$ is at least $2$-dimensional and hence evaluation
on points is not defined on $\Sob \Mfd$.

Let us first check the condition in~\eqref{eq:met.q.u.c}:
\begin{lemma}
  \label{lem:diff.av.mfd}
  We have
  \begin{equation}
    \label{eq:mfd.q.u.c}
    \nu_0(v) \Bigabssqr{\Gamma_v u - \dashint_X u \dd \nu_v}
    \le \wtdeltaC(v)^2 \energy_{\Xv}(u)
  \end{equation}
  with
  \begin{equation*}
    \wtdeltaC(v)^2 
    = \frac 1 {\alpha(v)}
    \max_{e \in E_v}
    \Bigl\{
      \frac 9 {2\lambda_2(\Xv,\psi_v)},
      4\kappa \ell_e^2
    \Bigr\}
  \end{equation*}
  (see~\eqref{eq:def.nu0} for the definition of $\nu_0(v)$
  and~\eqref{eq:def.alpha.mfd.v} for the definition of $\alpha(v)$).
\end{lemma}
\begin{proof}
  Summing~\eqref{eq:ew.2c} over $e \in E_v$, we obtain
  \begin{equation*}
    \normsqr[\Lsqr \cXv] u
    \le \max \Bigl\{\frac {9\kappa} 2, 1 \Bigr\}
       \normsqr[\Lsqr{\Xv, \dd \nu_v}] u
    + 4\kappa \max_{e \in E_v} \ell_e^2 \cdot \energy_{\Xv}(u)
  \end{equation*}
  (note that the contribution of $u$ on $\cXv \setminus \bigcup_{e \in
    E_v} \Xve$ also adds on the right hand side, hence the $1$ on the
  right hand side); the maximum can be estimated by $9/2$ as $0<\kappa
  \le 1$.  Now we plug in $u-\avint_{\Xv} u \dd \nu_v$ instead of $u$
  in the last inequality and obtain
  \begin{equation*}
    \bignormsqr[\Lsqr \cXv] {u -\avint_{\Xv} u \dd \nu_v}
    \le  \max_{e \in E_v}
    \Bigl\{
    \frac 9{2\lambda_2(\Xv,\psi_v)},
      4\kappa \ell_e^2
    \Bigr\}
    \energy_{\Xv}(u)
  \end{equation*}
  using \Lem{2nd.ev}.  Now we have
  \begin{equation*}
    \Bigabssqr{\Gamma_v u - \avint_X u \dd \nu_v}
    = \Bigabssqr{\dashint_{\cXv}(u -\avint_X u \dd \nu_v) \dd \nu}
    \leCS \frac 1 {\vol \cXv}
    \normsqr[\Lsqr \cXv] {u -\avint_X u \dd \nu_v}
  \end{equation*}
  hence the result follows with $\wtdeltaC(v)^2$ as above using
  $\alpha(v)=\vol \cXv/\nu_0(v)$.
\end{proof}

For~\eqref{eq:met.q.u.d}--\eqref{eq:met.q.u.d'} we set
\begin{equation*}
  \Gamma_{v,e} u := \dashint_{\bd_e \cXv} u,
\end{equation*}
where $\bd_e \cXv \cong Y_e$ is the boundary component of $\cXv$ at
the edge neighbourhood $\Xe$.
\begin{lemma}
  \label{lem:av.mfd}
  Assume that~\eqref{eq:ell.tau.y} holds
  then
  \eqref{eq:met.q.u.d}--\eqref{eq:met.q.u.d'} are fulfilled with
  \begin{equation*}
    \wtdeltaD(v)^2
    =\tau \max_{e \in E_v}
    \Bigl\{ 
      \kappa + \frac 2 {\kappa \ell_e^2 \lambda_2(\cXv)} 
    \Bigr\}.
  \end{equation*}
\end{lemma}
\begin{proof}
  The proof of~\eqref{eq:met.q.u.d} is again an application of the
  fundamental theorem of calculus: Assume that the vertex $v$
  corresponds to the endpoint $t=\ell_e$ of each adjacent $\Xe$, $e
  \in E_v$.  Then $\psi_{v,e}(t,y)=t/\ell_e$ for $x=(t,y) \in \Xe =
  [0,\ell_e] \times Y_e$ with derivative $1/\ell_e$ and we have
  \begin{equation}
    \label{eq:forms.eq}
    c^2 \tau \energy_\Mfd(u,\psi_v)
    = \sum_{e \in E_v} \frac {c^2 \tau}{\ell_e} 
    \int_0^{\ell_e}\int_{Y_e} u_e'(t,y) \dd t \dd y
    = \sum_{e \in E_v} \gamma_e \bigl(\Gamma_{v,e} u - \Gamma_{v_e,e} u \bigr)
  \end{equation}
  using~\eqref{eq:ell.tau.y} for the last equality.
  For~\eqref{eq:met.q.u.d'} we apply~\eqref{eq:ew.1a} and obtain
  \begin{align*}
    \sum_{e \in E_v} \gamma_e \abssqr{\Gamma_{v,e} u - \Gamma_v u}
    &\le \max_{e \in E_v} 
    \frac{\gamma_e}{\vol Y_e} 
    \Bigl(\kappa \ell_e + \frac 2{\kappa \ell_e \lambda_2(\cXv)}
    \Bigr) \normsqr[\Lsqr \cXv] u \\
    &= c^2\tau \max_{e \in E_v} 
    \Bigl\{\kappa + \frac 2 {\kappa \ell_e^2 \lambda_2(\cXv)}
    \Bigr\} \normsqr[\Lsqr \cXv] {du}
  \end{align*}
  using again~\eqref{eq:ell.tau.y} for the last equality.
\end{proof}

\begin{lemma}
  \label{lem:2nd.ev.mfd}
  We have 
  \begin{equation*}
    \lambda_2(\Xv)
    \le \lambda_2(\Xv,\psi_v)
  \end{equation*}
\end{lemma}
\begin{proof}
  The proof is almost the same as the proof of \Lem{2nd.ev.mg}.  Note
  that the function $\psi_v$ is harmonic on $\Xe$, namely affine
  linear and constant in transversal direction $Y_e$, hence the second
  eigenfunction is again an Airy function on $\Xe$ (in longitudinal
  direction), hence continuous and also in the unweighted Hilbert
  space $\Lsqr \Xv$; the result then follows from \Lem{2nd.ev.alt}.
\end{proof}
We will assume in this section that we have a lower bound on the
unweighted eigenvalue of the form
\begin{equation}
  \label{eq:2nd.ev.mfd}
  \frac {\lambda_{2,0}}{\ell_\infty^2} \le \lambda_2(\Xv)
\end{equation}
for some constant $\lambda_{2,0} \in (0,2]$ independent of $v \in V$,
where $\ell_\infty:=\sup_e \ell_e$.  Later, in our application with
shrinking graph-like manifolds in \Cor{mfd.q-u-e} or with graph-like
manifolds approximating fractals in \Subsec{q-u-e.frac.mfd}, we will
check that we can even choose $\lambda_{2,0}=1$ (see
\Prp{2nd.ev.mfd}).

We are now prepared to prove the main result of this subsection:
\begin{theorem}
  \label{thm:mfd.q-u-e}
  Assume that $(G,\mu,\gamma)$ is a weighted graph with
  \begin{equation*}
    \mu_\infty:=\sup_{v \in V} \mu(v) < \infty,
    \qquadtext{and}
    0<\gamma_0:=\inf_{e \in E} \gamma_e
    \le \gamma_\infty:=\sup_{e \in E} \gamma_e
    <\infty.
  \end{equation*}
  Assume in addition that $X$ is a graph-like manifold uniformly
  compatible with $(G,\mu,\gamma)$, i.e., the weights $\mu(v)$,
  $\gamma_e$, the edge lengths $\ell_e$, the transversal volumes $\vol
  Y_e$ and the core vertex neighbourhoods $\cXv$
  fulfil~\eqref{eq:comp.mfd}--\eqref{eq:def.alpha.mfd}
  and~\eqref{eq:lambda2.cxv}--\eqref{eq:vol.y.unif}.  Moreover, we
  assume that~\eqref{eq:2nd.ev.mfd} holds.  Then $\ell_0:=\inf_{e \in
    E} \ell_e>0$, and the graph energy form $\energy$ associated with
  the weighted discrete graph $(G,\mu,\gamma)$ and the rescaled
  graph-like manifold energy form $\wt \energy=\tau \energy_\Mfd$ are
  $\delta$-quasi-unitarily equivalent with
  \begin{equation*}
    \delta^2:=
    \max \Bigl\{
    2\alpha_\infty,\;
    \frac{18}{\lambda_{2,0} \alpha_0} \frac{\gamma_\infty}{\gamma_0} 
        \Bigl(\frac{\vol_\infty}{\vol_0}\Bigr)^2
        \cdot \frac{\mu_\infty}{\gamma_0},\;
    \kappa+\frac2{\kappa \ell_0^2 \check \lambda_2}
    \Bigr\}.
  \end{equation*}
\end{theorem}
\begin{remark*}
  The error terms have the following meaning:
  \begin{enumerate}
  \item The first containing $\alpha_\infty$ ensures that the core
    vertex manifold volume is small compared with the edge
    neighbourhood manifold volume.

  \item The second term is similarly as for metric graphs, having
    again the factor $\mu_\infty/\gamma_0$ inside which becomes small
    if the relative weight is large (similarly as in \Cor{mg.q-u-e}).
    Here we have an additional term $1/\alpha_0$ which in general is
    large, making the error a bit worse than in the metric graph case.

  \item The last term can best be understood in the setting of the
    next corollary or \Subsec{q-u-e.frac.mfd}; namely, if we introduce
    $\eps$ as length scale parameter, i.e., if $\eps \cXv$ denotes the
    manifold $\cXv$ with metric $\eps^2 g_{\cXv}$, then
    $\lambda_2(\eps\cXv)=\eps^{-2}\lambda_2(\cXv)$, and $\kappa
    \ell_e$ is of order $\eps$, hence the entire last term is of order
    $\eps/\ell_0$.
  \end{enumerate}
\end{remark*}
\begin{proof}[Proof of \Thm{mfd.q-u-e}]
  Note first that $(G,\mu,\gamma)$ is uniformly embedded into
  $(X,\nu,\energy_X)$ (see \Def{graph.emb.in.met.space}) by the
  assumptions of the theorem: in particular, $\lambda_2(X,\psi)\ge
  \lambda_2(X) \ge \lambda_{2,0}/\ell_\infty^2$
  by~\eqref{eq:2nd.ev.mfd} and \Lem{2nd.ev.mfd}.  The remaining
  assumptions of \Thm{q-u-e} are fulfilled by
  \Lems{diff.av.mfd}{av.mfd}.  We now estimate the terms in the
  definition of $\delta$ in~\eqref{eq:q-u-e.delta} in our model here:
  The first term is $2\alpha_\infty$ as above; the second one is
  $2\mu_\infty/\gamma_0$ contained already in the second term above;
  for the third one we estimate
  \begin{equation*}
    \frac 2 {\tau \lambda_2}
    \le \frac{2\ell_\infty^2}{\lambda_{2,0} \tau}
    \le c^4\tau \cdot \frac{2\vol_\infty^2}{\lambda_{2,0} \gamma_0^2}
    \le \frac{2\mu_\infty \gamma_\infty}{\vol_0^2} 
       \cdot \frac{2\vol_\infty^2}{\lambda_{2,0} \gamma_0^2}
    = \frac4{\lambda_{2,0}} \frac{\gamma_\infty}{\gamma_0} 
        \Bigl(\frac{\vol_\infty}{\vol_0}\Bigr)^2
        \cdot \frac{\mu_\infty}{\gamma_0},
  \end{equation*}
  where we have shown the first inequality already above, and where we
  used~\eqref{eq:ell.tau.y}, \eqref{eq:vol.y.unif} and $\gamma_e \ge
  \gamma_0$ for the second estimate
  and~\eqref{eq:cond.weighted.graph.mfd} and~\eqref{eq:vol.y.unif} for
  the last.  Since $\alpha_0 \le 1/2$, this term is already contained in
  the second term above.

  The fourth term in~\eqref{eq:q-u-e.delta} (the one with
  $2\sup_v\deltaC(v)^2/\tau$) contains one term of the form
  $2/(\tau\lambda_2) \cdot 9/(2\alpha_0)$ by \Lem{diff.av.mfd}
  and~\eqref{eq:2nd.ev.bdd}.  The other term from \Lem{diff.av.mfd} is
  of the form $8\kappa/\alpha_0 \cdot \ell_\infty^2/\tau$ and can be
  treated as above, hence we have
  \begin{equation*}
    \frac{8\kappa}{\alpha_0} \cdot \frac{\ell_\infty^2}\tau
    \le \frac {16\kappa}{\alpha_0} \frac{\gamma_\infty}{\gamma_0} 
        \Bigl(\frac{\vol_\infty}{\vol_0}\Bigr)^2
        \cdot \frac{\mu_\infty}{\gamma_0}.
  \end{equation*}
  As $16\kappa \le 18/\lambda_{2,0}$ ($\kappa \le 1$), this term is
  already in the above list for $\delta$.  Finally, the last term in
  \Thm{q-u-e} (the one with $4\sup_v \deltaD(v)^2/\tau$) gives the last term
  in the list for $\delta$ by \Lem{av.mfd}.
\end{proof}

We assume now that $Y_\edeps=\eps Y_e$ with $\vol Y_e=1$ and that
$\check X_\vxeps=\eps\cXv$.  Here, $rX$ is the Riemannian manifold
$(X,r^2g)$ if $X$ is the Riemannian manifold $(X,g)$; the factor $r$
is hence a change of \emph{length scale}.  In particular, the
$\eps$-scaling of $\eps \cXv$ implies that the length of the collar
neighbourhood of \Defenum{gl-mfd}{gl-mfd.c} $\kappa\ell_e$ is of order
$\eps$.  We denote the resulting $\eps$-depending graph-like manifold
by $X_\eps$:
\begin{corollary}
  \label{cor:mfd.q-u-e}
  Assume that $(G,\mu,\gamma)$ is a
  $(d_\infty,\muVar,\gammaVar)$-uniform weighted graph.  Assume in
  addition that $X$ is a corresponding (unscaled) graph-like manifold
  with (unscaled) transversal manifold $Y_e$ being isometric with a
  fixed one $Y_0$ with volume $\vol Y_0=1$.  Moreover, we assume that
  there exist $c_1>0$ and $\tau>0$ with
  \begin{equation}
    \label{eq:ell.tau.mfd.eps}
    \frac1{2\mu(v)} \sum_{e \in E_v} \ell_e
    = \frac 1{c_1^2}
    \qquadtext{and}
    \gamma_e \ell_e
    = c_1^2 \tau
 \end{equation}
 for all $v \in V$ and $e \in E$.  Finally, we assume that there are
 constants $\check \vol_0$, $\check \vol_\infty$ and $\check
 \lambda_2$ such that
 \begin{equation}
   \label{eq:cxv}
   0< \check \vol_0 
   \le \vol \cXv 
   \le \check \vol_\infty 
   < \infty
   \quadtext{and}
   \lambda_2(\cXv) \ge \check \lambda_2>0.
 \end{equation}
 Then $(G,\mu,\gamma)$ and $X_\eps$ are uniformly compatible (see
 \Def{mfd.comp}) and $\ell_0:=\inf_e \ell_e>0$ and $\ell_\infty
 :=\sup_e \ell_e<\infty$.  Moreover, there exists a constant $C_0>0$
 such that the graph energy form $\energy$ associated with the
 weighted discrete graph $(G,\mu,\gamma)$ and the energy form on the
 scaled graph-like manifold $\wt \energy=\tau \energy_{X_\eps}$ are
 $\delta$-quasi-unitarily equivalent with
  \begin{equation*}
    \delta^2:=
    \max \Bigl\{
    \Err\Bigl(\frac \eps {\ell_0} \Bigr),\;
    \Err\Bigl(\frac {\ell_0}\eps \cdot \frac{\mu_\infty}{\gamma_0}\Bigr)
    \Bigr\}
  \end{equation*}
  for all $0<\eps \le C_0^{-2} \ell_0$, where the errors depend only
  on the above-mentioned constants $d_\infty$, $\muVar$, $\gammaVar$,
  $\ellVar:=\ell_\infty/\ell_0$, $\check \vol_0$, $\check \vol_\infty$
  and $\check \lambda_2$.
\end{corollary}
\begin{proof}
  Note first that $\ell_e=c_1^2\tau/\gamma_e$, hence $\ell_0
  =c_1^2\tau/\gamma_\infty>0$ and
  $\ell_\infty=c_1^2\tau/\gamma_0<\infty$.  Denote objects associated
  with the $\eps$-depending manifold $X_\eps$ also with a subscript
  $(\cdot)_\eps$.  We apply simple scaling arguments such as
  $\vol(\eps \cXv)=\eps^d\vol \cXv$ etc.

  Let us now check that $(G,\mu,\gamma)$ and $X_\eps$ are uniformly
  compatible: From~\eqref{eq:ell.tau.mfd.eps} we
  conclude~\eqref{eq:comp.mfd}--\eqref{eq:ell.tau.y} with
  $c_\eps=\eps^{-(d-1)/2}c_1$ and $\tau_\eps=\tau$.  Moreover,
  $\alpha_\eps(v)=\eps \alpha(v)$, hence $\alpha_{\eps,\infty}:=\sup_v
  \alpha_\eps(v)=\eps\alpha_\infty\le 1/2$ once $\eps\le
  1/(2\alpha_\infty)$; moreover, we have $\alpha_\infty \le \check
  \vol_\infty/\ell_0<\infty$ by \eqref{eq:cxv}; in particular, we need
  $\eps/\ell_0 \le 1/(2\check \vol_\infty)$.  Similarly,
  $\alpha_{\eps,0}=\eps \alpha_0$ with $\alpha_0=\inf_v \alpha(v) \ge
  \check \vol_0/(d_\infty \ell_\infty)>0$,
  hence~\eqref{eq:def.alpha.mfd} is fulfilled.
  Conditions~\eqref{eq:lambda2.cxv}--\eqref{eq:vol.y.unif} follow
  directly from the assumptions.
 
  We now check the individual terms in the definition of
  $\delta=\delta_\eps$ in \Thm{mfd.q-u-e}: we have
  $\alpha_{\eps,\infty} =\eps \alpha_\infty \le \check \vol_\infty
  (\eps/\ell_0)$, hence the first term is of order
  $\Err(\eps/\ell_0)$.  For the second term in \Thm{mfd.q-u-e} we need
  the estimate $\alpha_{\eps,0}=\eps \alpha_0 \ge \check
  \vol_0/(d_\infty \ellVar)\cdot (\eps/\ell_0)$; moreover, for the
  validity of~\eqref{eq:2nd.ev.mfd} with $\lambda_{2,0}=1$, namely the
  existence of $C_v>0$ such that $\lambda_2(X_\vxeps) \ge
  1/\ell_\infty^2$ for all $0<\eps \le \eps_0=C_v^{-2}\ell_0$, we
  refer to \Prp{2nd.ev.mfd}.  The constant $C_v$ depends on lower
  estimates on $\lambda_2(Y_e)=\lambda_2(Y_0)$ and $\lambda_2(\cXv)
  \ge \check\lambda_2>0$, and on an upper estimate of $\vol
  \cXv/\sum_{e \in E_v} \vol Y_e=\vol \cXv/\deg v \le \check
  \vol_\infty$, hence $C_0:=\sup_v C_v<\infty$.

  The estimate on the last term in the definition of $\delta$ in
  \Thm{mfd.q-u-e} follows now from the scaling $\lambda_2(\eps
  \cXv)=\eps^{-2}\lambda_2(\cXv)$; moreover, the length of the collar
  neighbourhood $\kappa\ell_e$ is of order $\eps$ in the sense that
  \begin{equation*}
    0<\eps = \kappa \ell_0
    \le \kappa \ell_e 
    \le \kappa \ell_\infty
    = \kappa \ellVar \ell_0 = \ellVar \eps
  \end{equation*}
  for all $e \in E$.  In particular we have the estimate
  \begin{equation*}
    \kappa+\frac2{\kappa \inf_e \ell_e^2 \inf_v \lambda_2(\cXv)}
    \le \frac \eps{\ell_0} + \frac 2 {\eps \ell_0 \inf_v \lambda_2(\eps\cXv)}
    \le \Bigl(1+ \frac 2 {\check \lambda_2}\Bigr)
         \frac \eps{\ell_0}.
         \qedhere
  \end{equation*}
\end{proof}

%
\section{Convergence of energy forms on fractals and graph-like
  spaces}
\label{sec:fractals.met}
%

\subsection{Symmetric post-critically finite fractals}
\label{ssec:fractals}

For details on post-critically finite fractals we refer to our first
article~\cite{post-simmer:pre17a}, and, of course, to the
monographs~\cite{strichartz:06,kigami:01}.  We consider here only,
what we call \emph{symmetric} fractals, which basically means that the
corresponding quantity is \emph{independent} of the index $j$ of the
iterated function system.  The symmetry might appear from an
underlying dihedral symmetry of $K$, but what matters for us is mostly
the fact that we have explicit formulas for various quantities.  We
need the symmetry assumption mainly for the compatibility
assumptions~\eqref{eq:mu.nu.comp.mg} and~\eqref{eq:comp.mfd}.

A \emph{symmetric fractal} (in our setting here) is a compact subset
of $\R^d$ which is invariant under an iterated function system
$F=(F_j)_{j=1,\dots,N}$, i.e., for which we have
$K=F(K):=\bigcup_{j=1}^N F_j(K)$.  Each member $\map{F_j}{\R^d}{\R^d}$
is supposed to be a $\theta$-similitude, i.e., if there is $\theta \in
(0,1)$ such that
\begin{equation}
  \label{eq:frac.sym.a}
  \abs{F_j(x)-F_j(y)}
  =\theta \abs{x-y}
  \qquadtext{for all $x,y \in \R^d$.}
\end{equation}
Let $V_0$ be a non-empty subset of the fixed points of the $F_j$'s,
called \emph{boundary} of the fractal, and set $N_0 := \card {V_0}$
(then $N_0 \le N$).  We assume that the fractal is
\emph{post-critically finite (\pcf)}, i.e., that $F_j(K) \cap F_{j'}(K)
\subset F_j(V_0) \cap F_{j'}(V_0)$ for all $1 \le j < j' \le N$.  For
such fractals, there is a recursively defined sequence of simple
graphs $G_m=(V_m,E_m)$, starting with the complete graph $G_0$ over
the set of boundary points $V_0$, and such that $V_{m+1} := F(V_m)$,
and $e=\{x,y\} \in E_{m+1}$ if there exists an edge $e'=\{x',y'\} \in
E_m$ and $j \in \{1,\dots, N\}$ such that $F_j(x')=x$ and $F_j(y')=y$
(for short, we write $F_j(e')=e$).

On $G_m$, we assume that there is a (discrete) energy form
\begin{equation}
  \label{eq:disc.graph.energy}
  \energy_m(f)
  =\sum_{\{x,y\} \in E_m} \gamma_{m,\{x,y\}} \abssqr{f(x)-f(y)}.
\end{equation}
We call the sequence $(\energy_m)_m$ \emph{self-similar} and
\emph{symmetric} if there exists $r \in (0,1)$ with
  $\energy_{m+1}(f)
  =\sum_{j=1}^N \frac 1 r \energy_m(f \circ F_j)$
  for all $\map f {V_m} \C$.  We call $r$ the \emph{energy
    renormalisation parameter} (of the self-similarity).  We call the
  sequence $(\energy_m)_m$ \emph{compatible} if the vertex sets of
  $G_m$ are nested (i.e., $V_m \subset V_{m+1}$) and if
\begin{equation*}
  \energy_{G_m}(\phi)=
  \min\bigset{\energy_{G_{m+1}}(f)}
  {\map f {V_{m+1}} \C, \; f \restr {V_m}=\phi}
\end{equation*}
for all $\map \phi {V_m} \C$.  We assume here that $(\energy_m)_m$ is
self-similar, symmetric and compatible.  In this case, the discrete
edge weights $\map {\gamma_m}{E_m} {(0,\infty)}$ are given by
$\gamma_{m,e}=r^{-m} \gamma_{0,e_0}$ if there is a word $w \in
W_m:=\{1,\dots,N\}^m$ of length $m$ such that $F_w(e_0)=e$, where
$F_w:=F_{w_1} \circ \dots \circ F_{w_m}$ if $w=(w_1,\dots,w_m)$.

For a symmetric, self-similar and compatible sequence $(\energy_m)_m$,
there exists a self-similar and symmetric energy form $\energy_K$ on
$K$, i.e., a closed non-negative quadratic form with domain $\dom
\energy_K \subset \Cont K$ such that 
\begin{equation}
  \label{eq:frac.sym.b}
  \energy_K(u)
  =\sum_{j=1}^N \frac 1 r \energy_K(u \circ F_j)
\end{equation}
for all continuous $\map u {V_*:=\bigcup_m
  V_m} \C$.  For this energy form, a given $m \in N_0$ and $v \in
V_m$, there is a unique function $\map {\psi_{m,v}}K {[0,1]}$ with
$\psi_{m,v}(v')=1$ if $v=v'$ and $0$ if $v \ne v'$.  Moreover,
$\psi_{m,v} \in \dom \energy_K$ and $\energy_K(\psi_{m,v})$ is the
minimal value among all $\energy_K(u)$ with $ u \restr
{V_m}=\psi_{m,v} \restr {V_m}$.  The function $\psi_{m,v}$ is called
\emph{$m$-harmonic}.  Note that $(\psi_{m,v})_{v \in V_m}$ is a
partition of unity on $K$.

Moreover, we assume that we have a \emph{self-similar symmetric
  measure}, i.e., a finite measure $\mu$ on $K$ such that
\begin{equation}
  \label{eq:frac.sym.c}  
  \mu(F_w(K))
  =N^{-m} \mu(K)
  \qquadtext{for any word}
  w \in W_m:=\{1,\dots,N\}^m.
\end{equation}
For simplicity, we assume that $\mu(K)=1$.  We define
\begin{equation*}
  \mu_m(v):= \int_K \psi_{m,v} \dd \mu.
\end{equation*}
We call $(G_m,\mu_m,\gamma_m)_m$ a \emph{sequence of approximating
  weighted graphs} for the \pcf fractal $K$.

Let us summarise the above discussion and introduce the symmetry of
the boundary:
\begin{definition}
  \label{def:frac.sym}
  Let $K$ be a \pcf fractal given by some iterated function system
  $F=(F_j)_{j=1,\dots,N}$.
  \begin{enumerate}
  \item 
    \label{frac.sym.a}
    We say that $K$ is \emph{symmetric} if the similitude factor
    $\theta$ is the same for all functions in the iterated function
    system, i.e., if~\eqref{eq:frac.sym.a} holds.

  \item 
    \label{frac.sym.b}
    We say that a self-similar energy form $\energy_K$ on $K$ is
    \emph{symmetric} or \emph{homogeneous} if the energy
    renormalisation parameter $r \in (0,1)$ is the same for all $j \in
    \{1,\dots,N\}$, i.e., if~\eqref{eq:frac.sym.b} holds.

  \item 
    \label{frac.sym.c}
    We say that a self-similar measure $\mu$ on $K$ is
    \emph{symmetric} if the measure self-similar factor is the same
    for all $j \in \{1,\dots,N\}$, i.e., if~\eqref{eq:frac.sym.c}
    holds.

  \item 
    \label{frac.sym.d}
    We say that the boundary $V_0$ of a fractal $K$ is
    \emph{symmetric} if $\mu_0$ gives the same mass to all points,
    i.e., if $\mu_0(v_0)=1/N_0$ for all $v_0 \in V_0$ and if there is
    $C_0>0$ such that
    \begin{equation}
      \label{eq:frac.sym.d}
      \sum_{e_0 \in E_{v_0}(G_0)} \frac 1{\gamma_{0,e_0}} = C_0
    \end{equation}
    holds for all $v_0 \in V_0$.
  \end{enumerate}
  If all four conditions hold, we also say that
  $(K,\energy_K,\mu,V_0)$ is \emph{symmetric}.  We also call
  $(G_m,\mu_m,\gamma_m)$ the \emph{$m$-th approximation} of
  $(K,\energy_K,\mu,V_0)$ by a finite weighted graph.
\end{definition}
If $(K,\energy_K,\mu,V_0)$ is symmetric, then we have
\begin{align}
  \nonumber
  \mu_m(v)
  = \int_K \psi_{m,v} \dd \mu
  &= \sum_{w \in W_{m,v}} \int_{F_w(K)} \psi_{m,v} \dd \mu\\
  \label{eq:mu.m.v}
  &= \sum_{w \in W_{m,v}} \frac 1{N^m} \int_K \psi_{m,F_w^{-1}v} \dd \mu
  = \sum_{w \in W_{m,v}} \frac 1{N_0N^m}
  = \frac{\card{W_{m,v}}}{N_0N^m}
\end{align}
as $\supp \psi_{m,v} \subset \bigcup_{w \in W_{m,v}} F_w(K)$ (second
equality) and as $\mu$ is self-similar and symmetric, as well as
$\energy_K$ is symmetric (third equality).  The last equality uses
the symmetry of the boundary.

\begin{lemma}
  \label{lem:unif.weighed.graph}
  Let $(G_m,\mu_m,\gamma_m)$ be the sequence of weighted graphs
  associated with a symmetric $(K,\energy_K,\mu,V_0)$ then
  $(G_m,\mu_m,\gamma_m)$ is $(N_1(N_0-1),N_1,C_2/C_1)$-uniform, where
  \begin{equation}
    \label{eq:def.N1.C1}
    N_1:=\sup_{m \in \N_0, v \in V_m} \card{W_{m,v}},
    \qquad
    C_1:=\min_{e_0 \in E_0} \gamma_{0,e_0}
    \quadtext{and}
    C_2:=\max_{e_0 \in E_0} \gamma_{0,e_0}.
  \end{equation}
  In particular, all constants are independent of $m \in \N_0$.
\end{lemma}
\begin{proof}
  Each cell has degree maximal to $N_0-1$ (as we start with the
  complete graph in generation $0$ on $N_0$ vertices), and the maximal
  number of cells intersecting in one vertex is given by $N_1$.
  Moreover, $\mu_m(v)=\card{W_{m,v}}/(N_0N^m)$ hence
  $\mu_{m,\infty}/\mu_{m,0}=N_1$.  Finally, $\gamma_{m,e}=r^{-m}
  \gamma_{0,e_0}$ if $e=F_w (e_0)$ for some word $w \in W_m$, hence
  $\gamma_{m,\infty}/\gamma_{m,0}=C_2/C_1$.
\end{proof}

\subsection{Quasi-unitary equivalence of fractals and metric graphs}
\label{ssec:q-u-e.frac.mg}

We will now apply \Thm{mg.q-u-e} to the discrete graphs $G=G_m$ with
vertex weights $\mu=\mu_m$ and edge weights $\gamma=\gamma_m$ from the
fractal approximation, and a corresponding metric graph
$\MetGr=\MetGr_m$.  We fix the edge lengths in a way that $\MetGr_m$
and $(G_m,\mu_m,\gamma_m)$ are compatible in the sense of
\Def{mg.comp}:
\begin{lemma}
  \label{lem:met.comp}
  Assume that $(K,\energy_K,\mu,V_0)$ is symmetric, that
  $(G_m,\mu_m,\gamma_m)$ is a corresponding member of an approximating
  sequence of weighted graphs and that $\MetGr_m$ is a metric graph
  according to $G_m$ with edge length given by
  \begin{equation}
    \label{eq:met.gr.len}
    \ell_{m,e}=\frac{\ell_{0,0}}{\gamma_{0,e_0}} \cdot \Lambda^m
  \end{equation}
  for some $\Lambda \in (0,1)$ and $\ell_{0,0}>0$, where $e=F_we_0$
  with $w \in W_m$.  Then the metric graph $\MetGr_m$ and the weighted
  discrete graph $(G_m,\mu_m\gamma_m)$ are compatible.  Moreover, the
  isometric rescaling factor $c=c_m$ and the energy rescaling factor
  $\tau=\tau_m$ are given by
  \begin{equation}
    \label{eq:c.tau.m}
    c_m^2
    = \frac 2{C_0 N_0} \frac 1{(N\Lambda)^m}
    \qquadtext{and}
    \tau_m
    = \frac{\ell_{m,e}\gamma_{m,e}}{c_m^2}
    = \frac{C_0 N_0}2 \cdot \Bigl(\frac{N\Lambda^2}r\Bigr)^m.
  \end{equation}
\end{lemma}
\begin{proof}
  We have
  \begin{equation*}
    \nu_m(v)
    = \frac 12\sum_{e \in E_v(G_m)} \ell_{m,e}
    = \ell_{0,0} \frac{\Lambda^m}2\sum_{w \in W_{m,v}} \sum_{e_0 \in E_{F_w^{-1}v}(G_0))} 
         \frac1{\gamma_{0,e_0}}
    = \ell_{0,0} \frac{\Lambda^m}2 \card{W_{m,v}} C_0
  \end{equation*}
  using $\gamma_{m,e}=r^{-m}\gamma_{0,F_w^{-1}e}$, hence
  \begin{equation*}
    \frac1{c_m^2 }
    = \frac{\nu_m(v)}{\mu_m(v)}
    = \frac{\ell_{0,0} \Lambda^m \card{W_{m,v}} C_0}2
       \cdot \frac{N_0 N^m}{\card{W_{m,v}}}
    = \frac{\ell_{0,0} C_0N_0}2 \cdot (N \Lambda)^m
  \end{equation*}
  using~\eqref{eq:mu.m.v}.  In particular, the vertex weights are
  compatible, i.e.,~\eqref{eq:mu.nu.comp.mg} holds, and $c_m$ is given
  as in~\eqref{eq:c.tau.m} .  Moreover, we have
  \begin{equation*}
    c_m^2\tau_m
    =\ell_{m,e}\gamma_{m,e}
    =\frac{\Lambda^m}{\gamma_{0,e_0}}\cdot\frac{\gamma_{0,e_0}}{r^m}
    = \frac{\Lambda^m}{r^m},
  \end{equation*}
  i.e.,~\eqref{eq:ell.tau} holds, too, and $\tau_m$ is given as
  in~\eqref{eq:c.tau.m}.
\end{proof}
We can play a bit with the choice of parameters:

\myparagraph{Case 1 (geometric case):} If we set $\Lambda=\theta$,
then the length scale shrinks as the IFS with similitude factor
$\theta$.  In this case, $\ell_m$, $c_m$ and $\tau_m$ are given as
in~\eqref{eq:met.gr.len} and~\eqref{eq:c.tau.m} with $\Lambda$
replaced by $\theta$.

\myparagraph{Case 2 (edge weight is inverse of edge length):} Set
$\Lambda=r$ and $\ell_{0,0}=1$, then we have
$\ell_{m,e}=1/\gamma_{m,e}$ and
\begin{equation*}
  \ell_{m,e}
  = \frac{r^m}{\gamma_{0,e_0}},
  \qquad
  c_m^2
  = \frac2 {C_0 N_0} \cdot \frac1{(Nr)^m}
  \qquadtext{and}
  \tau_m
  = \frac {C_0 N_0} 2 \cdot (Nr)^m
\end{equation*}
if $e=F_we_0$ for some word $w \in W_m$.

\myparagraph{Case 3 (no energy rescaling factor):} We can also fix
$\tau_m=1$, then $\Lambda=\sqrt{r/N}$.  Moreover,
\begin{equation*}
  \ell_{m,e}
  = \sqrt{\frac 2{C_0N_0}}  \frac1{\gamma_{0,e_0}} \cdot 
  \Bigl(\frac rN \Bigr)^{m/2}
  \qquadtext{and}
  c_m^2
  = \sqrt{\frac2 {C_0 N_0}} \cdot \frac1{(Nr)^{m/2}}
\end{equation*}
if $e=F_we_0$ for some word $w \in W_m$.

We have the following orders of the length scale, the isometric
rescaling factor and the energy renormalisation factor:

\begin{center}
  \begin{tabular}{|l|l|l|l|}
    \hline
    \myfont{Case}  & $\ell_{m,e}=$ & $c_m=$ & $\tau_m=$\\
    \hline\hline
    \myfont 1      & $\Err(\theta^m)$ & $\Err((N\theta)^{-m/2})$
                       & $\Err((N\theta^2/r)^m)$\\
    \hline
    \myfont 2      & $\Err(r^m)$ & $\Err((Nr)^{-m/2})$ & $\Err((Nr)^m)$\\
    \hline
    \myfont 3      & $\Err((r/N)^{m/2})$ & $\Err((Nr)^{-m/4})$ & $1$\\
    \hline
  \end{tabular}
\end{center}

For the Sierpi\'nski triangle, we have $N=3$, $r=3/5$ and $\theta=1/2$,
hence the parameters have the following order:
\begin{center}
  \begin{tabular}{|l|l|l|l|}
    \hline
    \myfont{Case}  & $\ell_{m,e}=$ & $c_m=$ & $\tau_m=$\\
    \hline\hline
    \myfont 1      & $\Err(1/2^m)$ & $\Err((2/3)^{m/2})$
                       & $\Err((5/4)^m)$\\
    \hline
    \myfont 2      & $\Err((3/5)^m)$ & $\Err((5/9)^{m/2})$ & $\Err((9/5)^m)$\\
    \hline
    \myfont 3      & $\Err(1/5^{m/2})$ & $\Err((5/9)^{m/4})$ & $1$\\
    \hline
  \end{tabular}
\end{center}

\begin{theorem}
  \label{thm:mg.frac.q-u-e}
  Let $(G_m,\mu_m,\gamma_m)$ be the $m$-th generation of a symmetric
  \pcf fractal given by $(K,\energy_K,\mu,V_0)$ (see \Def{frac.sym}).
  Moreover, let $\energy_m$ be the discrete energy functional of
  $(G_m,\mu_m,\gamma_m)$ (see~\eqref{eq:disc.graph.energy}), and let
  $\wt \energy_m=\tau_m \energy_{\MetGr_m}$ be the rescaled metric
  graph energy (see~\eqref{eq:met.graph.energy}) for the metric graph
  $\MetGr_m$ constructed according to $G_m$ with edge lengths
  $(\ell_{m,e})_e$ as in~\eqref{eq:met.gr.len}.  Then $\energy_m$ and
  $\wt \energy_m$ are $\delta_m$-unitarily equivalent with
  $\delta_m=\Err((r/N)^{m/2})$.
\end{theorem}
\begin{proof}
  Note first that $\MetGr_m$ and $(G_m,\mu_m,\gamma_m)$ are compatible
  by \Lem{met.comp}.  We then apply \Thm{mg.q-u-e} and calculate the
  error term $\delta=\delta_m$ defined there.  Note first that
  $\gammaVar=\gamma_\infty/\gamma_0=\max_{e_0 \in E_0}
  \gamma_{e_0}/\min_{e_0 \in E_0} \gamma_{e_0}$.  Moreover,
  \begin{equation*}
    \frac{\mu_{m,\infty}}{\gamma_{m,0}}
    =\frac{\max_{v \in V_m} \mu_m(v)}{\min_{e \in E_m} \gamma_{m,e}}
    =\frac{N_1}{N_0N^m} \cdot \frac {r^m}{C_1}
    =\frac{N_1}{N_0 C_1} \cdot \Bigl( \frac rN \Bigr)^m
    =\Err\Bigl(\Bigl( \frac rN \Bigr)^m\Bigr)
  \end{equation*}
  where $N_1$, $C_1$ and $C_2$ are defined in~\eqref{eq:def.N1.C1}.
\end{proof}

Using the $\Err((r/N)^{m/2})$-quasi-unitary equivalence of the fractal
energy $\energy_K$ and the discrete graph energy $\energy_{G_m}$
proven in~\cite{post-simmer:pre17a} and the transitivity of
quasi-unitary equivalence (see~\Prp{trans.q-u-e}), we obtain:
\begin{corollary}
  \label{cor:mg.frac.q-u-e}
  Let $K$ be a \pcf fractal such that $(K,\energy_K,\mu,V_0)$ is
  symmetric.  Moreover, let $\wt \energy_m=\tau_m \energy_{\MetGr_m}$
  be the rescaled metric graph energy for the metric graph $\MetGr_m$
  constructed as above.  Then $\wt \energy_m$ and the fractal energy
  form $\energy_K$ are $\hat \delta_m$-unitarily equivalent with $\hat
  \delta_m \to 0$ as $m \to \infty$.
\end{corollary}

If we want to quantify the error estimate, we need the weaker notion
of operator quasi-unitary equivalence, see \Def{quasi-uni.op} and
\Prp{q-u-e.qf.op}, and the transitivity for this notion in
\Prp{trans.q-u-e.op}:
\begin{corollary}
  \label{cor:mg.frac.q-u-e'}
  Let $(K,\energy_K,\mu,V_0)$ be a symmetric \pcf fractal.  Moreover,
  let $\wt \Delta_m$ be the metric graph Laplacian associated with the
  metric graph $\MetGr_m$ as above.  Then $\wt \Delta_m$ and the
  fractal Laplacian $\Delta_K$ (associated with $\energy_K$) are $\hat
  \delta_m$-unitarily equivalent, where $\hat \delta_m$ is of order
  $\Err((r/N)^{m/2})$.
\end{corollary}

For the Sierpi\'nski triangle, we have $N=3$, $r=3/5$ and $\theta=1/2$,
hence the error estimate of $\hat \delta_m$ is $\Err(1/5^{m/2})$.

\subsection{Quasi-unitary equivalence of fractals and graph-like
  manifolds}
\label{ssec:q-u-e.frac.mfd}

Lastly, we will apply \Cor{mfd.q-u-e} to the case of a family of
discrete weighted graphs $G=G_m$ with vertex weights $\mu=\mu_m$ and
edge weights $\gamma=\gamma_m$ and a corresponding family of
graph-like manifolds $X_m$.  The scaling of the transversal and vertex
manifold are
\begin{equation}
  \label{eq:mfd.num1}
  Y_{m,e}=\eps_mY_e
  \qquadtext{and} 
  \check X_{m,v} = \eps_m \cXv
\end{equation}
(as before, $rX$ is the Riemannian manifold $(X,r^2g)$ if $X$ is the
Riemannian manifold $(X,g)$; the factor $r$ is hence a change of
\emph{length scale}).

We now determine the parameters of the graph-like manifold having
exponential dependency on $m$, namely
\begin{equation}
  \label{eq:mfd.num}
  \ell_{m,e}=\frac{\ell_{0,0}}{\gamma_{e_0,0}}\Lambda^m,
  \quadtext{and} 
  \eps_m = \eps_0 \Eps^m, 
  \quadtext{where}
  0 < \Eps < \Lambda < 1
\end{equation}
and where $e=F_w(e_0)$ for some word $w \in W_m$.  Here,
$\ell_{0,0}>0$ and $\eps_0>0$ are some constants.  Moreover, we assume
that $Y_e$ is isometric with a fixed manifold $Y_0$ with $\vol Y_0=1$
(for simplicity only).  Finally, we assume that there are constants
$\check \vol_0$, $\check \vol_\infty$ and $\check \lambda_2$ such that
the unscaled manifold $\cXv$ fulfil
\begin{equation}
  \label{eq:vol.frac.mfd}
  0<\check \vol_0 \le \vol \cXv \le \check \vol_\infty<\infty
  \quadtext{and}
  \lambda_2(\cXv) \ge \check \lambda_2 >0
\end{equation}
for all $e \in E$ and $v \in V$.

Our main result is now the following:
\begin{theorem}
  \label{thm:mfd.frac.q-u-e}
  Let $(G_m,\mu_m,\gamma_m)$ be the $m$-th generation of a symmetric
  \pcf fractal given by $(K,\energy_K,\mu,V_0)$ (see \Def{frac.sym}).
  Moreover, let $\energy_m$ be the discrete energy functional of
  $(G_m,\mu_m,\gamma_m)$ (see~\eqref{eq:disc.graph.energy}), and let
  $\wt \energy_m=\tau_m \energy_{\Mfd_m}$ be the rescaled graph-like
  manifold energy (see~\eqref{eq:mfd.energy}) for the
  graph-like manifold $X_m$ constructed according to $G_m$ with edge
  lengths $(\ell_{m,e})_e$, transversal manifolds ($Y_{m,e})_e$ and
  core vertex neighbourhoods $(\vol \check X _{m,v})_v$,
  \eqref{eq:mfd.num1}--\eqref{eq:vol.frac.mfd}.  Finally, we assume
  that
  \begin{equation*}
    \frac rN \Lambda < \Eps < \Lambda.
  \end{equation*}
  Then $\energy_m$ and $\wt \energy_m$ are $\delta_m$-unitarily
  equivalent with
  \begin{equation}
    \label{eq:err.mfd.frac}
    \delta_m
    =\max \Bigl\{
      \Err\Bigl(\Bigl(\frac \Eps \Lambda\Bigr)^{m/2}\Bigr),\;
      \Err\Bigl(\Bigl(\frac \Lambda \Eps \cdot \frac rN\Bigr)^{m/2}\Bigr)
      \Bigr\}.
  \end{equation}
  In particular, if $\Eps$ is the geometric mean of $\Lambda$ and
  $(r/N)\Lambda$, i.e., $\Eps=(r/N)^{1/2}\Lambda$, then the error
  estimate is $\delta_m=\Err((r/N)^{m/4})$, the best possible choice.
\end{theorem}
\begin{proof}
  The result follows from the assumptions and \Cor{mfd.q-u-e}.  Note
  that~\eqref{eq:ell.tau.mfd.eps} follows as in \Lem{met.comp}, since
  the situation is the same as for metric graphs.  Note also that
  $\eps_m/\ell_m \to 0$ as $m \to \infty$, hence the condition $\eps_m
  \le \ell_m/C_0^2$ is eventually fulfilled.
\end{proof}

\begin{remark}
  Note that since $\vol Y_e=1$, the compatibility conditions for the
  graph-like manifold in~\eqref{eq:comp.mfd}--\eqref{eq:ell.tau.y} are
  formally the same as for a metric graph
  in~\eqref{eq:mu.nu.comp.mg}--\eqref{eq:ell.tau}, hence the different
  cases for choices of the parameters in \Subsec{q-u-e.frac.mg} also
  apply here, we just have to take into account, that the isometric
  rescaling factor $c_m$ also depends on $\eps_m$, namely, $c_m$
  contains an extra factor $\eps_m^{-(d-1)/2}=\eps_0^{-(d-1)/2}
  (\Eps^{(d-1)/2})^{-m}$, while the energy renormalisation factor
  $\tau_m$ remains the same.  In particular, we have for
  $\Lambda=\theta$, $r$ and $(r/N)^{1/2}$ the following cases:
  \begin{center}
    \begin{tabular}{|l|l|l|l|l|l|}
      \hline
      \myfont{Case}  & $\ell_{m,e}=$ & $\eps_m=$ & $\Eps \in $ & $c_m=$ &
      $\tau_m=$\\
      \hline\hline
      \myfont 1      & $\Err(\theta^m)$ & $\Err(\Eps^m)$
      & $( r\theta/N,\theta)$
      & $\Err((\Eps^{d-1}N\theta)^{-m/2})$
      & $\Err((N\theta^2/r)^m)$\\
      \hline
      \myfont 2      & $\Err(r^m)$  & $\Err(\Eps^m)$
      & $(r^2/N,r)$
      & $\Err((\Eps^{d-1}Nr)^{-m/2})$ & $\Err((Nr)^m)$\\
      \hline
      \myfont 3      & $\Err((r/N)^{m/2})$  & $\Err(\Eps^m)$
      & $((r/N)^{3/2},(r/N)^{1/2})$
      & $\Err((\Eps^{d-1}Nr)^{-m/4})$ & $1$\\
      \hline
    \end{tabular}
  \end{center}
\end{remark}

\begin{example}
  \label{ex:mfd.sierpinski}
  For a Sierpi\'nski triangle we have $N=3$, $r=3/5$ and we can choose
  $\Lambda=1/2$, the length scale factor.  If we choose
  $\Eps=1/(2\sqrt 5)$, then the error estimate has the optimal rate
  $\delta_m=\Err((1/5)^{m/4})$.  In particular, the approximating
  manifold has longitudinal edge lengths scaling as $\Lambda^m=1/2^m$
  in generation $m$, while the transversal manifold has radius of
  order $\Eps^m=1/(2\sqrt 5)^m$, which shrinks \emph{faster} than the
  longitudinal scale $\Lambda^m$.  This is Case~1 in the tabular
  below.  Other choices for $\Lambda$ are $1/2$, $3/5$ and $1/5^{1/2}$
  (the $\Eps$ with optimal error rate then is $1/(2\sqrt 5)$,
  $3/(5\sqrt 5)$ and $1/5$):
  \begin{center}
    \begin{tabular}{|l|l|l|l|l|l|}
      \hline
      \myfont{Case}  & $\ell_{m,e}=$ & $\eps_m=$ & $\Eps \in $ & $c_m=$ &
      $\tau_m=$\\
      \hline\hline
      \myfont 1      & $\Err(1/2^m)$ & $\Err(\Eps^m)$
      & $(1/10,1/2)$
      & $\Err((\Eps^{-(d-1)}2/3)^{m/2})$
      & $\Err((5/4)^m)$\\
      \hline
      \myfont 2      & $\Err((3/5)^m)$  & $\Err(\Eps^m)$
      & $(3/25,3/5)$
      & $\Err((\Eps^{-(d-1)}5/9)^{m/2})$ & $\Err((9/5)^m)$\\
      \hline
      \myfont 3      & $\Err((1/5^{m/2})$  & $\Err(\Eps^m)$
      & $(1/5^{3/2},1/5^{1/2})$
      & $\Err((\Eps^{-(d-1)}5/9)^{m/4})$ & $1$\\
      \hline
    \end{tabular}
  \end{center}
  We can also choose $\Eps$ as close to $\Lambda$ (i.e., as large) as
  we want, the price is a worse error estimate.
\end{example}

\begin{remark}
  \label{rem:no-mfd-ifs}
  Note that we cannot assume directly $\Eps=\Lambda$.  This case is
  interesting since it would allow to apply directly the IFS to a
  suitable starting compact neighbourhood $X_0$ of the (metric) graph
  associated with $G_0$. In this case $\Eps=\Lambda=\theta$, where
  $\theta$ is the the similitude factor, but here, the graph-like
  manifold is not shrinking fast enough in transversal direction (the
  transversal scale is $\eps_m=\eps_0 \Eps^m$, while the longitudinal
  scale $\ell_{m,e} $ is of order $\Lambda^m$).

  Nevertheless we conjecture that if $\Eps=\Lambda$ and if the
  starting transversal parameter $\eps_0$ is sufficiently small, one
  can still conclude convergence results such as convergence of
  eigenvalues of graph-like manifolds (i.e., a sequence of graph-like
  manifolds $X_m$ generated by the IFS) with corresponding (Neumann)
  eigenvalues converging to the eigenvalues of the fractal.  Note that
  $\Eps<\Lambda$ is only used in the parameter $\eps/\ell_0$ resp.\
  $\alpha_\infty$ in \Cor{mfd.q-u-e} resp.\ \Cor{mfd.q-u-e}, and
  $\alpha_\infty$ is only needed for $\norm{J'-J^*} \le
  2\alpha_\infty$ in \Prpenum{q-u.b}{q-u.b.c}.  We will treat such
  questions in a subsequent publication.
\end{remark}

Using again the $\Err((r/N)^{m/2})$-quasi-unitary equivalence of the
fractal energy $\energy_K$ and the discrete graph energy
$\energy_{G_m}$ proven in~\cite{post-simmer:pre17a} and the
transitivity of quasi-unitary equivalence (see~\Prp{trans.q-u-e}), we
obtain:
\begin{corollary}
  \label{cor:mfd.frac.q-u-e}
  Let $K$ be a \pcf fractal such that $(K,\energy_K,\mu,V_0)$ is
  symmetric.  Moreover, let $\wt \energy_m=\tau_m \energy_{\Mfd_m}$ be
  the rescaled energy form of the graph-like manifold $\Mfd_m$
  constructed according to $G_m$ as in \Thm{mfd.frac.q-u-e}.  Then
  $\wt \energy_m$ and the fractal energy form $\energy_K$ are $\hat
  \delta_m$-unitarily equivalent with $\hat \delta_m \to 0$ as $m \to
  \infty$.
\end{corollary}

As for the metric graph approximation, we obtain a quantified error
estimate using the weaker notion of operator quasi-unitary
equivalence, see \Def{quasi-uni.op} and \Prp{q-u-e.qf.op}.  The
transitivity for this notion in \Prp{trans.q-u-e.op} gives a precise
error estimate:
\begin{corollary}
  \label{cor:mfd.frac.q-u-e'}
  Let $K$ be a \pcf fractal such that $(K,\energy_K,\mu,V_0)$ is
  symmetric.  Moreover, let $\wt \Delta_m=\tau_m \Delta_{\Mfd_m}$ be
  the Laplacian (associated with $\wt
  \energy_m=\tau_m\energy_{\Mfd_m}$) on the the graph-like manifold
  $\Mfd_m$ constructed as above.  Then $\wt \Delta_m$ and the fractal
  Laplacian $\Delta_K$ (associated with $\energy_K$) are $\hat
  \delta_m$-unitarily equivalent, where $\hat \delta_m$ is of order as
  in~\eqref{eq:err.mfd.frac}.
\end{corollary}

\appendix
%
\section{An abstract norm resolvent convergence result}
\label{app:norm.convergence}
%

In this appendix, we briefly present a general framework which assures
a \emph{generalised} norm resolvent convergence for operators
$\Delta_m$ converging to $\Delta_\infty$ as $\eps \to 0$,
see~\cite{post:12} for details.  Each operator $\Delta_m$ acts in a
Hilbert space $\HS_m$ for $m \in \N$; and the Hilbert spaces are
allowed to depend on $m$.

In one of our applications, the Hilbert spaces $\HS_m$ are of the form
$\Lsqr{X_m}=\lsqr{V_m,\mu_m}$ and a ``limit'' metric measure space
$(X,\mu)$ with Hilbert space $\wt \HS=\Lsqr{X,\mu}$.

In order to define the convergence, we define a sort of ``distance''
$\delta_m$ between $\Delta := \Delta_m$ and $\wt \Delta :=
\Delta_\infty$, in the sense that if $\delta_m \to 0$ then $\Delta_m$
converges to $\Delta_\infty$ in the above-mentioned generalised norm
resolvent convergence.  We start now with the general concept:

Let $\HS$ and $\wt \HS$ be two separable Hilbert spaces.  We say that
$(\energy,\HS^1)$ is an \emph{energy form in $\HS$} if $\energy$ is a
closed, non-negative quadratic form in $\HS$, i.e., if
$\energy(f)\coloneqq \energy(f,f)$ for some sesquilinear form $\map
{\energy}{\HS^1 \times \HS^1} \C$, denoted by the same symbol, if
$\energy(f)\ge 0$ and if $\HS^1=:\dom \energy$, endowed with the norm
defined by
\begin{equation}
  \label{eq:qf.norm}
  \normsqr[1] f
  \coloneqq \normsqr[\HS^1] f
  \coloneqq \normsqr[\HS] f + \energy(f),
\end{equation}
is itself a Hilbert space and dense (as a set) in $\HS$.  We call the
corresponding non-negative, self-adjoint operator by $\Delta$ (see
e.g.~\cite[Sec.~VI.2]{kato:66}) the \emph{energy operator} associated
with $(\energy,\HS^1)$.  Similarly, let $(\wt \energy,\wt \HS^1)$ be an
energy form in $\wt \HS$ with energy operator $\wt \Delta$.

Associated with an energy operator $\Delta$, we can define a natural
\emph{scale of Hilbert spaces} $\HS^k$ defined via the \emph{abstract
  Sobolev norms}
\begin{equation}
  \label{eq:def.abstr.sob.norm}
  \norm[\HS^k] f
  \coloneqq \norm[k] f 
  \coloneqq \norm{(\Delta+1)^{k/2}f}.
\end{equation}
Then $\HS^k=\dom \Delta^{k/2}$ if $k \ge 0$ and $\HS^k$ is the
completion of $\HS=\HS^0$ with respect to the norm $\norm[k] \cdot$
for $k<0$.  Obviously, the scale of Hilbert spaces for $k=1$ and its
associated norm agrees with $\HS^1$ and $\norm[1]\cdot$ defined above
(see~\cite[Sec.~3.2]{post:12} for details).  Similarly, we denote by
$\wt \HS^k$ the scale of Hilbert spaces associated with $\wt \Delta$.

We now need pairs of so-called \emph{identification operators} acting
on the Hilbert spaces and later also pairs of identification operators
acting on the form domains.

\begin{definition}
  \label{def:quasi-uni}
  \begin{subequations}
    \label{eq:quasi-uni}
    Let $\delta \ge 0$, and let $\map J \HS {\wt \HS}$ and $\map {J'}
    {\wt \HS}\HS$ be bounded linear operators.  Moreover, let $\delta
    \ge 0$, and let $\map {J^1} {\HS^1} {\wt \HS^1}$ and $\map
    {J^{\prime1}} {\wt \HS^1}{\HS^1}$ be bounded linear operators
    on the energy form domains.
    \begin{enumerate}
    \item We say that $J$ is \emph{$\delta$-quasi-unitary} with
      \emph{$\delta$-quasi-adjoint} $J'$ if and only if
      \begin{gather}
        \label{eq:quasi-uni.a}
        \norm{Jf}\le (1+\delta) \norm f, \quad
        \bigabs{\iprod {J f} u - \iprod f {J' u}}
        \le \delta \norm f \norm u
        \qquad (f \in \HS, u \in \wt \HS),\\
        \label{eq:quasi-uni.b}
        \norm{f - J'Jf}
        \le \delta \norm[1] f, \quad
        \norm{u - J'Ju}
        \le \delta \norm[1] u \qquad (f \in \HS^1, u \in \wt \HS^1).
      \end{gather}
      
    \item We say that $J^1$ and $J^{\prime1}$ are
      \emph{$\delta$-compatible} with the identification operators $J$
      and $J'$ if
      \begin{equation}
        \label{eq:quasi-uni.c}
        \norm{J^1f - Jf}\le \delta \norm[1]f, \quad
        \norm{J^{\prime1}u - J'u} \le \delta \norm[1] u
        \qquad (f \in \HS^1, u \in \wt \HS^1).
      \end{equation}
      
    \item We say that the energy forms $\energy$ and $\wt \energy$ are
      \emph{$\delta$-close} if and only if
      \begin{equation}
        \label{eq:quasi-uni.d}
        \bigabs{\wt \energy(J^1f, u) - \energy(f, J^{\prime1}u)} 
        \le \delta \norm[1] f \norm[1] u
        \qquad (f \in \HS^1, u \in \wt \HS^1).
      \end{equation}
      
    \item We say that $\energy$ and $\wt \energy$ are
      \emph{$\delta$-quasi-unitarily equivalent},
      if~\eqref{eq:quasi-uni.a}--\eqref{eq:quasi-uni.d} are fulfilled,
      i.e., if the following operator norm estimates hold:
      \begin{gather}
        \label{eq:quasi-uni.a'}
        \tag{\ref{eq:quasi-uni.a}'}
        \norm J \le 1+\delta, \qquad \norm{J^* - J'} \le \delta\\
        \label{eq:quasi-uni.b'}
        \tag{\ref{eq:quasi-uni.b}'}
        \norm{(\id_\HS - J'J)R^{1/2}} \le \delta, \qquad
        \norm{(\id_{\wt \HS} - J J')\wt R^{1/2}} \le \delta,\\
        \label{eq:quasi-uni.c'}
        \tag{\ref{eq:quasi-uni.c}'}
        \norm{(J^1-J)R^{1/2}} \le \delta, \qquad
        \norm{(J^{\prime1}-J')\wt R^{1/2}} \le \delta,\\
        \label{eq:quasi-uni.d'}
        \tag{\ref{eq:quasi-uni.d}'}
        \norm{\wt R^{1/2}(\wt \Delta J^1- J^{-1} \Delta)R^{1/2}} \le \delta,
      \end{gather}
      where $R\coloneqq(\Delta+1)^{-1}$ resp.\ $\wt R\coloneqq(\wt
      \Delta+1)^{-1}$ denotes the resolvent of $\Delta$ resp.\ $\wt \Delta$
      in $-1$.  Moreover, $\map{J^{-1} \coloneqq
        (J^{\prime 1})^*}{\HS^{-1}}{\wt \HS^{-1}}$, where $(\cdot)^*$ denotes
      here the dual map with respect to the dual pairing $\HS^1 \times
      \HS^{-1}$ induced by the inner product on $\HS$ and similarly on $\wt
      \HS$.  Moreover, $\Delta$ is interpreted as $\map \Delta {\HS^1}
      {\HS^{-1}}$, and similarly for $\wt \Delta$.
    \end{enumerate}
  \end{subequations}
\end{definition}

\begin{remark}
  \label{rem:quasi-uni}
  Let us explain the notation in two extreme cases assuring that
  $\delta$ is in some sense a ``distance'' between the two forms:
  \begin{enumerate}
  \item \emph{``$\delta$-quasi-unitary equivalence'' is a quantitative
    generalisation of ``unitary equivalence'':}

  Note that if $\delta=0$, $J$ is $0$-quasi-unitary if and only if $J$
  is unitary with $J^*=J'$.  Moreover, $\energy$ and $\wt \energy$ are
  $0$-quasi-unitarily equivalent if and only if $\Delta$ and $\wt
  \Delta$ are unitarily equivalent (in the sense that $JR=\wt RJ$).

\item \emph{``$\delta_m$-quasi-unitary equivalence'' is a
    generalisation of ``norm resolvent convergence'':} If $\HS=\wt
  \HS$, $J=J'=\id_\HS$, then the first two
  conditions~\eqref{eq:quasi-uni.a}--\eqref{eq:quasi-uni.b} are
  trivially fulfilled with $\delta=0$.  Moreover, if $\wt
  \Delta=\Delta_m$ and $\delta=\delta_m \to 0$ as $m \to \infty$, then
  $\energy$ and $\energy_m$ are $\delta_m$-quasi-unitarily equivalent
  if and only if $\norm{R_m^{-1/2}(R - R_m)R^{-1/2}} \to 0$ as $m \to
  \infty$, hence this implies that $\norm{R_m-R} \to 0$, i.e.,
  $\Delta_m$ \emph{converges} to $\Delta$ \emph{in norm resolvent
    sense}.
  \end{enumerate}
\end{remark}

Now from the concept of quasi-unitary equivalence of forms, many more
results follow (see e.g.~\cite{post:12}).

We now state the \emph{transitivity} of $\delta$-quasi-unitary
of \emph{energy forms}.  Assume that $\HS$, $\wt \HS$ and $\widehat
\HS$ are three Hilbert spaces with non-negative operators $\energy$,
$\wt {\energy}$ and $\widehat {\energy}$, respectively.  Moreover, assume
that
\begin{gather*}
  \map J \HS {\wt \HS}, \quad
  \map {\wt J} {\wt \HS} {\widehat \HS}, \quad
  \map {\wt J'} {\widehat \HS} {\wt \HS} \quadtext{and}
  \map {J'} {\wt \HS} \HS,\\
  \map {J^1} {\HS^1} {\wt \HS^1}, \;
  \map {\wt J^1} {\wt \HS^1} {\widehat \HS^1}, \:
  \map {\wt J^{\prime 1}} {\widehat \HS^1} {\wt \HS^1} \;\text{and}\;
  \map {J^{\prime 1}} {\wt \HS^1} {\HS^1}
\end{gather*}
are bounded operators. We define
\begin{align*}
  \map {\hat J &:= \wt J J} \HS {\widehat \HS},&
  \map {\hat J' &:= J' \wt J'}{\widehat \HS}  \HS,\\
  \map {\hat J^1 &:= \wt J^1 J^1} {\HS^1} {\widehat \HS^1},&
  \map {\hat J^{\prime 1} &:= J^{\prime 1} \wt J^{\prime 1}}
             {\widehat \HS^1}  {\HS^1}.
\end{align*}
In addition to \Def{quasi-uni} we assume that the identification
operators $J^1$ and $\wt J^1$ in \Def{quasi-uni} are $(1+\delta)$-
resp.\ $(1+\wt \delta)$-bounded, i.e.\
\begin{equation*}
  \norm[1 \to 1]{J^1} \le 1+\delta
      \qquadtext{resp.}
  \norm[1 \to 1]{\wt J^1} \le 1 + \wt \delta.
\end{equation*}
\begin{proposition}[{\cite[Prp.~4.4.16]{post:12}}]
  \label{prp:trans.q-u-e}
  Assume that $0 \le \delta$ and $\wt \delta \le 1$.  Assume in
  addition that $\energy$ and $\wt {\energy}$ are
  $\delta$-quasi-unitarily equivalent with identification operators
  $J$, $J^1$, $J'$ and $J^{\prime 1}$, and that $\wt {\energy}$ and
  $\widehat {\energy}$ are $\wt \delta$-quasi-unitarily equivalent
  with identification operators $\wt J$, $\wt J^1$, $\wt J'$ and $\wt
  J^{\prime 1}$.  Then ${\energy}$ and $\widehat {\energy}$ are $\hat
  \delta$-quasi-unitarily equivalent with identification operators
  $\hat J$, $\hat J^1$, $\hat J'$ and $\hat J^{\prime 1}$,
  where\footnote{The result in~\cite[Prp.~4.4.16]{post:12} is stated
    with a linear error term $\hat
    \delta=\Err(\delta)+\Err(\wt\delta)$, relying on a wrong estimate
    in~\cite[Thm.~4.2.9]{post:12}.  We will correct this estimate in a
    forthcoming publication.} $\hat \delta=\hat
  \delta(\delta,\wt\delta) \to 0$ as $\delta \to 0$ and $\wt \delta
  \to 0$.
\end{proposition}
If we want to quantify the error we need a slightly weaker notion of
unitary equivalence for \emph{operators}:

\begin{definition}
  \label{def:quasi-uni.op}
  \begin{subequations}
    \label{eq:quasi-uni.op}
    Let $\delta \ge 0$, and let $\map J \HS {\wt \HS}$ and $\map {J'}
    {\wt \HS}\HS$ be bounded linear operators.
    \begin{enumerate}
    \item We say that $J$ is \emph{$\delta$-quasi-unitary} with
      \emph{$\delta$-quasi-adjoint} $J'$ (for the operators $\Delta$
      and $\wt \Delta$) if and only if
      \begin{gather}
        \label{eq:quasi-uni.op.a}
        \norm{Jf}\le (1+\delta) \norm f, \quad
        \bigabs{\iprod {J f} u - \iprod f {J' u}}
        \le \delta \norm f \norm u
        \qquad (f \in \HS, u \in \wt \HS),\\
        \label{eq:quasi-uni.op.b}
        \norm{f - J'Jf}
        \le \delta \norm[2] f, \quad
        \norm{u - J'Ju}
        \le \delta \norm[2] u \qquad (f \in \HS^2, u \in \wt \HS^2).
      \end{gather}
      
    \item We say that the operators $\Delta$ and $\wt \Delta$ are
      \emph{$\delta$-close} if and only if
      \begin{equation}
        \label{eq:quasi-uni.op.d}
        \bigabs{\iprod[\wt \HS]{Jf}{\wt \Delta u} - 
                \iprod[\HS] {J\Delta f}u}
        \le \delta \norm[2] f \norm[2] u
        \qquad (f \in \HS^2, u \in \wt \HS^2).
      \end{equation}
      
    \item We say that $\Delta$ and $\wt \Delta$ are
      \emph{$\delta$-quasi-unitarily equivalent},
      if~\eqref{eq:quasi-uni.op.a}--\eqref{eq:quasi-uni.op.d} are fulfilled,
      i.e.,
      we have the following operator norm estimates
      \begin{gather}
        \label{eq:quasi-uni.op.a'}
        \tag{\ref{eq:quasi-uni.op.a}'}
        \norm J \le 1+\delta, \qquad \norm{J^* - J'} \le \delta\\
        \label{eq:quasi-uni.op.b'}
        \tag{\ref{eq:quasi-uni.op.b}'}
        \norm{(\id_\HS - J'J)R} \le \delta, \qquad
        \norm{(\id_{\wt \HS} - J J')\wt R} \le \delta,\\
        \label{eq:quasi-uni.op.d'}
        \tag{\ref{eq:quasi-uni.op.d}'}
        \norm{\wt R J - J R} \le \delta.
      \end{gather}
    \end{enumerate}
  \end{subequations}
\end{definition}
We have the following relation:
\begin{proposition}[{\cite[Prp.~4.4.15]{post:12}}]
  \label{prp:q-u-e.qf.op}
  If the forms $\energy$ and $\wt \energy$ are $\delta$-quasi
  unitarily equivalent then the operators $\Delta$ and $\wt \Delta$
  are $4\delta$-quasi-unitarily equivalent.
\end{proposition}
The transitivity for operator quasi-unitary equivalence gives a more
explicit error estimate:
\begin{proposition}[{\cite[Prp.~4.2.5]{post:12}}]
  \label{prp:trans.q-u-e.op}
  Assume that $0 \le \delta$ and $\wt \delta \le 1$.  Assume in
  addition that $\Delta$ and $\wt \Delta$ are $\delta$-quasi-unitarily
  equivalent with identification operators $J$ and $J'$, and that $\wt
  \Delta$ and $\widehat \Delta$ are $\wt \delta$-quasi-unitarily
  equivalent with identification operators $\wt J$ and $\wt J'$.  Then
  $\Delta$ and $\widehat \Delta$ are $\hat \delta$-quasi-unitarily
  equivalent with identification operators $\hat J=\wt J J$ and $\hat
  J'=J'\wt J$, where $\hat \delta=22\delta+43\wt\delta$.
\end{proposition}

%
\section{Some estimates on graph-like manifolds}
\label{app:2nd.ev}
%

We need some estimates on our graph-like manifold with respect to a
norm weighted by a harmonic function and also a lower bound on the
second eigenvalue of a graph-like manifold.  We use the notation of
\Sec{mfds}.

\begin{lemma}
  \label{lem:ew.2} We have
  \begin{subequations}
    \begin{gather}
      \label{eq:ew.2a} 
      \normsqr[\Lsqr{Y_e}]{u(t,\cdot)} 
      \le \frac 9 {2 \ell_e} \normsqr[\Lsqr{X_e,\psi_v \dd \nu}] u 
         + 4\ell_e \normsqr[\Lsqr{X_e \cup \Xve}] {du},\\
    \label{eq:ew.2c} 
    \normsqr[\Lsqr \Xve] u 
    \le \frac {9\kappa} 2 \normsqr[\Lsqr{X_e,\psi_v \dd \nu}] u 
    + 4 \kappa \ell_e^2 \normsqr[\Lsqr{X_e \cup \Xve}] {du}.
    \end{gather}
  \end{subequations}
\end{lemma}
\begin{proof} 
  The proof of the first assertion is again an application of the
  fundamental theorem of calculus: Assume that the vertex $v$
  corresponds to the endpoint $\ell_e$ of each adjacent $\Xe$, $e \in
  E_v$.  Note that $\Xe \cong [0,\ell_e]\times Y_e$ and $\Xve \cong
  [\ell_e,(1+\kappa)\ell_e] \times Y_e$; so we can use a common
  coordinate $t \in [0, (1+\kappa)\ell_e]$ for the first variable on
  both $\Xe$ and $\Xve$.  Then $\psi_{v,e}(x)=t/\ell_e$ for $x=(t,y)$
  and $t \in M_e$ with derivative $1/\ell_e$.

  Let $\chi_e(t)=(t/\ell_e)^{3/2}$ if $t \in [0,\ell_e]$ and
  $\chi_e(t)=1$ if $t \in [\ell_e,(1+\kappa)\ell_e]$.  Then $\chi_e$
  is Lipschitz continuous, $\chi_e u \in \Sob{X_e \cup \Xve}$ and
  $\chi_e(0)=0$.  Moreover,
  \begin{equation*}
    u(t,y) =u(t,y)\chi_e(t) 
    =\int_0^t (\chi_e u)'(s,y)\dd s
    =\frac 3{2\ell_e} \int_0^t
         \Bigl(\frac{s}{\ell_e}\Bigr)^{1/2} u(s,y) \dd s 
         + \int_0^t \chi_e(s) u'(s,y) \dd s
  \end{equation*}
  for $t \in [\ell_e,(1+\kappa)\ell_e]$, where $(\cdot)'$ denotes the
  derivative with respect to the first variable.  In particular,
  Cauchy-Schwarz and integrating with respect to $y \in Y_e$ gives
  \begin{align*} 
    \normsqr[\Lsqr{Y_e}]{u(t,\cdot)} 
    &= \int_{Y_e} \abssqr{u(\ell_e,y)} \dd y\\
    &\leCS \frac 9{2\ell_e} \int_{\Xe} \psi_v(s,y)
    \abssqr{u(s,y)} \dd y \dd s + 2 (1+\kappa)\ell_e \int_{\Xe\cup\Xve}
    \abssqr{u'(s,y)} \dd y \dd s
  \end{align*}
  using the fact that $\psi_v(s,y)=s/\ell_e$ on $(s,y) \in \Xe$ and
  that $\chi_e(s) \in [0,1]$.  Note also that $1+\kappa \le 2$.

  The second estimate follows from the first one by integrating over
  $t \in [\ell_e, (1+\kappa)\ell_e]$.
\end{proof}

\begin{lemma}
  \label{lem:ew.1}
  We have
  \begin{gather}
    \label{eq:ew.1a}
    \sum_{e \in E_v} \gamma_e \bigabssqr{\avint_{\cXv} u - \avint_{\bd_e
        \cXv} u}
    \le \max_{e \in E_v} 
    \frac{\gamma_e }{\vol Y_e}
    \Bigl(\kappa\ell_e + \frac 2{\kappa \ell_e \lambda_2(\cXv)}\Bigr)
    \normsqr[\Lsqr \cXv] {du}. 
  \end{gather}
\end{lemma}
\begin{proof}
  We have the following standard estimates for the average and the
  contribution of the boundary component $\bd_e \cXv$ in terms of the
  collar neighbourhood $\Xve$ from \Defenum{gl-mfd}{gl-mfd.c}, namely
  the estimates
  \begin{subequations}
    \label{eq:collar.est}
    \begin{align}
    \label{eq:collar.est.a}
      \bignormsqr{u - \avint_{\cXv} u}
      &\le \frac 1 {\lambda_2(\cXv)} \normsqr[\Lsqr \cXv]{du}
      \quad\text{and}\\
    \label{eq:collar.est.b}
      \int_{\bd_e \cXv} \abssqr u
      &\le \kappa \ell_e \normsqr[\Lsqr \Xve]{du}
      + \frac 2 {\kappa \ell_e} \normsqr[\Lsqr \Xve] u 
    \end{align}
  \end{subequations}
  see e.g.~\cite[Prp.~5.1.1 and Cor.~A.2.12]{post:12} (we omit the natural
  measures of the spaces).  We then have
  \begin{align*}
    \sum_{e \in E_v} \gamma_e \bigabssqr{\avint_{\cXv} u - \avint_{\bd_e \cXv} u}
    &= \sum_{e \in E_v} \gamma_e 
          \bigabssqr{\avint_{\bd_e \cXv}(u - \avint_{\cXv} u)}\\
    &\leCS \sum_{e \in E_v} \frac {\gamma_e}{\vol Y_e}
    \int_{\bd_e \cXv} \abssqr{u - \avint_{\cXv} u}\\
    &\le\sum_{e \in E_v} \frac{\gamma_e}{\vol Y_e}
     \Bigl(
       \kappa \ell_e \normsqr[\Lsqr \Xve]{du}
       + \frac 2 {\kappa \ell_e} \bignormsqr[\Lsqr \Xve] {u - \avint_{\cXv} u}
     \Bigr)\\
    &\le\max_{e \in E_v} 
    \frac{\gamma_e}{\vol Y_e} 
     \Bigl(\kappa\ell_e + \frac 2{\kappa \ell_e \lambda_2(\cXv)}\Bigr) 
     \normsqr[\Lsqr \cXv] {du}
  \end{align*}
  using~\eqref{eq:collar.est} for the last two estimates.
\end{proof}

Let us now provide a lower bound on the second eigenvalue of a
graph-like manifold $X_\vxeps$ \emph{independently} of the shrinking
parameter $\eps$.  Here, $X_\vxeps$ is a graph-like manifold
according to a star graph $M_v$ with central vertex $v$ and $e \in
E_v$ adjacent edges isometric with $M_e=[0,\ell_e]$ with vertices
$v_e$ of degree $1$: The core vertex neighbourhood is scaled as $\eps
\cXv$, the edge neighbourhood as $X_\edeps=[0,\ell_e] \times \eps
Y_e$ (the notation is explained in the paragraph before
\Cor{mfd.q-u-e}).
\begin{proposition}
  \label{prp:2nd.ev.mfd}
  Assume that 
  \begin{equation*}
    0 < \ell_0 \le \ell_e \le \ell_\infty < \infty
    \qquad\text{for all $e \in E$.}
  \end{equation*}
  Then there exists a constant $C_v$ depending
  only on $\ellVar=\ell_\infty/\ell_0$ and upper estimates on the
  unscaled quantities $\vol \cXv/(\sum_{e \in E_v} \vol Y_e)$,
  $1/\lambda_2(Y_e)$ and $1/\lambda_2(\cXv)$ such that
  \begin{equation}
    \label{eq:2nd.ev.mfd'}
    \lambda_2(X_\vxeps) \ge \frac 1{\ell_\infty^2}
    \quadtext{for all}
    0<\eps \le \eps_0:=\frac{\ell_0}{C_v^2}.
  \end{equation}
\end{proposition}
\begin{proof}
  We use a simple scaling argument to replace the parameters $\ell_e$
  and $\eps$ by $\ell_e/\ell_0 \in [1,\ellVar]$ and
  $\kappa=\eps/\ell_0$.  Denote by $\ell_0^{-1}X_\vxeps$ the scaled
  graph-like manifold (with metric $\ell_0^{-2}g_\vxeps$) and
  similarly denote by $\ell_0^{-1}M_v$ the metric graph with edge
  length $\ell_e/\ell_0$.  Now, the edge lengths are in $[1,\ellVar]$,
  and we can apply the convergence result for graph-like manifolds,
  proven e.g.\ in~\cite{exner-post:05}, giving us
  \begin{equation*}
    \bigabs{\lambda_k(\ell_0^{-1}X_\vxeps)-\lambda_k(\ell_0^{-1}M_v)}
    \le C_v \kappa^{1/2}
  \end{equation*}
  (the error estimate $\kappa^{1/2}$ is proven e.g.\
  in~\cite[Thm.~6.4.1 and Thm.~4.6.4]{post:12}; note that $\kappa$ is
  the transversal thickness, called $\eps$ in the cited works).  Since
  \begin{equation*}
    \lambda_k(\ell_0^{-1}X_\vxeps)
    = \ell_0^2\lambda_k(X_\vxeps)
    \quadtext{and}
    \lambda_k(\ell_0^{-1}M_v)
    = \ell_0^2 \lambda_k(M_v)
  \end{equation*}
  we obtain
  \begin{equation*}
    \lambda_2(X_\vxeps) 
    \ge \lambda_2(M_v) - \frac{C_v}{\ell_0^2}
        \Bigl(\frac \eps {\ell_0}\Bigr)^{1/2}
    \ge \frac 2{\ell_0^2} - \frac{C_v}{\ell_0^2}
        \Bigl(\frac \eps {\ell_0}\Bigr)^{1/2}
  \end{equation*}
  using \Lem{2nd.ev.mg} for the last estimate.  Choosing $0 < \eps \le
  \eps_0 :=C_v^{-2} \ell_0$ and $\ell_0 \le \ell_\infty$ we obtain the
  desired result.
\end{proof}

%
%

 \def\cprime{$'$}
\providecommand{\bysame}{\leavevmode\hbox to3em{\hrulefill}\thinspace}
\providecommand{\MR}{\relax\ifhmode\unskip\space\fi MR }
\providecommand{\MRhref}[2]{%
  \href{http://www.ams.org/mathscinet-getitem?mr=#1}{#2}
}
\providecommand{\href}[2]{#2}

\end{document}